\documentclass[11pt]{amsart}
\usepackage{lmodern}
\usepackage[T1]{fontenc}
\usepackage[english]{babel}
\usepackage{amsmath,amscd,amssymb,amsthm,amsxtra}
\usepackage{microtype}
\usepackage{eucal}
\usepackage{mathrsfs}
\usepackage{mathtools}
\usepackage{accents}
\usepackage{enumerate}
\usepackage{color,graphicx,esint}
\usepackage{epsfig,epstopdf}
\usepackage{tikz-cd}
\usetikzlibrary{decorations.pathreplacing,patterns}
\usepackage[margin=1.25in, top=1in, bottom=0.9in]{geometry}
\usepackage[notref,notcite]{}
\usepackage{hyperref}
\hypersetup{colorlinks={true}, citecolor={blue}}

\theoremstyle{plain}
\newtheorem{theorem}{Theorem}[section]
\newtheorem{corollary}[theorem]{Corollary}
\newtheorem{proposition}[theorem]{Proposition}
\newtheorem{lemma}[theorem]{Lemma}

\theoremstyle{definition}
\newtheorem{definition}[theorem]{Definition}
\newtheorem{remark}[theorem]{Remark}

\theoremstyle{remark}
\numberwithin{theorem}{section}
\numberwithin{equation}{section}
\numberwithin{figure}{section}

\def\R{\mathbb{R}}

\def\N{\mathbb{N}}

\def\1{{\bf 1}}

\def\d{\mathrm{d}}
\def\be{{\boldsymbol{e}}}
\def\bv{{\boldsymbol{v}}}
\def\ba{{\boldsymbol{a}}}

\def\a{\alpha}

\def\kap{\kappa}
\def\l{\lambda}

\def\nab{\nabla}

\def\vep{\varepsilon}

\def\pa{\partial}

\def\Xint#1{\mathchoice
      {\XXint\displaystyle\textstyle{#1}}%
      {\XXint\textstyle\scriptstyle{#1}}%
      {\XXint\scriptstyle\scriptscriptstyle{#1}}%
      {\XXint\scriptscriptstyle\scriptscriptstyle{#1}}%
\!\int}
   \def\XXint#1#2#3{{\setbox0=\hbox{$#1{#2#3}{\int}$}
        \vcenter{\hbox{$#2#3$}}\kern-.5\wd0}}
   
   \def\dashint{\Xint-}

\DeclareMathOperator{\diam}{diam}
\DeclareMathOperator{\dist}{dist}
\DeclareMathOperator{\DIV}{div}

\DeclareMathOperator*{\osc}{osc}

\DeclareMathOperator{\Span}{span}

\DeclareMathOperator{\Ext}{Ext}

\newcommand{\mres}{\mathop{\hbox{\vrule height 7pt width .5pt depth 0pt
\vrule height .5pt width 6pt depth 0pt}}\nolimits}

\begin{document}

\title[Singular Points in the Thin Obstacle Problem]{
On the Singular Set in the Thin Obstacle Problem:\\
Higher Order Blow-ups and the Very Thin Obstacle Problem
}



\author[X.\ Fern\'{a}ndez-Real]{Xavier\ Fern\'{a}ndez-Real}
\address{ETH Z\"{u}rich, Department of Mathematics, R\"{a}mistrasse 101, Z\"{u}rich 8092, Switzerland}
\email{xavierfe@math.ethz.ch}

\author[Y.\ Jhaveri]{Yash Jhaveri}
\address{Institute for Advanced Study, 1 Einstein Drive, Princeton, New Jersey 08540, USA}
\email{yjhaveri@math.ias.edu}
\thanks{XF and YJ were supported in part by the European Research Council (ERC) under the Grant Agreement No 721675.
YJ was also supported in part by NSF grant DMS-1638352.}

\begin{abstract}
In this work, we consider the singular set in the thin obstacle problem with weight $|x_{n+1}|^a$ for $a\in (-1, 1)$, which arises as the local extension of the obstacle problem for the fractional Laplacian (a non-local problem).
We develop a refined expansion of the solution around its singular points by building on the ideas introduced by Figalli and Serra to study the fine properties of the singular set in the classical obstacle problem. 
As a result, under a superharmonicity condition on the obstacle, we prove that each stratum of the singular set is locally contained in a single $C^2$ manifold, up to a lower dimensional subset, and the top stratum is locally contained in a $C^{1,\alpha}$ manifold for some $\alpha > 0$ if $a < 0$. 

In studying the top stratum, we discover a dichotomy, until now unseen, in this problem (or, equivalently, the fractional obstacle problem). 
We find that second blow-ups at singular points in the top stratum are global, homogeneous solutions to a codimension two lower dimensional obstacle problem (or fractional thin obstacle problem) when $a < 0$, whereas second blow-ups at singular points in the top stratum are global, homogeneous, and $a$-harmonic polynomials when $a \geq 0$.
To do so, we establish regularity results for this codimension two problem, what we call the very thin obstacle problem.

Our methods extend to the majority of the singular set even when no sign assumption on the Laplacian of the obstacle is made.
In this general case, we are able to prove that the singular set can be covered by countably many $C^2$ manifolds, up to a lower dimensional subset.
\end{abstract}
\begin{changemargin}{-0.7cm}{-0.7cm}
\maketitle
\end{changemargin}

\section{Introduction}

Lower dimensional obstacle problems are an important class of obstacle problems, arising in many areas of mathematics. 
For instance, they can be found in the theory of elasticity (see \cite{Sig33, Sig59, KO88}), and they also appear in describing osmosis through semi-permeable membranes as well as boundary heat control (see, e.g., \cite{DL76}).
Moreover, they often are local formulations of fractional obstacle problems, another important class of obstacle problems.
Fractional obstacle problems can be found in the optimal stopping problem for L\'evy processes, and can be used to model American option prices (see \cite{Mer76,CT04}). 
They also appear in the study of anomalous diffusion, \cite{BG90}, the study of quasi-geostrophic flows, \cite{CV10}, and in studies of the interaction energy of probability measures under singular potentials, \cite{CDM16}. 
(We refer to \cite{Ros18} for an extensive bibliography on the applications of obstacle-type problems.)

Broadly, lower dimensional obstacle problems are minimization problems for a given energy functional on class of functions constrained to sit above a given obstacle (function) defined on a lower dimensional manifold.
Obstacle problems are free boundary problems: the principal part of their study is the structure and regularity of the boundary of the contact set of the solution and the obstacle, the free boundary.
The lower dimensional obstacle problem we consider\,---\,the thin obstacle problem with weight $|x_{n+1}|^a$\,---\,has garnered much interest and attention (see \cite{AC04, CS07, ACS08, GP09, KRS16, FoSp18, CSV19, JN17}); it is a model setting, and has motivated the study of many other types of lower dimensional obstacle problems (see \cite{MS08,AM11, Fer16, RS17, RuSh17, FeSe18, FoSp18b, GR18, BLOP19}).

Nevertheless, the study of the non-regular part of the free boundary has been rather limited.
Only recently has significant progress been made (see \cite{GP09, FoSp18,  GR18, CSV19}). 
And many open questions still remain. 
In this work, we address some of these questions, focusing on the singular set (see Section~\ref{sssec.fb}).
In particular, we explore the fine properties of the solution and its expansion around singular points, inspired by \cite{FS18}.

We note that the techniques of \cite{FS18} have been further developed and improved in \cite{FRS19}, where the authors prove generic regularity (namely, the generic smoothness of the free boundary in the classical obstacle problem) in dimension three and the smoothness of the free boundary at almost every time for the three-dimensional Stefan problem.
We expect the machinery built here to be useful in tackling genericness-type questions of this nature in the context of the thin/fractional obstacle problem.

\subsection{The Thin Obstacle Problem}

In this paper, we consider a class of lower dimensional obstacle problems in $\R^{n+1} := \{ X = (x,y)\in \R^n \times \R \}$ with weight $|y|^a$ where $\R^n\times\{0\}$ acts as the lower dimensional manifold. 
We will often refer to them as, simply, the thin obstacle problem, even though this name is usually reserved for the case $a = 0$.
In particular, for an analytic {\it obstacle} $\varphi : B_1 \cap \{ y = 0 \} \to \R$, we look at the thin obstacle problem:
\begin{equation}
\label{eqn: Signorini problem}
\min_{w \in \mathscr{A}} \bigg\{ \int_{B_1} |\nab w|^2|y|^a \, \d X \bigg\}, \quad\text{with}\quad a \in (-1,1),
\end{equation}
where $\mathscr{A}$ is the convex subset of the Sobolev space $W^{1, 2}(B_1 , |y|^a \, \d X)$ (which, for simplicity, we call $W^{1, 2}(B_1, |y|^a)$) defined by 
\[
\mathscr{A} := \{ w \in W^{1,2}_0(B_1, |y|^a) + g : w(x, 0) \ge \varphi(x) \text{ and } w(x,-y) = w(x,y) \},
\]
given some boundary data $g \in C(B_1)$ (even with respect to $y$) such that $g|_{\pa B_1 \cap\{y = 0\}} \geq \varphi$.
The condition that $w$ sits above $\varphi$ on the {\it thin space} $\R^n \times \{ 0 \}$ needs to be understood in the trace sense, a priori. 

If $u$ is the (unique) solution to \eqref{eqn: Signorini problem}, then $u$ satisfies the Euler--Lagrange equations
\begin{equation}
\label{eq.EL}
\left\{ 
\begin{array}{rcll}
u(x,y) &\geq& \varphi(x) &\text{on } B_1 \cap \{ y = 0 \} \\
L_a u(x,y) &\leq& 0 &\text{in } B_1\\
L_a u(x,y) &=& 0 &\text{in } B_1 \setminus \Lambda(u) \\
u(x, y) &=& u(x,-y) &\text{in } B_1 \\
u(x,y) &=& g(x,y) &\text{on } \pa B_1
\end{array}\right.
\end{equation}
where 
\[
L_a u(x, y) := \DIV (|y|^a\nabla u(x,y))
\]
and
\[
\Lambda(u) := \{ (x,0) : u(x,0) = \varphi(x) \}.
\]
The set $\Lambda(u)$ is called the {\it contact set}, and is an unknown of the problem. 
Its topological boundary in $\R^n$
\[
\Gamma(u) := \pa \Lambda(u) \subset \R^n \times \{ 0 \}
\]
is called the {\it free boundary}.

\begin{remark}
\label{rmk: min super}
A useful equivalent characterization of the minimizer $u$ of \eqref{eqn: Signorini problem} is that $u$ is the smallest super $a$-harmonic function in $\mathscr{A}$: $u \in \mathscr{A}$, $L_a u \leq 0$, and $u \leq w$ for all $w \in \mathscr{A}$ such that $L_a w \leq 0$.   
\end{remark}

\begin{remark}
In this work, we consider analytic obstacles.
Clearly, this regularity restriction can be relaxed; the thin obstacle problem \eqref{eqn: Signorini problem} can be well-formulated with significantly less regular obstacles (e.g., continuous obstacles).
That said, the analytic setting allows us to understand the model behavior of $\Gamma(u)$, and for this reason, it deserves special consideration.
\end{remark}

\subsubsection{The Obstacle Problem for the Fractional Laplacian}
\label{sssec.FOP}

As shown in \cite{CSS08}, the Euler--Lagrange equations \eqref{eq.EL} appear naturally in the context of the obstacle problem for the fractional Laplacian, or the fractional obstacle problem. 
Indeed, let $\varphi: \R^n\to \R$ be an obstacle (with sufficient decay at infinity) and let $\bar u$ solve the fractional obstacle problem
\begin{equation}
\label{eq.FOP}
\left\{ 
\begin{array}{rcll}
\bar u &\geq& \varphi &\text{in } \R^n \\
(-\Delta)^s \bar u& \geq & 0 &\text{in } \R^n\\
(-\Delta)^s \bar u& = & 0 &\text{in }  \{\bar u >\varphi \}\\
\lim_{|x|\to \infty} \bar u(x) & = & 0
\end{array}\right.
\quad\text{with}\quad s := \frac{1-a}{2} \in (0,1).
\end{equation}
Then, the even in $y$, $a$-harmonic extension of $\bar u$ to $\R^{n+1}$ (i.e., $u: \R^{n+1}\to \R$ such that $L_a u(x, y) = 0$ for $|y|>0$, $u(x, 0) = \bar u(x)$, $u(x,y) = u(x,-y)$, and $\lim_{|(x,y)| \to \infty} u(x,y) = 0$) solves \eqref{eq.EL} in $\R^{n+1}$ (and, in particular, with its own boundary data, in $B_1$).
Consequently, all of the results we prove in this work can be translated into statements regarding the fractional obstacle problem.
We leave this translation to the interested reader.

\subsection{Known Results}
\label{sec: known results}
Let us briefly summarize some of the known properties of the solution to the thin obstacle problem and its free boundary.
To do so, it will be useful to ``normalize'' $\varphi$, and it will be necessary to define a collection of rescalings of $u$.

Since $\varphi = \varphi(x)$ is analytic, we can extend it from a function defined on $B_1 \cap \{ y = 0 \}$ to an $a$-harmonic, even in $y$ function defined on $\overline{B_1}$ (see \cite[Lemma 5.1]{GR18}).
For simplicity, we still denote this extension by $\varphi$.
So if we let
\begin{equation}
\label{def: u - Phi}
\tilde{u} := u - \varphi,
\end{equation}
\eqref{eq.EL} becomes
\begin{equation}
\label{eq.EL_0}
\left\{ 
\begin{array}{rcll}
\tilde{u}(x,y) &\geq& 0 &\text{on } B_1 \cap \{ y = 0\}\\
L_a \tilde{u}(x,y) &\leq& 0 &\text{in } B_1\\
L_a \tilde{u}(x,y) &=& 0 &\text{in } B_1 \setminus \Lambda(\tilde{u}) \\
\tilde{u}(x, y) &=& \tilde{u}(x,-y) &\text{in } B_1 \\
\tilde{u}(x,y) &=& \tilde{g}(x,y) &\text{on } \pa B_1,
\end{array}\right.
\end{equation}
with $\tilde{g} := g - \varphi$ and 
\[
\Lambda(\tilde{u}) := \{(x, 0) : \tilde{u}(x, 0) = 0\} = \Lambda(u).
\]
Furthermore, 
\begin{equation}
\label{eqn: Lap of u}
L_a \tilde{u} = 2\lim_{y \downarrow 0 } y^a\pa_y \tilde{u}(x, y) \mathcal{H}^{n}\mres {\Lambda(\tilde{u})}.
\end{equation}
Hence, considering \eqref{eq.EL_0},
\[
\lim_{y \downarrow 0 } y^a\pa_y \tilde{u}(x, y) \leq 0 \quad\text{for}\quad |x| < 1,
\]
\[\lim_{y \downarrow 0 } y^a\pa_y \tilde{u}(x, y) = 0 \quad\text{for}\quad |x| < 1 \text{ and } \tilde{u}(x,0) > 0,
\]
and
\[
\tilde{u}\, L_a \tilde{u} = 0\quad\text{in}\quad B_1.
\]
(See \cite{CSS08,GP09,FoSp18,GR18}.)
All of the above expressions must be understood in a distributional sense.

As we have mentioned, we need to introduce a collection of rescalings of $u$ around a free boundary point $X_\circ\in \Gamma(u)$ in order to outline the existing literature on \eqref{eqn: Signorini problem}.
They are 
\begin{equation}
\label{eq.rescGR}
\tilde{u}_{X_\circ,r}(X) := \frac{\tilde{u}_{X_\circ}(rX)}{\left(\frac{1}{r^{n+a}}\int_{\pa B_r} \tilde{u}_{X_\circ}^2|y|^a\right)^{1/2}} \quad\text{where}\quad \tilde{u}_{X_\circ}(X) := \tilde{u}(X_\circ + X).
\end{equation}

\subsubsection{Blow-ups and Optimal Regularity}
 
In \cite{ACS08,CSS08}, Athanasopoulos, Caffarelli, and Salsa and Caffarelli, Salsa, and Silvestre, for $a = 0$ and $a \in (-1,1)$ respectively, proved that the set $\{\tilde{u}_{X_\circ,r}\}_{r > 0}$ is weakly precompact in $W_{\rm loc}^{1,2}(\R^{n+1},|y|^a)$, and that the limit points of $\{\tilde{u}_{X_\circ,r}\}_{r > 0}$ as $r \downarrow 0$ or {\it blow-ups of $u$ at $X_\circ$} are global $\lambda_{X_\circ}$-homogeneous solutions to \eqref{eq.EL_0} with 
\[
\lambda_{X_\circ} \in [1+s,\infty) \quad\text{for}\quad s := \frac{1-a}{2}.
\]
It is important to note that the homogeneity of blow-ups depends only on the point $X_\circ\in \Gamma(u)$ at which they are taken, and is independent of the sequence along which the weak limit is produced.

Moreover, in \cite{AC04,CSS08}, it was shown that $u$ is optimally $C^{1,s}$ on either side of the thin space (but only Lipschitz across).

\subsubsection{The Free Boundary}
\label{sssec.fb}
The free boundary $\Gamma(u)$ can be partitioned into three sets:
\[
\Gamma(u) = {\rm Reg}(u) \cup {\rm Sing}(u) \cup {\rm Other}(u),
\]
the set of {\it regular points}, the set of {\it singular points}, and set of {\it other points} (see \cite{GP09,FoSp18,GR18}), and they can be characterized by the value of $\lambda_{X_\circ}$ with $X_\circ \in \Gamma(u)$.

${\rm Reg}(u)$ is the set of free boundary points where blow-ups are $(1+s)$-homogeneous.
In \cite{ACS08,CSS08}, it was proved that ${\rm Reg}(u)$ is relatively open, that blow-ups at points in ${\rm Reg}(u)$ are unique, and that ${\rm Reg}(u)$ is an $(n-1)$-dimensional $C^{1,\alpha}$ submanifold of the thin space (it is analytic, in fact, as proved in \cite{KRS16}).

${\rm Sing}(u)$ is the set of points in $\Gamma(u)$ where the contact set has zero $\mathcal{H}^n$-density,
\[
\text{Sing}(u) := \bigg\{ X_\circ \in \Gamma(u) : \lim_{r \downarrow 0} \frac{\mathcal{H}^{n}(\Lambda(u) \cap B_r(X_\circ))}{r^n} = 0 \bigg\}.
\]
In \cite{GP09, GR18}, Garofalo and Petrosyan and Garofalo and Ros-Oton, for $a = 0$ and $a \in (-1,1)$ respectively, proved that the points of ${\rm Sing}(u)$ are those at which blow-ups are evenly homogeneous and unique.
In addition, they showed that ${\rm Sing}(u)$ is contained in the countable union of $m$-dimensional $C^1$ manifolds with $m$ ranging from $0$ to $n-1$. 
(The regularity of the covering manifolds was later improved to a more quantitative $C^{1,\log^{\vep_\circ}}$ in \cite{CSV19} when $a = 0$.) The goal of this manuscript is to achieve a better understanding of singular points.

Finally, ${\rm Other}(u)$ is the remainder of the free boundary, and is not yet fully characterized.
That said, in \cite{FoSp18}, Focardi and Spadaro proved that $\Gamma(u)$, in particular, ${\rm Other}(u)$, has finite $(n-1)$-dimensional Minkowski content, which implies that the free boundary is $\mathcal{H}^{n-1}$-rectifiable.
Moreover, they showed that outside of an at most Hausdorff $(n-2)$-dimensional subset of $\Gamma(u)$, the possible homogeneities of blow-ups take values in $\{ 2k, 2k - 1 + s, 2k + 2s \}_{k \in \N}$ (the same result was proved for $a = 0$ by Krummel and Wickramasekera in \cite{KW13}).

\subsubsection{The Non-degenerate Problem}

We have already seen that the study of the thin obstacle problem for an analytic obstacle can be reduced to the study of the thin obstacle problem for the zero obstacle, \eqref{eq.EL_0}.
An alternative normalization is to reduce to the zero boundary data case by subtracting off the $a$-harmonic extension of $g$ to $B_1$. 
Indeed, for simplicity, let $g$ be its own $a$-harmonic extension to $B_1$, i.e., assume that $g$ is defined on $\overline{B_1}$ and $L_a g = 0$ in $B_1$.
Then, $u-g$ solves \eqref{eqn: Signorini problem} with zero boundary data and obstacle $\varphi_g := (\varphi-g)|_{\{y=0\}}$.
(This procedure does not require $\varphi$ to be analytic.)
Under this normalization, Barrios, Figalli, and Ros-Oton proved that if $\varphi_g$ is strictly superharmonic, then 
\[
\lambda_{X_\circ} \in \{ 1+s , 2 \},
\]
for all $X_\circ \in \Gamma(u)$ (see \cite{BFR18}).
Consequently, we make the following definition.

\begin{definition}
\label{def.nondeg}
We say that the thin obstacle problem \eqref{eqn: Signorini problem} or, equivalently, \eqref{eq.EL} is non-degenerate if
\begin{equation}
\label{eqn: nondeg cond}
\Delta_x \varphi_g \leq -c < 0 \quad\text{on}\quad B_1 \cap \{ y = 0 \}.
\end{equation}
Analogously, we say the Euler--Lagrange equations \eqref{eq.EL_0} are non-degenerate if they arise from \eqref{eqn: Signorini problem} or \eqref{eq.EL} satisfying \eqref{eqn: nondeg cond}; i.e., $\Delta_x \tilde g \geq c > 0$ on $B_1 \cap \{ y= 0\}$, where $\tilde g$ denotes its own $a$-harmonic extension of $\tilde g$ to $\overline{B_1}$.
\end{definition}

\begin{remark}
In the context of the obstacle problem for the fractional Laplacian in all of $\R^n$, \eqref{eq.FOP}, the problem is non-degenerate under the less restrictive assumption $\Delta \varphi \le 0$ in $\{\varphi > 0\}\subset \R^{n}$.
\end{remark}

\subsection{Main Results}

We are interested in studying the fine properties of $u$ at points in ${\rm Sing}(u)$, in the spirit of the work of Figalli and Serra (\cite{FS18}), wherein such a study is undertaken for the classical obstacle problem given obstacles with Laplacian identically equal to $-1$, i.e., under a non-degeneracy condition (cf. Definition~\ref{def.nondeg}).
To do so, we establish a framework to better characterize the structure of singular points and the behavior of $u$ around singular points: we develop a higher order expansion of $u$ around singular points, which, up to lower dimensional sets, yields a more regular covering of ${\rm Sing}(u)$. 
Our approach and results are new even for the case $a = 0$.

Before stating our results, it will be convenient to expand our discussion of ${\rm Sing}(u)$ and the work of \cite{GP09, GR18}, and introduce some notation.
Let
\[
\Sigma_\kap(u) := \{ X_\circ \in \Gamma(u) : \lambda_{X_\circ} = \kappa \}
\]
denote the set of free boundary points where the homogeneity of blow-ups is $\kap$.
Consequently,
\begin{equation}
\label{eq.strat1}
{\rm Sing}(u) = \bigcup_{\kap \in 2\N} \Sigma_\kap(u).
\end{equation}
As noted, in \cite{GP09, GR18}, the authors showed that one and only one blow-up exists, which is evenly homogeneous, at each singular point.
In fact, they proved much more: the unique blow-up at a singular point is a non-trivial, $a$-harmonic, evenly homogeneous polynomial that is even in $y$ and non-negative on the thin space.
In other words, blow-ups at singular points belong to the set of polynomials
\[
\mathscr{P}_\kap := \{ p :  L_a p = 0,\, X\cdot \nabla p (X) = \kap p(X), \, p(x,0) \geq 0,\, p(x,-y) = p(x,y)\}
\]
for $\kappa \in 2\N$.
Furthermore, they produce the first term in the expansion of $u$ around $X_\circ \in \Sigma_\kappa(u) \subset {\rm Sing}(u)$; they show that
\begin{equation}
\label{eq.resck}
\frac{\tilde{u}(X_\circ + r\,\cdot\,)}{r^{\kap}} \to p_{\ast,X_\circ} \in \mathscr{P}_\kap \quad\text{locally uniformly as}\quad r \downarrow 0.
\end{equation}
The polynomial $p_{\ast,X_\circ}$, which we call the {\it first blow-up of $u$ at $X_\circ$}, is a constant (non-zero) multiple of the blow-up of $u$ at $X_\circ$ given by the rescalings \eqref{eq.rescGR}. 
With the rescalings \eqref{eq.resck}, we have
\begin{equation}
\label{eq.littleork}
\tilde{u}(X) = p_{*,X_\circ}(X-X_\circ) + o(|X-X_\circ|^\kap).
\end{equation}
Finally, consider
\[
L(p_{\ast, X_\circ}) := \{ \xi \in \R^{n} : \xi \cdot \nab_{x} p_{\ast,X_\circ}(x,0) = 0\text{ for all } x \in \R^{n} \}
\]
the {\it invariant set} or {\it spine} of $p_{*,X_\circ}$ on $\{ y = 0\}$ as well as
\[
m_{X_\circ} := \dim L(p_{*,X_\circ}).
\]
Observe that $L(p_{\ast, X_\circ})$ is a linear subspace of $\R^n$.
Also, since $p_{*,X_\circ} \not\equiv 0$ on $\R^{n} \times \{ 0 \}$,
\[
m_{X_\circ}\in \{ 0,1, \dots, n-1 \},
\] 
and this number accounts for the dimension of the contact set around a singular point. 
Thus, the singular set can be further stratified:
\begin{equation}
\label{eq.strat2}
{\rm Sing}(u) = \bigcup_{\kappa \in 2\N} \bigcup_{m =0}^{n-1} \Sigma_\kap^m(u) \quad\text{where}\quad \Sigma_\kap^m(u) := \{ X_\circ \in \Sigma_\kap(u)  : m_{X_\circ} = m \}.
\end{equation}

In particular, by \cite{BFR18}, if the problem is non-degenerate (see Definition~\ref{def.nondeg}), then
\[
\Gamma(u) = {\rm Reg}(u) \cup {\rm Sing}(u) =  {\rm Reg}(u) \cup \Sigma_2(u) = {\rm Reg}(u) \cup \bigcup_{m = 0}^{n-1} \Sigma^m_2(u).
\]
Now we are ready to present the main results of this work.
First, given a non-degenerate obstacle, we prove that each $m$-dimensional component of $\text{Sing}(u)$ can be locally covered by a single $C^2$ manifold outside a lower dimensional set:  

\begin{theorem}
\label{thm: nondeg main}
Let $u$ solve \eqref{eqn: Signorini problem} in the non-degenerate case (see Definition~\ref{def.nondeg}).
Then, 
\begin{enumerate}[(i)]
\item $\Sigma_2^{0}(u)$ is isolated in ${\rm Sing}(u) = \Sigma_2^0(u) \cup \cdots \cup \Sigma_2^{n-1}(u)$.
\item There exists an at most countable set $E_{1} \subset \Sigma_2^1(u)$ such that $\Sigma_2^1(u) \setminus E_{1}$ is locally contained in a single one-dimensional $C^{2}$ manifold. 
\item For each $m \in \{2,\dots,n-1\}$, there exists a set $E_m \subset \Sigma_2^m(u)$ of Hausdorff dimension at most $m-1$ such that $\Sigma_2^m(u) \setminus E_m$ is locally contained in a single $m$-dimensional $C^{2}$ manifold.
\item If $a\in (-1, 0)$, $\Sigma_2^{n-1}(u)$ is locally contained in a single $(n-1)$-dimensional $C^{1,\alpha}$ manifold, for some $\alpha > 0$ depending only on $n$.
\end{enumerate}
\end{theorem}

The framework we develop in order to prove Theorem~\ref{thm: nondeg main} is rather robust, and only sees the non-degeneracy condition \eqref{eqn: nondeg cond} superficially.
As a result, we can suitably extend Theorem~\ref{thm: nondeg main} to the bulk of ${\rm Sing}(u)$, the top stratum $\Sigma^{n-1}(u) := \bigcup_{\kappa \in 2\N} \Sigma_\kappa^{n-1}(u)$, in the general case.
Recall that the lower stratum $\Sigma^{< n-1}(u) := \text{Sing}(u) \setminus \Sigma^{n-1}(u)$ is strictly lower dimensional; it is contained in the countable union of $(n-2)$-dimensional $C^1$ manifolds.
More precisely, we prove
\begin{theorem}
\label{thm: deg main}
Let $u$ solve \eqref{eqn: Signorini problem}.
Then, 
\begin{enumerate}[(i)]
\item $\Sigma_2^{0}(u)$ is isolated in ${\rm Sing}(u) = \bigcup_{\kappa \in 2\N} \bigcup_{m =0}^{n-1} \Sigma_\kap^m(u)$.
\item There exists an at most countable set $E_{2,1} \subset \Sigma_2^1(u)$ such that $\Sigma_2^1(u) \setminus E_{2,1}$ is contained in the countable union of one-dimensional $C^{2}$ manifolds.
\item For each $m \in \{2,\dots,n-1\}$, there exists a set $E_{2,m} \subset \Sigma_2^m(u)$ of Hausdorff dimension at most $m-1$ such that $\Sigma_2^m(u) \setminus E_{2,m}$ is contained in the countable union of $m$-dimensional $C^{2}$ manifolds. 
\end{enumerate}

Moreover, for each $\kap \in 2\N$,
\begin{enumerate}[(i)]
\setcounter{enumi}{3}
\item If $n = 2$, there exists an at most countable set $E_{\kappa,1} \subset \Sigma_\kappa^{1}(u)$ such that $\Sigma_\kappa^{1}(u) \setminus E_{\kappa,1}$ is contained in the countable union of $1$-dimensional $C^{2}$ manifolds.
\item If $n \geq 3$, there exists a set $E_{\kappa,n-1} \subset \Sigma_\kappa^{n-1}(u)$ of Hausdorff dimension at most $n-2$ such that $\Sigma_\kappa^{n-1}(u) \setminus E_{\kappa,n-1}$ is contained in the countable union of $(n-1)$-dimensional $C^{2}$ manifolds.
\item If $n\ge 2$ and $a \in (-1,0)$, $\Sigma_\kap^{n-1}(u)$ can be covered by a countable union of $(n-1)$-dimensional $C^{1,\alpha_\kap}$ manifolds, for some $\alpha_\kap > 0$ depending only on $n$ and $\kappa$.
\end{enumerate}
\end{theorem}

\begin{remark}
Notice that from the lower-dimensionality of $\Sigma_\kap^{<n-1}(u)$, by Theorem~\ref{thm: deg main}(iv) and (v), we find that the whole singular set can be covered by countably many $(n-1)$-dimensional $C^2$ manifolds up to a lower dimensional subset.
\end{remark}

\begin{remark}
When $n = 1$, it is well-known that singular points are isolated. 
Recall that $\tilde{u}(X_\circ+\,\cdot\,) = p_{*, X_\circ} + o(|X|^\kap)$ if $X_\circ \in \Sigma_\kappa(u)$.
Since $n=1$, $p_{*, X_\circ} > 0$ in a neighborhood of $0$, so that $\tilde{u} > 0$ around $X_\circ$ and $X_\circ$ is isolated.
\end{remark}

Before stating Theorem~\ref{thm: deg main}, we noted that our methods see the non-degeneracy of the problem superficially.
Indeed, if we could show that $p_{\ast,X_\circ}$'s nodal set $\{ (x,0) : p_{\ast,X_\circ}(x,0) = |\nab_{x}p_{\ast,X_\circ}(x,0)| = 0\}$ and $p_{\ast,X_\circ}$'s spine align for every $X_\circ \in E_{\kap,m}$ (see Section~\ref{sec.thirdorder} (also \ref{sec:size}) for a description of $E_{\kap, m}$), then our analysis would immediately imply that $E_{\kap, m}$ is lower dimensional, and $\Sigma_\kappa^m(u) \subset {\rm Sing}(u)$ is contained in a countable union of $C^2$ manifolds up to an $(m-1)$-dimensional subset for all $m \in \{ 0, \dots, n-1 \}$, and not just when $m = n-1$.
 
We remark that due to potential accumulation of lower homogeneity singular points to higher homogeneity singular points, the countable covers of Theorem~\ref{thm: deg main} cannot be improved to single covers, as done in the the non-degenerate setting, Theorem~\ref{thm: nondeg main} (and also as done in \cite{FS18}).

\subsection{Strategy of the Proof}

From this point forward, we do not distinguish $u$ and $\tilde{u}$, as defined in \eqref{def: u - Phi} (or we assume that $\varphi \equiv 0$); we will always assume that we are in the normalized situation \eqref{eq.EL_0}.
Furthermore, in this section, whenever we discuss $\Sigma_{\kappa}(u)$, $\kappa \in 2\N = \{2,4,6,\dots\}$.

Theorems~\ref{thm: nondeg main} and \ref{thm: deg main} are the culmination of a procedure that constructs the second term in the expansion of $u$ at singular points, outside of a lower dimensional set.
In order to study the higher infinitesimal behavior of $u$ at $X_\circ \in \Sigma_{\kappa}(u)$, we, quite naturally, consider the rescalings
\[
\tilde{v}_{X_\circ,r}(X) := \frac{v_{X_\circ}(rX)}{\left(\frac{1}{r^{n+a}} \int_{\pa B_r} v_{X_\circ}^2|y|^a\right)^{1/2}} \quad\text{where}\quad v_{X_\circ}(X) := u(X_\circ + X) - p_{\ast,X_\circ}(X)
\]
(cf. \eqref{eq.rescGR}).

First, we show that the set $\{ \tilde{v}_{X_\circ,r} \}_{r > 0}$ is weakly precompact in $W_{\rm loc}^{1,2}(\R^{n+1},|y|^a)$ and classify its limit points as $r \downarrow 0$ or blow-ups (see Sections~\ref{sec:monotone} and \ref{sec:Blow-up Analysis}):

\begin{proposition}
\label{prop.main1}
Let $u$ solve \eqref{eqn: Signorini problem}, and let $X_\circ \in \Sigma_\kappa^{m}(u)$ for $m \in \{0,\dots,n-1\}$.
\begin{enumerate}[(i)]
\item If $a \in [0, 1)$, the limit points of $\{ \tilde{v}_{X_\circ,r} \}_{r > 0}$ as $r \downarrow 0$ are $\lambda_{\ast,X_\circ}$-homogeneous, $a$-harmonic polynomials with $\lambda_{\ast,X_\circ} \geq \kappa$.
\item If $m < n -1$ and $\kap = 2$, the limit points of $\{ \tilde{v}_{X_\circ,r} \}_{r > 0}$ as $r \downarrow 0$ are $\lambda_{\ast,X_\circ}$-homogeneous, $a$-harmonic polynomials with $\lambda_{\ast,X_\circ} \geq 2$.
\item If $m = n - 1$ and $a \in (-1, 0)$, the limit points of $\{ \tilde{v}_{X_\circ,r} \}_{r > 0}$ as $r \downarrow 0$ are $\lambda_{\ast,X_\circ}$-homogeneous, global solutions to the very thin obstacle problem (or fractional thin obstacle problem) \eqref{eq.VTOP} on $L(p_{\ast,X_\circ}) \subset \R^n \times \{ 0 \}$ with $\lambda_{\ast,X_\circ} \ge \kappa +\alpha_\kappa$, for some $\alpha_\kappa > 0$ depending only on $n$ and $\kappa$.
\end{enumerate} 
\end{proposition}

As far as we know, Proposition~\ref{prop.main1} is the first instance of truly distinct behavior within our class of lower dimensional obstacle problems; in all previous studies of \eqref{eqn: Signorini problem}, the class parameterized by $a \in (-1,1)$ was treatable uniformly.
The key difference is that if $a \geq 0$, subsets of the thin space $\{y = 0\}$ of Hausdorff dimension $n-1$ have zero $W^{1,2}(\R^{n+1},|y|^a)$-capacity or $a$-harmonic capacity, while if $a < 0$, subsets of the thin space $\{y = 0\}$ of Hausdorff dimension $n-1$ have positive $a$-harmonic capacity.
This capacitory distinction permits the formulation of, what we call, a very thin obstacle problem, i.e., a search for a weighted Dirichlet energy minimizer, as in \eqref{eqn: Signorini problem}, within a class of functions constrained to sit above a given function defined on an $(n-1)$-dimensional submanifold of $\R^n \times \{ 0 \}$ (see Section~\ref{sec.VTOP}), 
or, equivalently, a lower dimensional obstacle problem for the fractional Laplacian $(-\Delta)^s$ where $s> \frac12$
(see Section~\ref{sec.GlobPbs} and cf. Section~\ref{sssec.FOP}).

We remark that the above classification in the case $a < 0$ is analogous to the classification found in \cite{FS18}, wherein Figalli and Serra consider the classical obstacle problem.
There, the analogous blow-ups in the top stratum of the singular set are global, homogeneous solutions to the  thin obstacle problem \eqref{eqn: Signorini problem} with zero obstacle and $a = 0$.
And in the lower stratum of the singular set, the analogous blow-ups set are homogeneous, harmonic polynomials. 
That said, while Figalli and Serra could rely on developed theory (for the thin obstacle problem) for their analysis, we cannot; the very thin obstacle problem has, until now, been unstudied (Section~\ref{sec.VTOP}).

Given Proposition~\ref{prop.main1} and our desire to produce the next term in the expansion of $u$ at $X_\circ$, we then show that collection of points for which $\lambda_{\ast,X_\circ} \in [\kappa, \kappa + 1)$ is lower dimensional (for $\kap = 2$ or $m = n-1$).
More specifically, if we define 
\[
\Sigma_{\kappa}^{m,{\rm a}}(u) := \{ X_\circ \in \Sigma_{\kappa}^{m}(u) : \lambda_{\ast,X_\circ} \in [\kappa, \kappa + 1) \},
\]
then we have the following proposition.

\begin{proposition}
\label{prop.main2}
Let $u$ solve \eqref{eqn: Signorini problem}.
Then, 
\begin{enumerate}[(i)]
\item $\Sigma_{2}^{0,{\rm a}}(u)$ is empty.
\item For each $m \in \{ 1 \dots, n-1\}$, $\Sigma_{2}^{m,{\rm a}}(u)$ has Hausdorff dimension at most $m-1$.
\item For each $\kappa \in 2\N$, $\Sigma_{\kappa}^{n-1,{\rm a}}(u)$ has Hausdorff dimension at most $n-2$.
\end{enumerate}
\end{proposition}
\begin{remark}
\label{rem.zerodim}
In fact, we can show that for $n = 2$, if $a \in (-1, 0)$, then $\Sigma_{\kap}^{1,{\rm a}}(u)$ is countable; and if $a \in [0,1)$, then $\Sigma_{\kap}^{1,{\rm a}}(u)$ is discrete. 
Moreover, for $n \ge 3$, $\Sigma_2^{1,{\rm a}}(u)$ is discrete.
\end{remark}
\noindent In turn, we call $\Sigma_{\kappa}^{m,{\rm a}}(u)$ the set of {\it anomalous} points of $\Sigma_{\kappa}^{m}(u)$ and
\[
\Sigma_{\kappa}^{m,{\rm g}}(u) := \Sigma_{\kappa}^{m}(u) \setminus \Sigma_{\kappa}^{m,{\rm a}}(u)
\] 
the {\it generic} points of $\Sigma_{\kappa}^{m}(u)$ (cf. \cite{FS18}).
(See Sections~\ref{sec:accumulation} and \ref{sec:size}.) In order to prove Proposition~\ref{prop.main2}, we use two Federer-type dimension reduction arguments.
When $a \geq 0$ or $m < n-1$, we argue as in \cite{FS18}, while when $a < 0$ and $m = n-1$, we adopt the arguments pioneered in \cite{FRS19}.

After the statement of Theorem~\ref{thm: deg main}, we remarked that if the nodal set and spine of $p_{\ast,X_\circ}$ were aligned for each $X_\circ \in E_{\kap,m}$, then Theorem~\ref{thm: deg main} would immediately hold for all $m \in \{0,\dots,n-1\}$ and all $\kappa \in 2\N$.
(Notice that this alignment is always true when $m \in \{0,\dots,n-1\}$ if $\kappa = 2$, but only when $m = n-1$ if $\kappa > 2$.)
Another way to understand this remark is as follows. 
If the nodal set and spine of $p_{\ast,X_\circ}$ were aligned for each $X_\circ \in \Sigma_{\kappa}^{m,{\rm a}}(u)$, then our analysis would directly show that $\Sigma_\kappa^{m,{\rm a}}(u)$ is at most $(m-1)$-dimensional (in the Hausdorff sense), extending Proposition~\ref{prop.main2} to every $(\kappa,m)$ pair.
Hence, Theorem~\ref{thm: deg main} would immediately hold for all $m \in \{0,\dots,n-1\}$ and all $\kappa \in 2\N$ since every other aspect of our analysis is indifferent to this issue. 
Nonetheless, it is unclear if such a statement is true; in fact, Remark~\ref{rem.counterexample} indicates (but does not prove) the opposite.

Thanks to Propositions~\ref{prop.main1} and \ref{prop.main2}, and Whitney's Extension Theorem, generic points are contained in the countable union of $C^{1,1}$ manifolds; and so, we have the following result, which is Theorem~\ref{thm: deg main}, but with $C^{1,1}$ coverings.

\begin{theorem}
\label{thm.main1}
Let $u$ solve \eqref{eqn: Signorini problem}.
Then, 
\begin{enumerate}[(i)]
\item $\Sigma_2^{0}(u)$ is isolated in ${\rm Sing}(u) = \bigcup_{\kappa \in 2\N} \bigcup_{m = 0}^{n-1} \Sigma_\kap^{m}(u)$.
\item For each $m \in \{1,\dots,n-1\}$, $\Sigma_2^{m}(u)\setminus \Sigma_2^{m,{\rm a}}(u)$ is contained in the countable union of $m$-dimensional $C^{1,1}$ manifolds, where $\dim_{\mathcal{H}} \Sigma_2^{m,{\rm a}}(u)\le m-1$.
\end{enumerate}
Moreover, for each $\kappa \in 2\N$,
\begin{enumerate}[(i)]
\setcounter{enumi}{2}
\item $\Sigma_\kappa^{n-1}(u)\setminus \Sigma_\kappa^{n-1,{\rm a}}(u)$ is contained in the countable union of $(n-1)$-dimensional $C^{1,1}$ manifolds, where $\dim_{\mathcal{H}} \Sigma_\kappa^{n-1,{\rm a}}(u)\le n-2$.
\item In addition, if $a \in (-1,0)$, each $\Sigma_\kap^{n-1}(u)$ can be covered by a countable union of $(n-1)$-dimensional $C^{1,\alpha_\kap}$ manifolds, for some $\alpha_\kap > 0$ depending only on $n$ and $\kappa$.
\end{enumerate}
\end{theorem}
\noindent (See Section~\ref{sec:whitney}.) 
We refer to Remark~\ref{rem.zerodim} for the size of the anomalous set in the cases $n = 2$ and $m = 1$, which corresponds to parts (ii) and (iv) of Theorem~\ref{thm: deg main}.
Just as Theorem~\ref{thm.main1} is a $C^{1,1}$ precursor to Theorem~\ref{thm: deg main}, we note that a $C^{1,1}$ precursor to Theorem~\ref{thm: nondeg main} also holds.

To conclude the proofs of our main results and produce the next term in the expansion of $u$ outside a lower dimensional set (and go from $C^{1,1}$ to $C^2$ covering manifolds), we prove that outside of an at most $(m-1)$-dimensional (in the Hausdorff sense) subset of $\Sigma_\kappa^{m,{\rm g}}(u)$, when $\kappa = 2$ and $m \in \{0, \dots, n-1 \}$ as well as when $\kappa > 2$ and $m=n-1$, the blow-ups classified in Proposition~\ref{prop.main1} are $(\kappa+1)$-homogeneous polynomials, and not just higher homogeneous, global solutions to a codimension two obstacle problem.
In particular, we show that
\[
\frac{v_{X_\circ}(r \,\cdot\,)}{r^{\kappa + 1}} \to q_{\ast,X_\circ} \quad\text{locally uniformly as}\quad r \downarrow 0
\]
where $q_{\ast,X_\circ}$ is a non-trivial, $(\kappa+1)$-homogeneous, $a$-harmonic polynomial at all but strictly lower dimensional set of $X_\circ \in \Sigma_\kappa^{m,{\rm g}}(u)$, again, when $\kappa = 2$ and $m \in \{0, \dots, n-1 \}$ as well as when $\kappa > 2$ and $m=n-1$.
(See Section~\ref{sec.thirdorder}.)

\subsection{Notation}
\label{ssec.notation}
We define the balls
\[
\begin{split}
B_r(X_\circ) & := \{ X\in \R^{n+1} : |X - X_\circ| < r \},\\
B_r^*(x_\circ) & := \{ x \in \R^n : |x-x_\circ|< r\},\\
B_r'(x_\circ') & := \{ x' \in \R^{n-1} : |x'-x_\circ'|< r\},
\end{split}
\]
i.e., the balls of radius $r$ centered at $X_\circ$, $x_\circ$, and $x_\circ'$ in $\R^{n+1}$, $\R^n$, and $\R^{n-1}$ respectively. 
We will also denote $B_r := B_r(0)$, $B_r^* := B_r^*(0) $, and $B_r' := B_r'(0)$.
Similarly, we let
\[
D_r \subset \R^2
\]
be the disc of radius $r > 0$, centered at the origin.

For a polynomial $p : \R^{n} \to \R$, consider
\begin{equation}
\label{eq.ext}
\Ext_a(p) := p + \sum_{j=1}^{\infty} c_j\frac{(-1)^j}{(2j)!}y^{2j}\Delta^j_{x}p \quad\text{with}\quad c_j := \prod_{i=1}^j\frac{2i-1}{2i-1-a}.
\end{equation}
Notice that $\Ext_a(p):\R^{n+1}\to\R$ is the unique even in $y$, $a$-harmonic extension of $p$ to $\R^{n+1}$ (see \cite[Lemma 5.2]{GR18}); $L_a (\Ext_a(p)) = 0$.

\subsection{Structure of the Work}

In Section~\ref{sec:monotone}, we introduce a collection of monotonicity formulae (in particular, Almgren's frequency function), and prove some basic but useful estimates.
In Section~\ref{sec:Blow-up Analysis}, we start a blow-up analysis of the solution around singular points. 
We show the existence of second blow-ups and prove some facts about them.
We also show Proposition~\ref{prop.main1} holds. 
In Section~\ref{sec:accumulation}, we gather some important lemmas regarding the accumulation of singular points, which are then used to study the size of the anomalous set in Section~\ref{sec:size}. 
Whence, we prove Proposition~\ref{prop.main2} and Remark~\ref{rem.zerodim}.
In Section~\ref{sec:whitney}, we show that the set of generic points is contained in a countable union of $C^{1,1}$ manifolds, which combined with previous results yields the proof of Theorem~\ref{thm.main1}. 
Finally, we conclude the proofs of our main results in Section~\ref{sec.thirdorder}, Theorems~\ref{thm: nondeg main} and \ref{thm: deg main}, by studying the case of $(\kap+1)$-homogeneous, $a$-harmonic second blow-ups. 
Specifically, we show that those points at which the second-blow up is not the next order term in the expansion are collectively lower-dimensional. 
Finally, Section~\ref{sec.VTOP} is dedicated to studying the very thin obstacle problem.
Here, we prove the estimates and claims on the very thin obstacle problem made use of throughout the work. 
In Section~\ref{sec.GlobPbs}, we make a final remark on global obstacle problems.

\medskip
\noindent{\bf Acknowledgments.}
We are grateful to Alessio Figalli and Joaquim Serra for their guidance and useful discussions.
We would also like to thank Luis Silvestre for sharing  his intuition on the Poisson kernel for the very thin obstacle problem in Section~\ref{sec.VTOP}. 
Finally, we are grateful to Alessio Figalli, Xavier Ros-Oton, and Joaquim Serra for giving us the idea to consider sequences along which the frequency is continuous within our second dimension reduction argument, as they do in their paper \cite{FRS19}, used in Section~\ref{sec:size}.

\section{Monotonicity Formulae and Preliminary Results}
\label{sec:monotone}

We recall that we will always assume that we are dealing with the zero obstacle case \eqref{eq.EL_0}.

Let $X_\circ$ be a singular point for $u$ of order $\kappa_{X_\circ} \in 2\N := \{2,4,6,\dots\}$, and let $p_{\ast,X_\circ}$ be the (unique) first blow-up of $u$ at $X_\circ$,
\begin{equation}
\label{eq.firstblowup}
p_{\ast, X_\circ}(X)  := \lim_{r\downarrow 0} \frac{u(X_\circ + rX)}{r^{\kappa_{X_\circ}}}
\end{equation}
(see \eqref{eq.resck}). 
Recall that $p_{\ast,X_\circ} \in \mathscr{P}_{\kappa_{X_\circ}}$, i.e., it is an $a$-harmonic, $\kappa_{X_\circ}$-homogeneous polynomial, non-negative on the thin space, and even in $y$, and $\kappa_{X_\circ}$ is equal to Almgren's frequency of $u$ at the $X_\circ$:
\[
\kappa_{X_\circ} = N(0^+,u,X_\circ) := \lim_{r \downarrow 0} \frac{r\int_{B_r(X_\circ)} |\nab u|^2|y|^a}{ \int_{\pa B_r(X_\circ)} u^2|y|^a}
\]
(see \cite{ACS08,CSS08,GP09,GR18}).

We often assume that $X_\circ = 0$ (which we can do without loss of generality after a translation), and we let $p_{\ast} := p_{\ast,0}$.
In particular, define 
\[
v_\ast := u - p_{\ast},
\]
and set
\begin{equation}
\label{eq.spine}
\kappa_\ast := \kappa_0, \quad L_\ast := L(p_{\ast}), \quad\text{and}\quad m_\ast := m_{0},
\end{equation}
so that $m_\ast$ is the dimension of the spine of $p_{*}$ in $\{y = 0\}$, $L_\ast$, which is $\kappa_\ast$-homogeneous.

Let, for $\kap\in 2\N$,
\[
p \in \mathscr{P}_\kap \quad\text{and}\quad v = u-p,
\]
and observe that
\begin{equation}
\label{eqn: v Lap v}
v\,L_a v = -p\, L_a u \geq 0.
\end{equation}
Since $L_a u(x,y) = 2\lim_{y \downarrow 0 }y^a\pa_y u(x, y) \mathcal{H}^{n-1}\mres {\Lambda(u)} \le 0$, $v\,L_a v$ is non-negative as soon as $p$ is non-negative on $\Lambda(u) \setminus \mathcal{N}(u)$ where
\begin{equation}
\label{eq.nodal}
\mathcal{N}(u) := \{ (x,0) : u(x,0) = |\nab_{x}u(x,0)| = \lim_{y \downarrow 0 }y^a\pa_y u(x, y) = 0 \}.
\end{equation}
The set $\mathcal{N}(u)$ is called the \emph{nodal set} of $u$.

\begin{remark}
\label{rmk:v soln to TOP}
Notice that $v = u - p$ is a solution to the thin obstacle problem with obstacle $\varphi = -p|_{B_1 \cap \{ y = 0\}}$ and subject to its own boundary data.
(This follows easily by Remark~\ref{rmk: min super}.)
\end{remark}

The goal of this section is to prove monotonicity-type results and estimates for $v = u - p$ for any $p \in \mathscr{P}_\kap$.
We stress that $\kappa$ might not be equal to $\kappa_\ast$, and so we will sometimes write $N(0^+,u) := N(0^+,u,0)$ instead.
Yet we will most often apply these results and estimates to $v_\ast$. 

\subsection{Monotonicity Formulae}

To begin we study Almgren's frequency function on $v$ at the origin, and prove that it is non-decreasing provided that $\kappa \leq \kappa_\ast = N(0^+,u)$.

\begin{proposition}
\label{prop: Almgren} 
Suppose that $\kappa \leq N(0^+,u)$, and let $v = u - p$ for $p \in \mathscr{P}_\kappa$.
Then, Almgren's frequency function on $v$
\[
r \mapsto N(r,v) = \frac{r\int_{B_r} |\nab v|^2|y|^a}{ \int_{\pa B_r} v^2|y|^a}
\]
is non-decreasing.
Moreover,
$N(0^+,v) \geq \kappa$.
\end{proposition}

Before proceeding with the proof of Proposition~\ref{prop: Almgren}, let us recall a few definitions and facts.
Let $W_\lambda(r, u)$ denote the {\it $\lambda$-Weiss energy of $u$ at $r$}: 
\begin{equation}
\label{eq.Weissenergy}
W_\lambda (r, u ) := \frac{1}{r^{2\lambda}} D(r, u)  - \frac{\lambda}{r^{2\lambda}} H(r, u)
\end{equation}
where
\begin{equation}
\label{eq.D}
D(r, u) := \frac{1}{r^{n+a-1}}\int_{B_r} |\nab u|^2|y|^a = r^{2} \int_{B_1} |\nab u(rX)|^2|y|^a
\end{equation}
and
\begin{equation}
\label{eq.H}
H(r, u) := \frac{1}{r^{n+a}} \int_{\pa B_r} u^2|y|^a =  \int_{\pa B_1} u(rX)^2|y |^a.
\end{equation}
By \cite[Theorem 2.11]{GR18}, we have that $N(r, u)$ is non-decreasing, from which, we immediately deduce that 
\begin{equation}
\label{eq.Weisspos}
W_\kap(r, u) = \frac{H(r, u)}{r^{2\kap}} (N(r, u) -\kap) \ge 0
\end{equation}
(recall $N(0^+, u) \geq \kap$).
In turn, we have the following lemma: 
 
\begin{lemma}
\label{lem:key ineqs}
Suppose that $\kappa \leq N(0^+,u)$, and let $v = u - p$ for $p \in \mathscr{P}_\kappa$. 
Then,
\begin{equation}
\label{eqn: key ineq 1}
\frac{1}{r^{n-1+a+2\kap}}\int_{B_r} |\nab v|^2|y|^a \geq \frac{\kap}{r^{n+a+2\kap}}\int_{\pa B_r} v^2|y|^a 
\end{equation}
and
\begin{equation}
\label{eqn: key ineq 2}
\frac{1}{r^{n+a+2\kap}}\int_{\pa B_r} v(X \cdot \nab v - \kap v)|y|^a \geq \frac{1}{r^{n-1+a+2\kap}}\int_{B_r} v \,L_a v.
\end{equation}
\end{lemma}

\begin{proof}
We proceed as in the proof of \cite[Theorem 1.4.3]{GP09}. By \cite[Theorem~2.11]{GR18}, $N(r, p ) \equiv \kap$, from which it follows that $W_\kappa(r, p) \equiv 0$. Using \eqref{eq.Weisspos} and integrating by parts, we immediately have that 
\[
\begin{split}
0 & \le W_\kap (r, u) - W_\kap (r, p) \\
& = \frac{1}{r^{n-1+a+2\kap}}\int_{B_r} \left(|\nabla v|^2+2\nabla v \cdot \nabla p \right)|y |^a - \frac{\kap}{r^{n+a+2\kap}}\int_{\pa B_r}\left(v^2+2vp \right)|y|^a\\
& = \frac{1}{r^{n-1+a+2\kap}}\int_{B_r}|\nabla v|^2|y|^a 
- \frac{\kap}{r^{n+a+2\kap}}\int_{\pa B_r}v^2|y|^a
+ \frac{2}{r^{n+a+2\kap}} \int_{\pa B_r} v(X\cdot \nabla p - \kap p )|y|^a
\\
& = \frac{1}{r^{n-1+a+2\kap}}\int_{B_r}|\nabla v|^2|y|^a 
- \frac{\kap}{r^{n+a+2\kap}}\int_{\pa B_r}v^2|y|^a,
\end{split}
\]
which directly yields \eqref{eqn: key ineq 1}. 
Continuing, integrating by parts again, we get 
\begin{align*}
\frac{1}{r^{n-1+a+2\kap}}& \int_{B_r}|\nabla v|^2|y|^a 
- \frac{\kap}{r^{n+a+2\kap}}\int_{\pa B_r}v^2|y|^a\\
& = -\frac{1}{r^{n-1+a+2\kap}}\int_{B_r}v\,L_av + \frac{1}{r^{n+a+2\kap}}\int_{\pa B_r}v(X\cdot \nabla v - \kap v)|y|^a,
\end{align*}
which implies \eqref{eqn: key ineq 2}. 
\end{proof}

With Lemma~\ref{lem:key ineqs} in hand, we can now prove Proposition~\ref{prop: Almgren}.

\begin{proof}[Proof of Proposition~\ref{prop: Almgren}]
Notice that
\[
N(r, v) = \frac{D(r, v)}{H(r, v)},
\]
where $D$ and $H$ are given by \eqref{eq.D} and \eqref{eq.H}. By scaling (namely, $N(\rho, u_r) = N(r\rho, u)$, for the rescaling \eqref{eq.rescGR}), it is enough to show $N'(1, v) \ge 0$ or, equivalently, that
\begin{equation}
\label{eqn: Alm mono ineq}
D'(1)H(1) - H'(1)D(1) \geq 0,
\end{equation}
where we have let $D(1) = D(1, v)$ and $H(1) = H(1, v)$. 

We compute $D'(1)$ and $H'(1)$.
First,
\[
\begin{split}
D'(1) &= 2\int_{B_1} |\nab v|^2|y|^a + 2\int_{B_1} \nab v \cdot D^2 v \cdot X\,|y|^a\\
&= 2\int_{B_1} \nab v \cdot \nab(X \cdot \nab v) |y|^a\\
&= 2\int_{\pa B_1} v^2_\nu |y|^a - 2\int_{B_1} L_a u\,(X \cdot \nab u) + 2\int_{B_1} L_a u\,(X \cdot \nab p ),
\end{split}
\]
using integration by parts and that $p$ is $a$-harmonic. 
Now notice that, by the regularity of the solution, $L_a u\, (X\cdot \nabla u) \equiv 0$. 
This, together with the fact that $p$ is $\kap$-homogeneous, yields
\[
D'(1)= 2\int_{\pa B_1} v^2_\nu |y|^a + 2\kap\int_{B_1} p\, L_a u = 2\int_{\pa B_1} v^2_\nu|y|^a - 2\kap\int_{B_1} v\,L_a v,
\]
where the last inequality follows by \eqref{eqn: v Lap v}.
On the other hand,
\[
H'(1) = 2\int_{\pa B_1} v v_\nu |y|^a.
\]

Now letting
\[
I := \int_{B_1} v\,L_a v
\]
and using
\[
\int_{B_1} |\nab  v|^2 |y|^a = \int_{\pa B_1} v v_\nu |y|^a - I,
\]
in addition to the Cauchy--Schwarz inequality, we find that
\[
\begin{split}
D'(1) & H(1) -  H'(1)D(1) 
\\& = \bigg(2\int_{\pa B_1} v^2_\nu|y|^a - 2\kap I\bigg)\int_{\pa B_1} v^2|y|^a - 2\int_{\pa B_1} v v_\nu|y|^a\bigg(\int_{\pa B_1} v v_\nu|y|^a - I\bigg)\\
&=2\bigg(\int_{\pa B_1} v^2_\nu|y|^a\int_{\pa B_1} v^2|y|^a - \kap I\int_{\pa B_1} v^2|y|^a - \bigg(\int_{\pa B_1} v v_\nu|y|^a\bigg)^2 + I\int_{\pa B_1} v v_\nu|y|^a \bigg)\\
&\geq - 2\kap I\int_{\pa B_1} v^2|y|^a + 2I\int_{\pa B_1} v v_\nu|y|^a\\
&= 2I\int_{\pa B_1} v(X \cdot \nab v - \kap v)|y|^a. 
\end{split}
\]
Hence, by \eqref{eqn: v Lap v} and \eqref{eqn: key ineq 2}, we deduce that \eqref{eqn: Alm mono ineq} holds, as desired.
\end{proof}

We end the subsection with a lemma on a Monneau-type monotonicity statement and Weiss-type monotonicity statement, arguing as in \cite[Lemma 2.6 and 2.8]{FS18}, and a important Monneau-type limit.

\begin{lemma}
\label{lem.HMon}
Suppose that $\kappa \leq N(0^+,u)$, and let $v = u-p$ for $p \in \mathscr{P}_{\kappa}$.
Given $\lambda > 0$, define
\begin{equation}
\label{eq.HMon}
H_\lambda(r, v) := \frac{1}{r^{n+a+2\lambda}}\int_{\pa B_r} v^2|y|^a = \frac{1}{r^{2\lambda}} H(r, v).
\end{equation}
Then, $r\mapsto H_\lambda(r, v)$ is non-decreasing for all $0\le \lambda\le N(0^+, v)$. 
Moreover, the $\lambda$-Weiss energy 
\[
r \mapsto W_\lambda(r , v)
\] 
on $v$ is also non-decreasing for all $\lambda > 0$.
\end{lemma}

\begin{proof}
Let $v_r(X) := (u-p)(rX)$; then,
\[
\frac{H_\lambda'}{H_\lambda}(r, v) = \frac{2r\int_{\pa B_1} v_r(X)(X\cdot \nabla v(rX))|y |^a-2\lambda \int_{\pa B_1} v_r^2|y|^a}{r\int_{\pa B_1} v_r^2|y|^a}.
\]
Notice also that
\[
r\int_{\pa B_1} v_r(X)(X\cdot \nabla v(rX))|y |^a = \int_{\pa B_1} v_r(X\cdot \nabla v_r)|y |^a = \int_{B_1} |\nabla v_r|^2|y |^a +\int_{B_1} v_r\,L_a v_r,
\]
and $v_r\,L_av_r\ge 0$ (see \eqref{eqn:  v Lap v}).
Hence, since $N(1, v_r) = N(r, v)$, 
\begin{equation}
\label{eq.HMon2}
\frac{H_\lambda'}{H_\lambda}(r, v) \ge \frac{2}{r}\left(N(r, v) -\lambda\right).
\end{equation}
Now using that $N(r, v) \ge N(0^+, v) \ge \lambda$, we reach the desired result, \eqref{eq.HMon}.

To see the monotonicity of $W_\lambda(r, v)$ for $0\le \lambda\le N(0^+, v)$, we simply combine the expressions \eqref{eq.Weisspos} and \eqref{eq.HMon}, so that $W_\lambda(r, v)$ is product of two non-decreasing non-negative functions.

On the other hand, if $\lambda > N(0^+, v)$, a simple manipulation (see the proof of Proposition~\ref{prop:  Almgren}) yields
\[
\begin{split}
W_\lambda'(1) &= D'(1) -\lambda H'(1) -2\lambda(D(1) - \lambda H(1)) \\
&= 2\int_{\pa B_1}(v_\nu-\lambda v)^2|y|^a + 2(\lambda - \kap)\int_{B_1} v\, L_a v.
\end{split}
\]
As $v \, L_a v \ge 0$ and $\lambda > N(0^+, v) \geq \kappa$ (by Proposition~\ref{prop: Almgren}), we conclude. 
\end{proof}

Notice also that if we set 
\[
\lambda_\ast := N(0^+,v_\ast) \geq \kappa_\ast = N(0^+,u),
\]
then
\[
\lim_{r \downarrow 0} H_\l(r,v_\ast) = \infty \quad\text{for all}\quad \lambda > \lambda_\ast,
\]
which follows arguing exactly as in \cite[Corollary 2.9]{FS18}.

\subsection{Estimates}

Let us define, for any function $f$, the positive and negative parts as
\[
f^{+} := \max\{f,0\} \quad\text{and}\quad f^{-} := \max\{-f,0\} = - \min\{f,0\}.
\]
Hence, $f = f^+ - f^-$.

We start with an $L^\infty$--$L^2$ estimate on $v$.

\begin{lemma}
\label{lem:L2Linfty}
Let $v = u-p$ for $p \in \mathscr{P}_\kap$.
Then,
\begin{equation}
\label{eq.solisbounded00}
\|v\|_{L^\infty(B_{1/2})} \leq C\|v\|_{L^2(B_1,|y|^a)},
\end{equation}
for some constant $C$ depending only on $n$ and $a$. 
\end{lemma}

\begin{proof}
Observe that $v^-$ is sub $a$-harmonic in $B_1$ as the maximum of two sub $a$-harmonic functions in $B_1$.

Let us show that $v^+$ is also sub $a$-harmonic in $B_1$.
To this end, first, by Remark~\ref{rmk:v soln to TOP}, recall that $v$ is the solution to \eqref{eqn: Signorini problem} with $\varphi = -p|_{B_1 \cap \{ y = 0\}}$ and its own boundary data.
Now let $\eta$ be any smooth compactly supported function in $B_1$ such $0 \le \eta\le 1$. 
In addition, let $h_\delta$ be an approximation of the Heaviside function: $h_\delta(t) = 0$ for $t\le 0$, $h_\delta(t)= t/\delta$ for $t \in(0, \delta)$, and $h_\delta(t) = 1 $ for $t \ge \delta$.
Finally, for $0<\varepsilon<\delta$, define $v_\varepsilon := v -\varepsilon\eta h_\delta(v)$.

Since $p(x, 0) \ge 0$, observe that $v_\varepsilon(x,0) \geq -p(x,0)$ and $v_\varepsilon|_{\pa B_1} = v|_{\pa B_1}$.
Therefore, 
\[
\int_{B_1} |\nabla v-\varepsilon \nabla(\eta h_\delta(v))|^2|y|^a \ge \int_{B_1} |\nabla v|^2|y|^a,
\]
which implies that, after dividing through by $\vep$ and letting $\vep\downarrow 0$,
\[
\int_{B_1} \nabla v \cdot \nabla(\eta h_\delta(v))|y|^a \leq 0.
\]
Expanding,
\[
\int_{B_1} h_\delta(v) \nabla v \cdot \nabla \eta |y|^a \leq  -\int_{B_1} \eta  |\nabla v|^2 h'_\delta(v) |y|^a \le 0
.
\]
In turn, if $H_\delta' = h_\delta$ with $H_\delta(0) = 0$, then
\[
\int_{B_1} \nabla (H_\delta(v)) \cdot \nabla \eta  
|y|^a \leq 0.
\]
(Obviously, $H_\delta$ here is not the Monneau-type function from Lemma~\ref{lem.HMon}.)
Because $\eta$ was arbitrary, we find that $H_\delta (v)$ is sub $a$-harmonic in $B_1$. 
So letting $\delta\downarrow 0$, we determine that $v^+$ is sub $a$-harmonic in $B_1$ ($H_\delta (v)$ is an approximation of $v^+$).

To conclude, see that by the local boundedness of subsolutions for $L_a$ (see, e.g., \cite[Proposition 2.1]{JN17}), we have that
\[
\sup_{B_{1/2}} v^\pm \le C\left(\int_{B_{1}}  |v^{\pm}|^2|y|^a\right)^{1/2},
\]
and \eqref{eq.solisbounded00} holds. 
\end{proof}

Next, we prove Lipschitz and semiconvexity estimates on $v$ along the spine of $p$.
But before doing so, we prove a characterization lemma on the spine of a generic $\kap$-homogeneous polynomial. 

\begin{lemma}
\label{lem: equiv}
Let $\kap \in \N$, and let $p: \R^n\to \R$ be a $\kap$-homogeneous polynomial.
Then, the following sets are equal. 
\begin{enumerate}[(i)]
\item $L(p) := \{\xi\in \R^{n} : \xi\cdot\nabla p(x) = 0 \textrm{ for all }x\in \R^{n}\}$.
\item $I(p) := \{\xi\in \R^{n} : p(x+\xi) = p(x) \textrm{ for all }x\in \R^{n}\}$. 
\item $D_{\kap -1}(p) := \{ \xi\in \R^{n} : D^{\alpha} p(\xi) = 0 \textrm{ for all }\alpha = (\alpha_1,\dots,\alpha_{n}) : |\alpha| = \kap -1 \}$.
\end{enumerate}
\end{lemma}
\begin{proof}
We prove that (i) and (ii) as well as (ii) and (iii) are equivalent.

\smallskip
\noindent -- $L(p) \subset I(p)$: Let $\xi\in L(p)$. Then, 
\[
p(x+\xi) = p(x) + \int_0^1 \xi \cdot \nabla p(x+t\xi)\, \d t = p(x). 
\]

\noindent -- $I(p) \subset L(p)$: We start by noticing that $I(p)$ is actually a linear space, thanks to the homogeneity of $p$. 
Indeed, the additive property is clear; it is also clear that $-\xi\in I(p)$ if $\xi\in I(p)$. 
Now suppose $\xi \in I(p)$ and consider $\beta\xi$ for some $\beta > 0$.
Then, $p(x+\beta\xi) = \beta^{\kap}p(\beta^{-1}x + \xi) = \beta^\kap p(\beta^{-1}x) = p(x)$ for all $x\in \R^{n}$, so that $\beta\xi\in I(p)$.

Let $\xi\in I(p)$. 
Now for all $h > 0$ and for all $x\in \R^{n}$, $p(x+h\xi) = p(x)$.
Hence,
\[
\xi\cdot \nabla_{x} p(x) = \lim_{h\downarrow 0} \frac{p(x+h\xi, 0) - p(x)}{h} = 0,
\]
that is, $\xi \in L(p)$.

\smallskip
\noindent -- $I(p) \subset D_{\kap-1}(p)$: Let $\xi\in I(p)$. 
Then, $p(\xi+x) = p(x)$ and $D^\alpha p(x+\xi) = D^\alpha p(x)$ for any $\alpha = (\alpha_1,\dots,\alpha_{n-1})$ with $|\alpha| = \kap - 1$. 
Taking $x = 0$, we conclude thanks to the $\kap$-homogeneity of $p$.

\noindent --  $D_{\kap-1}(p)\subset I(p)$.
Let $\xi\in D_{\kap-1}(p)$. 
Consider the degree $\kap$ polynomial $q(x) := p(x+\xi)$. 
Notice that from the definition of $D_{\kap-1}$, $q$ is homogeneous. 
Now let $\beta > 0$. 
Using the homogeneity of $q$ and $p$, 
\[
p(x+\xi) = q(x) = \beta^\kap q(\beta^{-1}x)= \beta^\kap p(\beta^{-1}x + \xi) = p(x + \beta\xi)
\] 
for all $\beta > 0$. 
Taking $\beta \downarrow 0$, we see that $\xi\in I(p)$.

This concludes the proof.
\end{proof}

Notice that the equivalence of (i) and (ii) also holds for general $\kap$-homogeneous functions.

\begin{remark}
Lemma~\ref{lem: equiv} will be applied to $p(x, 0)$ for $p \in \mathscr{P}_\kap$.
\end{remark}

The following lemma shows that derivatives of $v$ along the invariant set of $p$ are bounded. 
Recall that $L(p)$ denotes the invariant set of $p(x, 0)$. 
The lemma is proved by means of a Bernstein's technique for integro-differential equations, as introduced by Cabr\'e, Dipierro, and Valdinoci, in \cite{CDV19}. 

\begin{lemma}
\label{lem: Bern}
Let $v = u-p$ for $p \in \mathscr{P}_\kap$.
Then, for all $\be \in L(p) \cap \mathbb{S}^n$,
\[
\|\pa_\be v\|_{L^\infty(B_{1/2})} \leq C\|v\|_{L^2(B_1,|y|^a)},  
\] 
for some constant $C$ depending only on $n$ and $a$. 
\end{lemma}

\begin{proof}
We proceed by Bernstein's technique (see \cite{CDV19}). 
Let $\eta \in C^\infty_c(B_{1/2})$ be even in $y$ and such that $\eta \equiv 1$ in $B_{1/4}$. 
Consider the function,
\[
\psi := \eta^2(\pa_{\be} v)^2 + \mu v^2,
\]
for some $\mu > 0$ to be chosen.

Since $v$ is $a$-harmonic outside $\Lambda(u)$, in $B_{1/2} \setminus \Lambda(u)$,
\[
L_a (v^2) = 2v\, L_a v + 2 |\nabla v|^2|y|^{a} \geq 2 |\nabla v|^2|y|^a.
\]
Similarly, because $\pa_\be v$ is $a$-harmonic outside $\Lambda(u)$, we have that in $B_{1/2} \setminus \Lambda(u)$, $L_a \pa_{\be} v = L_a \pa_{\be} u = 0$.  
Therefore, we find that in $B_{1/2} \setminus \Lambda(u)$,
\[
\begin{split}
L_a (\eta^2 (\pa_{\be} v)^2 ) & = (\pa_{\be} v)^2 L_a (\eta^2) + \eta^2 L_a ((\pa_{\be} v)^2) + 2|y|^a\nabla (\pa_{\be} v)^2 \cdot \nabla \eta^2\\
&= (\pa_{\be} v)^2 L_a (\eta^2) + 2\eta^2|\nabla \pa_{\be} v|^2|y|^{a} + 2|y|^a\nabla (\pa_{\be} v)^2 \cdot \nabla \eta^2\\
&\geq (\pa_{\be} v)^2 L_a (\eta^2) + 2\eta^2|\nabla (\pa_{\be} v)|^2|y|^a - 8|y|^a|\nabla \pa_{\be} v||\pa_{\be} v||\nabla \eta| \eta\\
& \geq |y|^a|\pa_{\be} v|^2 (|y|^{-a}L_a (\eta^2)  - 8|\nabla \eta|^2)
\end{split}
\]
where there last inequality follows from 
\[
\eta^2|\nabla \pa_{\be} v|^2+4|\pa_{\be} v|^2|\nabla \eta|^2 \ge 4|\pa_{\be} v| |\nabla \pa_{\be} v|\eta|\nabla \eta|.
\]
So in $B_{1/2} \setminus \Lambda(u)$,
\[
\begin{split}
L_a \psi &\geq |y|^a|\pa_{\be} v|^2 (|y|^{-a}L_a (\eta^2)  - 8|\nabla \eta|^2) + |y|^a|\nab v|^22\mu \\
&\geq |y|^a|\nab v|^2(2 \mu - |y|^{-a}|L_a (\eta^2)| - 8|\nabla \eta|^2).
\end{split}
\]
Now as $\eta$ is even in $y$ and smooth, $|y|^{-a}|L_a (\eta^2)| + 8|\nabla \eta|^2 \leq C_\eta$ in $B_{1/2}$, from which we deduce that
\[
L_a \psi \geq 0 \quad\text{in}\quad B_{1/2} \setminus \Lambda(u)
\]
provided $2\mu \geq C_\eta$.
 
By the maximum principle then, $\psi$ must attain its maximum at the boundary of $B_{1/2} \setminus \Lambda(u)$.
Being that $\pa_\be p = \pa_\be u = 0$ on $\Lambda(u)$ and $\eta|_{\pa B_{1/2}} = 0$, $\psi = \mu v^2$ on $\pa B_{1/2} \cup \Lambda(u)$.
Hence,
\[
\sup_{B_{1/2}} \psi \leq \mu \sup_{B_{1/2}} v^2.
\]
In particular, as $\eta \equiv 1$ on $B_{1/4}$,
\[
\|\pa_{\be} v\|_{L^\infty(B_{1/4})} \le \mu^{1/2}\|v\|_{L^\infty(B_{1/2})}.
\]
Thus, by Lemma~\ref{lem:L2Linfty} and a covering argument, we find the desired estimate. 
\end{proof}

Finally, we show that $v$ is semiconvex along the spine of $p$. 
Naturally, for $h > 0$, let
\[
\delta_{\be,h}^2f := \frac{f(\,\cdot\,+h\be) + f(\,\cdot\,-h\be) - 2f}{h^2}
\]
be the second order $h$-incremental quotient of the function $f$ in the direction $\be \in \mathbb{S}^{n}$. 

\begin{lemma}
\label{lem: Bern X2}
Let $v = u-p$ for $p \in \mathscr{P}_\kap$.
Then, for all $\be \in L(p) \cap \mathbb{S}^n$,
\[
\inf_{B_{1/2}} \pa_{\be\be}v \ge -C\|v\|_{L^2(B_1, |y|^a)},
\]
for some constant $C$ depending only on $n$ and $a$. 
\end{lemma}

\begin{proof}
For any $\gamma> 0$, let $u_\gamma$ be the solution to
\begin{equation}
\label{eq.EL_0_g}
\left\{ 
\begin{array}{rcll}
{u}_\gamma(x,y) &\geq& 0 &\text{on } B_{7/8} \cap \{ y = 0\}\\
L_a  u_\gamma(x,y) &\leq& 0 &\text{in } B_{7/8}\\
L_a u_\gamma (x,y) &=& 0 &\text{in } B_{7/8} \setminus \Lambda({u}_\gamma) \\
{u}_\gamma(x,y) &=& u(x,y)+\gamma &\text{on } \pa B_{7/8}.
\end{array}\right.
\end{equation}
That is, in $B_{7/8}$, $u_\gamma$ is the solution to the thin obstacle problem with zero obstacle and boundary data $u + \gamma$. 
Notice that since $u$ is continuous in $B_1$, we have that $u_\gamma \downarrow u$ uniformly in $B_{7/8}$, as $\gamma\downarrow 0$.
Also, $u_\gamma > 0$ in $B_{7/8} \setminus B_{7/8-\beta}$ for some $\beta = \beta(\gamma) > 0$, by the continuity of $u_\gamma$. 
In particular, $u_\gamma$ is $a$-harmonic in the annulus $B_{7/8}\setminus B_{7/8-\beta}$. 

Consider the function 
\[
f_\gamma(x) := (\pa_{\be\be} u_\gamma(x))^-
\]
as the pointwise limit of $(\delta^2_{\be, h} u(x))^-$ as $h\downarrow 0$. 
To do so, we define 
\[
g^\gamma_{\vep, h, \be}(x) := \min\{\delta^2_{\be, h} u_\gamma(x), -\vep\}.
\]
Observe that $L_a(\delta^2_{\be, h} u_\gamma) \le 0$ in $B_{7/8}\setminus \Lambda(u_\gamma)$ (since $L_a u_\gamma \le 0$ in $B_{7/8}$ and $L_a u_\gamma = 0$ in $B_{7/8}\setminus \Lambda(u_\gamma)$). 
Moreover, since $u_\gamma$ is continuous and $\delta^2_{\be, h} u_\gamma \ge 0$ on $\Lambda(u_\gamma)$, we have $g^\gamma_{\vep, h, \be} = -\vep$ in a neighbourhood of $\Lambda(u_\gamma)$. 
Thus, $L_a g^\gamma_{\vep, h, \be} \le 0$ in $B_{7/8}$. 

We now want to let $\vep \downarrow 0$ and then $h\downarrow 0$ to deduce that $L_a f_\gamma\ge 0$ in $B_{3/4}$ and $f_\gamma\equiv 0$ on $\Lambda(u_\gamma)$. 
In order to pass $L_a g^\gamma_{\vep, h, \be} \le 0$ to the limit (as $\vep, h\downarrow 0$), it is enough to show that $|g^\gamma_{\vep, h, \be}|\le C$ for some $C$ independent of $\vep$ and $h$ (but possibly depending on $\gamma$). 
As $g^\gamma_{\vep, h, \be}$ is super-$a$-harmonic in $B_{7/8}$, its minimum must be achieved on the boundary. 
In particular, since $g^\gamma_{\vep, h, \be}\le 0$,
\[
\sup_{B_{3/4}} |g^\gamma_{\vep, h, \be}|\le \sup_{\pa B_{7/8-\beta/2}} |g^\gamma_{\vep, h, \be}| \le C(\beta), 
\]
where in the last inequality, we have used that $g^\gamma_{\vep, h, \be}$ is $a$-harmonic in $B_{7/8}\setminus B_{7/8-\beta}$ and corresponding $C^2$ estimates in the tangential direction for $a$-harmonic functions.
Hence, we can indeed pass $L_a g^\gamma_{\vep, h, \be} \le 0$ in $B_{3/4}$ to the limit and obtain that $L_a f_\gamma \ge 0$ in $B_{3/4}$ and $f_\gamma \equiv 0$ on $\Lambda(u_\gamma)$. 

With the sub-$a$-harmonicity and nonnegativity of $f_\gamma$ in hand, it is easy to see that $f_\gamma$ is continuous in $B_{3/4}$.
Indeed, sub-$a$-harmonic functions are upper semi-continuous (see \cite[Theorem 3.63]{HKM93}).
So being that $f_\gamma$ is continuous when $f_\gamma > 0$ and $f_\gamma$ is nonnegative in general, we determine the continuity of $f_\gamma$, as desired.
  
To conclude, we again proceed by Bernstein's technique (see \cite{CDV19}). 
Let $\eta \in C^\infty_c(B_{1/2})$ be even in $y$ and such that $\eta \equiv 1$ in $B_{1/4}$, and set
\[
\psi_\gamma := \eta^2 f_\gamma^2 + \mu (\pa_{\be} u_\gamma)^2,
\]
where we recall $f_\gamma(x) := (\pa_{\be\be} u_\gamma(x))^-$.
By the discussion above, $\psi_\gamma$ is continuous in $B_{1/2}$. 
Recall that $f_\gamma\equiv 0$ on $\Lambda(u_\gamma)$, and therefore, $\psi_\gamma \equiv 0$ on $\Lambda(u_\gamma)$. 
On the other hand, on the boundary of $B_{1/4}$, we have that $\psi_\gamma = \mu(\pa_\be u_\gamma)^2$.
Following the proof of Lemma~\ref{lem: Bern} exactly (and using that $L_a f_\gamma \ge 0$ in $B_{3/4}$), we see that $L_a \psi_\gamma \ge 0$ in $B_{1/2}\setminus \Lambda(u_\gamma)$ if $\mu$ is large enough, and so, its maximum must be achieved at the boundary.
In turn,
\[
\|f_\gamma\|_{L^\infty(B_{1/4})} \leq \mu^{1/2}\|\pa_\be u_\gamma\|_{L^\infty(B_{1/2})} = \mu^{1/2}\|\pa_\be (u_\gamma - p)\|_{L^\infty(B_{1/2})} \leq C\|u_\gamma - p\|_{L^2(B_1,|y|^a)},
\]
where we have used Lemma~\ref{lem: Bern} in the last inequality. 
This implies the family $\{ u_\gamma \}$, for $0 <\gamma \leq 1$, is uniformly semiconvex. 
Letting $\gamma \downarrow 0$ then and applying a covering argument, we deduce the desired result (using that semiconvexity passes to the limit). 
\end{proof}

\begin{remark}
Notice that $p$'s polynomial nature plays no role in Lemmas~\ref{lem:L2Linfty}, \ref{lem:  Bern}, and \ref{lem:  Bern X2}.
We have only used that $p$ is non-negative in the thin space and $a$-harmonic in Lemma~\ref{lem:L2Linfty}, and that $p$ is non-negative in the thin space, $a$-harmonic, and invariant in the $\be$ directions in Lemmas~\ref{lem:  Bern} and \ref{lem:  Bern X2}.
\end{remark}

\section{Blow-up Analysis}
\label{sec:Blow-up Analysis}

Recall, after a translation, we may assume that $0 \in {\rm Sing}(u)$ represents any singular point.
And, as such, the first blow-up of $u$ at $0$ is an element of $\mathscr{P}_\kappa$ for some $\kappa \in 2\N$.
As in Section~\ref{sec:monotone}, we let $p_\ast$ denote the first blow-up of $u$ at $0$, and define 
\[
v_\ast := u - p_\ast, \quad\kappa_\ast := \kappa_0, \quad L_\ast := L(p_\ast), \quad m_\ast := m_0, \quad\text{and}\quad \lambda_\ast := N(0^+,v_\ast).
\]

For notational simplicity, from this point forward, we often suppress the star subscript when denoting the homogeneity of $p_\ast$, and simply write $\kappa$ instead of $\kappa_\ast$.

In this section, we are interested in classifying the {\it second blow-ups of $u$ at $0$}, that is, the limit points of the set $\{ \tilde{v}_{r} \}_{r > 0}$, which is weakly precompact by Proposition~\ref{prop: Almgren}, as $r \downarrow 0$, with
\begin{equation}
\label{eq.vtilde2}
\tilde{v}_{r} := \frac{v_{r}}{\|v_{r}\|_{L^2(\pa B_1, |y|^a)}} \quad\text{and}\quad v_{r}(X) := u(rX) - p_\ast(rX).
\end{equation}
In turn, we will prove Proposition~\ref{prop.main1}.
 
We will work according to two cases, determined by the value of $a$ and the alignment of $L_\ast$ and the nodal set of $p_\ast$,
\[
\mathcal{N}_\ast := \mathcal{N}(p_\ast)
\]
(see \eqref{eq.nodal}).
Notice that by Lemma~\ref{lem: equiv}, if we consider $L(p)$ as a subset of $\R^n \times \{ 0 \}$, then
\[
L(p) \subset \mathcal{N}(p)
\]
for all $p \in \mathscr{P}_\kappa$; yet $L(p)$ may be smaller than $\mathcal{N}(p)$.
In particular, we define \ref{eq.Case1} and \ref{eq.Case2} as follows. 
\begin{align}
\label{eq.Case1}
\begin{array}{rl}
\text{Either}&\quad a \in [0, 1) \\
\quad\text{or}&\quad a \in (-1,0) \text{ and }\dim_{\mathcal{H}} \mathcal{N}_\ast \le n-2
\end{array}
\tag{Case 1}
\end{align}
\[
\text{and}
\]
\begin{align}
\label{eq.Case2}
a \in (-1,0) \text{ and }\dim_{\mathcal{H}} \mathcal{N}_\ast = \dim_{\mathcal{H}} L_\ast = n-1.
\tag{Case 2}
\end{align}

\begin{remark}
\label{rmk: all poss}
We remark that \ref{eq.Case1} and \ref{eq.Case2}, a priori, do not cover all possibilities. 
Indeed, the case when $a \in (-1,0)$ and $\dim_{\mathcal{H}} L_* < \dim_{\mathcal{H}} \mathcal{N}_\ast = n - 1$ is missing. 
In fact, it is currently unknown if such a situation can occur when $u \not\equiv p_{*}$.
\end{remark}

Before we proceed with our classification results, we make a pair of observations, the second of which will play a key feature in \ref{eq.Case2}.
Since $p_\ast \geq 0$ on $\R^{n+1} \cap \{ y = 0\}$, we have that
\begin{equation}
\label{eqn:nodal is zero level}
\{ (x,0) : p_\ast(x,0) = 0 \} = \{ (x,0) : p_\ast(x,0) = |\nab_x p_\ast(x,0)| = 0 \} = \mathcal{N}_\ast.
\end{equation}
Furthermore, if $L_\ast \cong \R^{n-1}$, as it is in \ref{eq.Case2}, then $p_\ast|_{\R^n \times \{ 0\}}$ is a one-dimensional polynomial, and so we can identify $L_\ast$ and $\mathcal{N}_\ast$ as the same subset of $\R^n \times \{ 0 \}$.

Let us start by studying second blows-up in \ref{eq.Case1}.

\begin{proposition}
\label{prop.case1}
In \ref{eq.Case1}, for every sequence $r_j\downarrow 0$, there is a subsequence $r_{j_\ell}\downarrow 0$ such that $\tilde v_{{r_j}_\ell}\rightharpoonup q$ weakly in $W^{1, 2}(B_1, |y|^a)$ as $\ell \to \infty$, and $q\not\equiv 0$ is a $\lambda_*$-homogeneous, $a$-harmonic polynomial. 
In particular, $\lambda_*\in \{\kap,\kap+1,\kap+2,\dots\}$.
\end{proposition}

\begin{proof}
By Proposition~\ref{prop: Almgren}, we see that given any sequence $r_j \downarrow 0$, the sequence $\tilde v_{r_j}$ is uniformly bounded in $W^{1,2}(B_1, |y|^a)$.
Hence, there is a subsequence $r_{j_\ell}\downarrow 0$ such that
\[
\tilde v_{r_{j_\ell}} \rightharpoonup q \quad\text{in}\quad W^{1,2}(B_1, |y|^a),
\]
for some $q$, and as $\|\tilde v_{r_{j_\ell}}\|_{L^2(\pa B_1, |y|^a)} = 1$, we have that
\[
\|q\|_{L^2(\pa B_1, |y|^a)} = 1.
\]

Observe that $L_a\tilde v_{r}$ is a non-positive measure as
\[
L_a v_{r} = 2r \lim_{y \downarrow 0 } y^a\pa_y u_r \mathcal{H}^{n}\mres {\Lambda(u_r)}\le  0
\]
in the sense of distributions.
Furthermore, let $K \subset B_1$ be a any compact set and $\eta_K \in C_c^\infty(B_1)$ be such that $\eta_K \equiv 1$ on $K$ and $0 \leq \eta_K \leq 1$ in $B_1$.
By H\"older's inequality,
\[
0 \leq \int_{K} -L_a \tilde{v}_{r} \leq \int_{B_1} -\eta_K L_a \tilde{v}_{r} = \int_{B_1}  \nabla \eta_K \cdot \nabla \tilde{v}_{r}|y|^a
\le C_K \| \nab \tilde{v}_{r}\|_{L^{2}(B_1,|y|^a)}
\]
Since the family $\tilde{v}_{r}$ is uniformly bounded in $W^{1,2}(B_1, |y|^a)$ by Proposition~\ref{prop: Almgren}, it follows that the collection of measures $L_a \tilde{v}_{r}$ is tight.
So, up to a further subsequence, which we still denote by $r_{j_\ell}$, we have that $L_a q$ is a non-positive measure.
Then, as $r^{-\kap}u_r \to p_\ast$ locally uniformly, with $u_r(X) := u(rX)$, the sets $\Lambda(u_r)$ converge to $\mathcal{N}_\ast$ in the Hausdorff sense (recall \eqref{eqn:nodal is zero level}).
Therefore, the distribution $L_a q$ is supported on $\{ (x,0) : p_\ast(x,0) = 0 \}$.
Yet we are in \ref{eq.Case1}, and $\mathcal{N}_\ast$ is of zero $a$-harmonic capacity $\R^{n+1}$.
Indeed, as $p_\ast|_{\R^{n} \times \{0\}} \not\equiv 0$, the set $\mathcal{N}_\ast$ has locally finite $\mathcal{H}^{n-1}$ measure.
If $a \geq 0$, then the $a$-harmonic capacity of $\mathcal{N}_\ast$ is smaller than the harmonic capacity of $\mathcal{N}_\ast$, which is zero.
If $a < 0$, then, by assumption, $\mathcal{N}_\ast$ has locally finite $\mathcal{H}^{n-2}$ measure, which implies that it is of zero $a$-harmonic capacity (see \cite[Corollary~2.12]{Kil94}).
Thus, $q$ is $a$-harmonic, i.e., $L_aq \equiv 0$. 

Let us now show that $q$ is homogeneous, arguing as in \cite[Lemma 2.12]{FS18}, with homogeneity $\lambda_* := N(0^+, v_*)$. 
In order to do so, by \cite[Theorem 2.11]{GR18}, it suffices to show that 
\begin{equation}
\label{eq.HMon4}
\lambda_* = N(\rho, q) \quad\textrm{for all}\quad \rho\in (0, 1).
\end{equation}
Notice, first, that since $q$ is $a$-harmonic, $N(\rho, q)$ is non-decreasing. 
On the other hand, by the lower semicontinuity of the weighted Dirichlet integral,
\[
N(1, q) \le \liminf_{\ell \to \infty} N(1, \tilde v_{r_{j_\ell}}) = \liminf_{\ell \to \infty} N(1, v_{r_{j_\ell}}) = \liminf_{\ell \to \infty} N(r_{j_\ell},  v_{*}) = \lambda_*. 
\]
Also, by Lemma~\ref{lem.HMon} applied to $\tilde v_{r_{j_\ell}}$, and taking $\ell \to \infty$, 
\begin{equation}
\label{eq.HMon3}
\frac{1}{\rho^{n+a+2\lambda_*}} \int_{\pa B_\rho} q^2|y|^a \le \int_{\pa B_1} q^2|y|^a = 1.
\end{equation}
However, because $L_a q = 0$ and by \eqref{eq.HMon2}, we know that 
\[
\frac{H_\lambda'}{H_\lambda} (\rho, q) = \frac{2}{\rho} (N(\rho, q)-\lambda). 
\]
Suppose now that $N(\rho_\circ, q) = \lambda_\circ < \lambda_*$ for some $\rho_\circ \in (0, 1)$. In particular, by the previous representation of $H_\lambda$, $H_{\lambda_\circ}$ is non-increasing for $\rho \in (0, \rho_\circ)$, so that 
\[
\frac{1}{\rho^{n+a+2\lambda_*}} \int_{\pa B_\rho} q^2|y|^a\ge \frac{\rho^{2(\lambda_\circ-\lambda_*)}}{\rho_\circ^{n+a+2\lambda_\circ}} \int_{\pa B_{\rho_\circ}} q^2|y|^a>0 \quad\text{for all}\quad \rho \in (0, \rho_\circ).
\]
But this contradicts \eqref{eq.HMon3} for $\rho $ small enough. 
Therefore, \eqref{eq.HMon4} holds and $q$ is homogeneous of degree $\lambda_*$. 
And by \cite[Lemma 2.7]{CSS08}, we deduce that $q$ is a polynomial. 
In particular, $\lambda_* \geq \kap $ is an integer. 

All in all, we have that $q\not\equiv 0$ is an $a$-harmonic, even in $y$, and $\lambda_*$-homogeneous polynomial with $\lambda_*\in \{\kap, \kap+1,\kap+2,\dots\}$. 
In particular, $q\big|_{\R^n \times \{ 0 \}}\not\equiv 0$.
\end{proof}

Before moving to \ref{eq.Case2}, let us state and prove a lemma which will help us to compare $p_\ast$ and $q$ when working in \ref{eq.Case1}.
That said, this lemma is independent of \ref{eq.Case1} and \ref{eq.Case2}, and holds generically.

\begin{lemma} 
\label{lem: L2 sign}
Assume that $\tilde v_{r_{\ell}} \rightharpoonup q$ in $W^{1,2}(B_1, |y|^a)$ for some sequence $r_{\ell} \downarrow 0$.
Then,
\begin{equation}
\label{eq.ineq1}
\int_{\pa B_1} qp_*|y|^a  = 0
\end{equation}
and
\begin{equation}
\label{eq.ineq2}
\int_{\pa B_1}qp|y|^a \leq 0 \quad\text{for all}\quad p \in \mathscr{P}_\kap.
\end{equation}
\end{lemma}

\begin{proof}
We proceed as in \cite[Lemmas 2.11-2.12]{FS18}.
In order to see \eqref{eq.ineq1}, we use $H_\lambda(r,u-p)$ is non-decreasing for $\lambda = \kap \le N(0^+, u-p)$ (see Lemma~\ref{lem.HMon}), recalling that $\kappa = N(0^+,u)$, by assumption. 
In particular, we have
\begin{equation}
\label{eq.ineq3}
\begin{split}
\frac{1}{r^{n+a+2\kap}}\int_{\partial B_r} (u-p)^2|y|^a & \ge \lim_{\rho \downarrow 0}\frac{1}{\rho^{n+a+2\kap}}\int_{\partial B_\rho} (u-p)^2|y|^a\\ 
&=\lim_{\rho \downarrow 0}\int_{\partial B_1} (\rho^{-\kap}u(\rho X)-p)^2 |y|^a \\
&=\int_{\partial B_1} (p_\ast-p)^2 |y|^a,
\end{split}
\end{equation}
using the local uniform convergence of $r^{-\kap}u_r$ to $p_\ast$ as $r\downarrow 0$, with $u_r(X) := u(rX)$, and the $\kappa$-homogeneity of $p$. 
By the definition of $p_\ast$, notice that
\[
h_r := \|v_{r}\|_{L^2(\partial B_1, |y|^a)} = o(r^\kap) \quad\text{as}\quad r\downarrow 0
\qquad\textrm{and}\qquad
\varepsilon_r := \frac{h_r}{r^\kap} = o(1) \quad\text{as}\quad r\downarrow 0.
\]
Furthermore, for some subsequence, which we still denote by $r_{\ell}$, we have that $\tilde v_{r_{\ell}} = v_{r_\ell}/h_{r_\ell}\to q$ in $L^2(\partial B_1, |y|^a)$. 
Thus,
\[
\int_{\partial B_1} \left(\frac{v_{r}}{r^\kap}+p_\ast-p\right)^2|y|^a= \frac{1}{r^{n+a+2\kap}}\int_{\partial B_r} (u-p)^2 |y|^a\ge \int_{\partial B_1} (p_\ast-p)^2 |y|^a \quad\text{for all}\quad r > 0.
\]  
Since $r^{-\kap} v_{r}  = \tilde v_{r} \varepsilon_r$, taking the subsequence $r_\ell$ and expanding, we obtain
\[
\varepsilon_{r_\ell}^2 \int_{\partial B_1} \tilde v_{r_\ell}^2|y|^a + 2\varepsilon_{r_\ell} \int_{\partial B_1}\tilde v_{r_\ell} (p_\ast - p)|y|^a \ge 0 \quad\text{for all}\quad p \in \mathscr{P}_\kap.
\]
Dividing by $\varepsilon_{r_\ell}$ and taking the limit as $\ell \to \infty$,
\[
\int_{\pa B_1} q(p_\ast - p)|y|^a \ge 0 \quad\text{for all}\quad p \in \mathscr{P}_\kap.
\]

Now taking $p = 2p_\ast$ and $p = 2^{-1}p_\ast$, which are both members of $\mathscr{P}_\kap$, we deduce that
\[
\int_{\pa B_1} qp_\ast|y|^a = 0,
\]
from which \eqref{eq.ineq2} follows immediately.
\end{proof}

Let us now deal with \ref{eq.Case2}. 
As we noted before, in this case, the spine and the nodal set of $p_\ast$ can be identified: $L_\ast = \mathcal{N}_\ast$. 

\begin{proposition}
\label{prop.case2}
In \ref{eq.Case2}, for every sequence $r_j\downarrow 0$, there is a subsequence $r_{j_\ell}\downarrow 0$ such that $\tilde v_{{r_j}_\ell}\rightharpoonup q$ weakly in $W^{1, 2}(B_1, |y|^a)$ as $\ell \to \infty$, and $q\not\equiv 0$ is a $\lambda_*$-homogeneous solution to the very thin obstacle problem with zero obstacle on $L_*$,
\begin{equation}
\label{eq.verythin1}
\begin{cases}
q \ge 0 &\text{on } L_*\\
L_a q \le 0 &\text{in } \R^{n+1}\\
L_a q = 0 &\text{in } \R^{n+1}\setminus L_*\\
qL_aq = 0 &\text{in } \R^{n+1}.
\end{cases}
\end{equation}
Moreover, $\lambda_*\ge \kap+\alpha_\kap$, for some constant $\alpha_\kap>0$ depending only on $n$ and $\kap$. 
\end{proposition}
\begin{proof}
Without loss of generality, we will assume that $L_* = \{x_n = y = 0\}$. We divide the proof into several steps.

\smallskip
\noindent{\bf Step 1: Weak limit and non-negativity on $L_*$.}
As in the proof of Proposition~\ref{prop.case1}, we have that
\begin{equation}
\label{eq.wconvq}
\tilde v_{r_\ell} \rightharpoonup q\quad\text{in}\quad W^{1, 2}(B_1, |y|^a),
\end{equation}
for some $q$, and $L_a\tilde v_{r}$ is converging weakly$^*$ as measures to a non-positive measure $L_a q$ supported on $L_*$.
Unlike before, the set on which $L_a q$ is supported is now a set of strictly positive $a$-harmonic capacity (since $m = n-1$).

Consider the following trace operators
\[
\gamma: W^{1, 2}(B_1, |y|^a)\to W^{s, 2}(B_1^*) \quad\text{and}\quad \tilde \gamma: W^{s, 2}(B_1^*)\to W^{s-\frac12, 2}(B_1').
\] 
By \cite{NLM88} (see also \cite{Kim07}), since $s > 1/2$, $\gamma$ is continuous; and $\tilde{\gamma}$ is the standard continuous trace operator.
(Recall that $a = 1-2s$.)
The operator $\tau := \tilde \gamma \circ \gamma$ then is continuous.
Hence, considering \eqref{eq.wconvq},
\[
\tau(\tilde v_{r_\ell}) \rightharpoonup \tau(q) \quad\text{in}\quad W^{s-\frac12, 2}(B_1') \qquad\text{and}\qquad \tau(\tilde v_{r_\ell}) \to \tau(q) \quad\text{in}\quad L^2(B_1').
\]

Now $\tau(\tilde v_{r_\ell}) \ge 0$ on $B_1'$ for all $\ell\in \N$, since $p_* \equiv 0$ and $u \ge 0$ on $L_*$. 
Thus, from the strong convergence above, $\tau(q) \ge 0$, or $q \ge 0$ on $L_*$. 

\smallskip
\noindent{\bf Step 2: Semiconvexity in directions parallel to $L_*$.} By Lemma~\ref{lem: Bern X2},
\begin{equation}
\label{eq.semiconv}
\inf_{B_{1/2}} \pa_{\be\be}\tilde v_{r} \ge -C \quad\text{for all}\quad \be \in L_*\cap \mathbb{S}^n,
\end{equation}
for some constant $C$ independent of $r$. 
Namely, the sequence of functions $\tilde v_{r}$ is locally uniformly semiconvex (and, therefore, locally uniformly Lipschitz) in the directions parallel to $L_*$.

\smallskip
\noindent{\bf Step 3: Strong convergence.} 
We show that for every $0 < \vep \ll 1$, there exists a constant $C_\vep > 0$ independent of $r_\ell$ for which
\begin{equation}
\label{eq.unifhold}
[\tilde v_{r_\ell}]_{C^{-a-\varepsilon}(B_{1/2})} \le C_\varepsilon.
\end{equation}
Thus, by a covering argument, $\tilde v_{r_\ell}\to q$ locally uniformly in $B_1$, and, in fact, $q \in C_{\rm loc}^{-a-\varepsilon}(B_1)$.

Recall that $L_* = \{x_n = y = 0\}$ and $X = (x', x_n, y)$ for $x'\in \R^{n-1}$.  
For simplicity, in the following computations, set
\[
w := \tilde v_{r}.
\]
Let $Q_{r_\circ} := B_{r_\circ}'\times D_{r_\circ} \subset B_{1}$, for some $r_\circ > 0$.
Recall that $D_r$ denotes the disc of radius $r$ in $\R^2$ centered at the origin. 
For convenience, rescale and assume $r_\circ = 1$.
By Step 2, $\|w(x', \cdot, \cdot)\|_{L^2(D_1, |y|^a)}^2$ is Lipschitz, as a function of $x'$. 
Hence,
\[
\osc_{B'_1} \|w(x', \cdot, \cdot)\|_{L^2(D_1, |y|^a)}^2 \leq C.
\]
Recalling that $\|w\|_{L^2(B_1, |y|^a)} \le C$ (we have rescaled to work in $Q_1$, else this bound would be $1$), we have that
\begin{equation}
\label{eq.L2loc}
\int_{B_1'}\|w(x', \cdot, \cdot)\|_{L^2(D_1, |y|^a)}^2 \, \d x'\le C,
\end{equation}
and so $\|w(x', \cdot, \cdot)\|_{L^2(D_1, |y|^a)}$ has bounded oscillation and integral. 
In turn,
\begin{equation}
\label{eq.L2loc2}
\|w(x', \cdot, \cdot)\|_{L^2(D_1, |y|^a)} \le C \quad \text{for all}\quad x'\in B'_1.
\end{equation}
We also recall that 
\begin{equation}
\label{eqn: Law 0}
\lim_{y \downarrow 0} y^a \pa_y w \le 0
\qquad\text{and}\qquad
L_a w = 0 \quad\text{in}\quad B_1 \cap \{ y > 0\}.
\end{equation}

\smallskip
\noindent{\it -- Step 3.1.} In this subset, we prove that the measure
\[
n_w(x', x_n):= \lim_{y\downarrow 0} y^a\pa_y w \le 0
\]
is finite on each $x'$ slice. 
Equivalently, we show that
\begin{equation}
\label{eq.normbound}
0\ge \int_{-1}^{1} \zeta(|(x',x_n)|) n_w(x', x_n) \,\d x_n \ge -C \quad\text{for all}\quad x' \in B'_1
\end{equation}
where $\zeta$ is a smooth test function $\zeta = \zeta(r):[0, \infty)\to [0, 1]$ such that $\zeta \equiv 1$ in $[0, 1/2]$ and $\zeta \equiv 0 $ in $[3/4,\infty)$. 

Let $\zeta = \zeta(|(x',x_n, y)|)$.
By the divergence theorem, 
\begin{equation}
\label{eq.diveq}
\begin{split}
\int_{-1}^{1} \zeta n_w \,\d x_n  &= -\int_{D_1\cap \{y > 0\}} \DIV_{x_n, y}(\zeta y^a \nabla_{x_n,y} w)\,\d x_n\,\d y \\
&= - \int_{D_1\cap \{y > 0\}}  \zeta L^{x_n,y}_a w  - \int_{D_1\cap \{y > 0\}}  y^a \nab_{x_n,y} \zeta \cdot \nab_{x_n,y} w\\
&=: \text{I} + \text{II}
\end{split}
\end{equation}
where $L^{x_n,y}_a f := \DIV_{x_n,y} (|y|^a \nab_{x_n,y} f)$.
On one hand, observe that 
\[
L_a^{x_n, y} w = L_a w - y^a\Delta_{x'} w = - y^a\Delta_{x'} w  \quad\text{in}\quad D_1 \cap \{ y > 0 \}
\]
by \eqref{eqn: Law 0}.
And so, by \eqref{eq.semiconv},
\begin{equation}
\label{eqn: I}
\text{I} \geq - C.
\end{equation}
On the other hand, by the symmetries of $\zeta$ (i.e., $\pa_y \zeta = O(y)$ as $\pa_y \zeta\big|_{y = 0} = 0$ and $\zeta$ is smooth),
\[
|w L_a^{x_n,y} \zeta| = | w y^a y^{-a} L_a^{x_n,y} \zeta| =  |w|y^a |\pa_{nn}\zeta + \pa_{yy} \zeta + ay^{-1}\pa_y \zeta| \leq C |w| y^a.
\]
So, by the symmetries of $\zeta$ again, H\"older's inequality, and \eqref{eq.L2loc2}, we deduce that
\begin{equation}
\label{eqn: II}
\begin{split}
\text{II}  = - \int_{D_1\cap \{y > 0\}}  y^a \nab_{x_n,y} \zeta \cdot \nab_{x_n,y} w = \int_{D_1\cap \{y > 0\}} w L_a^{x_n,y} \zeta \geq -C. 
\end{split}
\end{equation}
We have also used that the boundary term at $y = 0$ vanishes in the integration by parts, $y^a\pa_y\zeta \equiv 0$ on $\{y = 0\}$. 
Therefore, combining \eqref{eq.diveq}, \eqref{eqn: I}, and \eqref{eqn: II}, we see that \eqref{eq.normbound} holds, as desired.

\smallskip
\noindent{\it -- Step 3.2.} Now we conclude. 
Consider the fundamental solution for the operator $L_a$ (see, e.g, \cite{CS07}) given by 
\[
\Gamma_a(X) := C_{n, a} |X|^{-n+1-a}.
\]
More precisely, $\Gamma_a$ is such that $L_a \Gamma_a = 0$ in $\{|y| > 0\}$ and $\lim_{y\downarrow 0} y^a\pa_y \Gamma_a = \delta(x)$, the Dirac delta at $x$.
Let
\[
\bar w(x, y) := \Gamma_a(\,\cdot\,, y) *_x (\zeta n_w),
\] 
where $\zeta$ is the test function defined in Step 3.1, with $\zeta = \zeta(|x|)$ here. We have that $L_a\bar w = 0$ in $|y|> 0$, and $\lim_{y\downarrow 0} y^a \pa_y\bar w = \zeta n_w$. 
We claim that $\bar w$ is bounded. 
Indeed, by \eqref{eq.normbound},
\begin{align*}
|\bar w(x', x_n, y)|&  \leq \int_{B_1'}  \int_{-1}^1  \frac{(\zeta n_w )(z', z_n)}{|(x'-z', x_n-z_n, y)|^{n-1+a}}  \, \d z_n  \,\d z' \\
& \le C\int_{B_1'} \frac{\d z'}{|(x'-z', 0, y)|^{n-1+a}} \le C.
\end{align*} 
By means of the previous proof, $(-\Delta)^{\bar s}_X \bar w = \left((-\Delta)^{\bar s}_X\Gamma_a *_x (\zeta n_w)\right)$ is bounded as long as $2\bar s < -a$, since $(-\Delta)^{\bar s}_X|X|^{-n+1-a} = C|X|^{-n+1-a-2\bar s}$, and $\zeta n_w$ does not depend on $y$. 
Thus, $(-\Delta)^{\bar s}_X \bar w$ is bounded as long so $2\bar s < -a$, and by interior regularity for the fractional Laplacian (suppose $\bar s \neq 1/2$), $\bar w$ is $C^{2\bar s}$ (see \cite[Theorem 1.1]{RS16}).

Finally, notice that $L_a (\bar w - w) = 0$ in $B_1 \cap \{|y|>0\}$ and $\lim_{y\downarrow 0} y^a\pa_y\left(\bar w - w\right)=0$ in $B_{1/2}\cap \{ |y|> 0\}$.
It follows that $L_a (\bar w - w) = 0$ in $B_{1/2}$, and then $\bar w - w \in C^1_{\rm loc}(B_{1/2})$ by interior estimates for $a$-harmonic functions (and recalling that $a \in (-1,0))$. 
In turn, $w$ inherits the regularity of $\bar w$; that is, $w$ is $C^{2\bar s}$, so long as $2\bar s < -a$, and \eqref{eq.unifhold} is proved. 

In particular, by Arzel\`a--Ascoli and a covering argument, we have that 
\begin{equation}
\label{eq.strongconv}
\tilde v_{r_\ell} \to q \quad\text{in}\quad C^0_{\rm loc}(B_1),
\end{equation}
and $q\in C^{-a-\varepsilon}_{\rm loc}(B_1)$ for any $\varepsilon > 0$.

\smallskip
\noindent{\bf Step 4: Homogeneous solution to the very thin obstacle problem in $B_1$.} First, we show that $q$ is a solution to the very thin obstacle problem, \eqref{eq.verythin1}; the only condition that remains to be checked is that $q\,L_a q \equiv 0$.

By the proof of Proposition~\ref{prop:  Almgren} and \eqref{eqn: key ineq 2},
\begin{align*}
\frac{rN'(r,v_\ast)}{N(r,v_\ast)} = \frac{\d}{\d \rho}\log N(\rho, v_r) \bigg|_{\rho = 1} & \ge \frac{2\left(\int_{B_1}v_r L_a v_r\right)^2}{\int_{B_1} |\nabla v_r|^2|y|^a \int_{\pa B_1} v_r^2|y|^a} \geq 0.
\end{align*}
Hence, by the definition of $\tilde v_r$,
\begin{equation}
\label{eq.rNp}
rN'(r, v_\ast) \ge 
2\left(\int_{B_1} \tilde v_r L_a \tilde v_r\right).
\end{equation}
Furthermore, reasoning as in \cite[Lemma 2.12]{FS18}, since $N(r, v)\downarrow \lambda_*$ as $r\downarrow 0$,
\[
\dashint_{r_{j_\ell}}^{2r_{j_\ell}} r N'(r, v_\ast) \, \d r \le 2(N(2r_{j_\ell}, v_\ast)-N(r_{j_\ell}, v_\ast))\to 0 \quad\text{as}\quad \ell \to \infty.
\]
And so, by the mean value theorem, we can find $\bar r_{j_\ell} \in [r_{j_\ell},2r_{j_\ell}]$ with $\bar r_{j_\ell} N'(\bar r_{j_\ell}, v_\ast) \to 0$ as $\ell \to \infty$. 
In turn, the non-negativity of $v_\ast\, L_a v_\ast$ and \eqref{eq.rNp} then imply that
\[
\int_{B_1} \tilde v_{r_{j_\ell}} \,L_a \tilde v_{r_{j_\ell}} \le \int_{B_{\bar \rho_{j_\ell}}} \tilde v_{r_{j_\ell}} \,L_a \tilde v_{r_{j_\ell}} \to 0
\]
with $\bar \rho_{j_\ell} := \bar r_{j_\ell}/r_{j_\ell}$.
Therefore, since $L_a \tilde v_{r_{j_\ell}} \rightharpoonup L_a q$ weakly$^*$ as measures in $B_1$, $\tilde v_{r_{j_\ell}}  \to q$ strongly in $C^0_{\rm loc}(B_1)$ by Step 3, \eqref{eq.strongconv}, and $\tilde v_{r}\, L_a \tilde v_{r} \ge 0$, we obtain that 
\[
\int_{B_R} q\,L_aq = 0 \quad\text{for all}\quad R < 1,
\]
so that, in fact, $q \, L_a q \equiv 0$ in $B_1$. 

Thus, $q$ is a solution to the very thin obstacle problem \eqref{eq.verythin1} inside $B_1$. 

To conclude, we show that $q$ is homogeneous with homogeneity $\lambda_* := N(0^+, v_\ast)$. 
Since $q$ solves the very thin obstacle problem, by Lemma~\ref{lem.MonFreq}, it suffices to show that 
\begin{equation}
\label{eq.HMon0_2}
\lambda_* = N(\rho, q) \quad\text{for all}\quad \rho\in (0, 1).
\end{equation}
But this follows from arguing exactly as in the proof of Proposition~\ref{prop.case1}, where we obtained that $q$ is homogeneous in \ref{eq.Case1}, using Lemma~\ref{lem.MonFreq}, \eqref{eq.HMon2_2}, and Lemma~\ref{lem.HMon_VTOP}.

\smallskip
\noindent{\bf Step 5: $\lambda_*\ge \kap+\alpha_\kap$.}
We argue by contradiction (or compactness).
Suppose, to the contrary, that there exists a bounded sequence of solutions $u_{\ell}$ such that  $0\in \Sigma_\kap(u_\ell)$, $\dim_{\mathcal{H}} L(p_{\ast,\ell}) = n-1$, and $\lambda_{\ast,\ell} \le \kap+\ell^{-1}$.
Let $p_{\ast,\ell}$ be the first blow-up and $q_\ell$ be a second blow-up of $u_\ell$ at $0$ (the homogeneity of $q_\ell$ is $\lambda_{\ast,\ell}$).
Up to a subsequence (we can assume the sequences enjoy uniform bounds in appropriate H\"older spaces), taking $\ell$ to infinity, we find a solution $u_\infty$ whose first blow-up at $0$ is of order $\kap$, whose spine has Hausdorff dimension equal to $n-1$, and whose second blow-up $q_\infty$ is homogeneous of order $\kap$.

Since $q_\infty$ is a $\kap$-homogeneous, global solution to the very thin obstacle problem, it is an $a$-harmonic polynomial. 
Indeed, by \cite[Proposition 4.4]{GR18}, any global, evenly homogeneous function $u_\circ$ with $L_a u_\circ$ non-negative and supported on $\R^n \times \{ 0 \}$ is actually an $a$-harmonic polynomial of degree $\kap$.  
In particular, we have that $\|q_\infty\|_{{\rm Lip}(B_1)} \le C$ for some constant depending only on $n$, $a$, and $\kap$. 
Also, by assumption, $q_\infty \geq 0$ on $L(p_{\ast,\infty})$, where $p_{\ast,\infty}$ is the first blow-up of $u_\infty$ at $0$.  

For simplicity, let $q = q_\infty$ and $p_* = p_{\ast,\infty}$, and let us assume that $L(p_{\ast,\infty}) = \{x_n = 0\}$, so that $p_*$ depends only on $x_n$ in the thin space $\{y = 0\}$. 
By Lemma~\ref{lem:  L2 sign},
\begin{equation}
\label{eq.comp}
\langle q, p \rangle_a := \int_{\pa B_1} q p|y|^a \le 0\quad\text{for all}\quad p\in \mathscr{P}_\kap \quad\text{and}\quad \langle q, p_*\rangle_a = 0.
\end{equation}
Since $p_\ast$ is $\kappa$-homogeneous and depends only on $x_n$, a constant $c_\ast > 0$ exists for which $p_\ast|_{B_1 \cap \{y =0 \}} =c_\ast|x_n|^\kappa$.
Now for any $\vep > 0$, observe that
\[
C_\varepsilon p_* + q \ge -\varepsilon \quad\text{on}\quad \pa B_1 \cap \{y = 0 \}
\]
with
\[
C_\vep := c_\ast^{-1} \vep^{-\kappa}\|q\|_{L^\infty(B_1 \cap \{y =0 \})} \|q\|_{{\rm Lip}(B_1 \cap \{y =0 \})}^\kappa.
\]
Indeed, if $|x_n|\ge \vep/\|q\|_{{\rm Lip}(B_1 \cap \{y =0 \})}$, then $C_\varepsilon p_*|_{B_1 \cap \{y =0 \}} + q|_{B_1 \cap \{y =0 \}} \geq 0$, by the definition of $C_\vep$.
On the other hand, if $|x_n|\le \vep/\|q\|_{{\rm Lip}(B_1 \cap \{y =0 \})}$, then $q|_{B_1 \cap \{y =0 \}} \ge -\vep$ since $q \ge 0$ on $\{x_n = 0\}$ (recall $p_\ast \geq 0$ on the thin space).
Thus, $C_\varepsilon p_*+q+\varepsilon\Ext_a(|x|^\kappa) \in \mathscr{P}_\kappa$ for every $\varepsilon > 0$ (see \eqref{eq.ext}). 
So \eqref{eq.comp} implies that
\[
\|q\|_{L^2(\pa B_1, |y|^a)}\le -\varepsilon\langle \Ext_a(|x|^\kappa), q\rangle_a.
\]
Taking $\vep \downarrow 0$, we deduce that $q\equiv 0$, a contradiction. 
\end{proof}

With Propositions~\ref{prop.case1} and \ref{prop.case2} in hand, we can now prove Proposition~\ref{prop.main1}.

\begin{proof}[Proof of Proposition~\ref{prop.main1}]
The proof is a simple consequence of Propositions~\ref{prop.case1} and \ref{prop.case2}.
Without loss of generality, $X_\circ = 0$.
\begin{enumerate}[(i)]
\item If $a\in [0,1)$, we are in \ref{eq.Case1}.
So by Proposition~\ref{prop.case1}, our claim holds.
\item When $\kap = 2$, since $p_*\ge 0$ on the thin space, we have that $L_* = \mathcal{N}_*$. 
Thus, since $m < n-1$, we are again in \ref{eq.Case1}, and we conclude by Proposition~\ref{prop.case1} once more.
\item Finally, if $m = n-1$ and $a \in (-1, 0)$, we are in \ref{eq.Case2} (recall $L_\ast \subset \mathcal{N}_\ast$). 
Thus, applying Proposition~\ref{prop.case2}, we arrive at our desired conclusion.
\end{enumerate}
This completes the proof. 
\end{proof}

\section{Accumulation Lemmas}
\label{sec:accumulation}

In this section, we gather some important lemmas concerning accumulation points of ${\rm Sing}(u)$.
These lemmas are the key tools used in estimating the size of the points where we can construct the next term in the expansion of $u$. 
The lemmas of this section are analogous to the accumulation lemmas of \cite{FS18}, although several new, interesting technical challenges appear in our setting.

Let us start by proving an auxiliary lemma. 

\begin{lemma}
\label{lem: k-deg h poly Almgren bd}
Let $q$ be a $\kap$-degree, $a$-harmonic polynomial, for $\kappa \geq 1$, and let $X_\circ\in \R^{n+1}$.
Then,
\[
N(r,q,X_\circ) = \frac{r\int_{B_r(X_\circ)} |\nab q|^2|y|^a}{ \int_{\pa B_r(X_\circ)} q^2|y|^a} \leq \kap \quad\text{for all}\quad r > 0.
\]
Moreover,
\[
N(0^+,q,X_\circ) = m_\circ
\]
where $m_\circ$ is the smallest integer for which the $m_\circ$-homogeneous part of $q(X_\circ + \cdot)$ is non-zero.
\end{lemma}

\begin{proof}
Without loss of generality, we assume that $X_\circ = 0$.
Let
\[
q = \sum_{m = 0}^{\kap} q_0
\]
where $q_m$ denotes the $m$-homogeneous part of $q$.
Since $q$ is $a$-harmonic and $a\in (-1,1)$, each of its homogeneous parts is $a$-harmonic.
Notice that if $p_1$ and $p_2$ are homogeneous $a$-harmonic polynomials with non-zero homogeneities $m_1\ne m_2$, then they are orthogonal in $L^1(\pa B_r, |y|^a)$. 
Indeed, using that $m_i p_i = x\cdot \nabla p_i = r \pa_\nu p_i$ on $\pa B_r$ and integrating by parts,
\begin{align*}
(m_1-m_2)\int_{\pa B_r} p_1 p_2|y|^a & = r\int_{\pa B_r} p_2 \pa_\nu p_1  |y|^a -  r\int_{\pa B_r} p_1\pa_\nu p_2  |y|^a\\
& = - r\int_{B_r} \nabla p_1\cdot\nabla p_2 |y|^a + r\int_{B_r} \nabla p_1\cdot\nabla p_2 |y|^a = 0,
\end{align*}
where we have also used that $L_a p_i = 0$. 

Now, by means of the $m$-homogeneity of $q_m$ and the orthogonality in $L^2(\pa B_r, |y|^a)$ of homogeneous $a$-harmonic polynomials of different homogeneities, we find that
\[
\int_{B_r} \nab q \cdot \nab q_m |y|^a 
= \frac{m}{r} \int_{\pa B_r} q_m^2|y|^a.
\]
Thus,
\[
r \int_{B_r} |\nab q|^2|y|^a = \sum_{m = 1}^{\kap} m \int_{\pa B_r} q_m^2|y|^a \leq \kap \sum_{m = 1}^{\kap} \int_{\pa B_r} q_m^2|y|^a.
\]
Pythagoras's theorem also implies that
\[
\int_{\pa B_r} q^2|y|^a = \sum_{m = 0}^{\kap} \int_{\pa B_r} q_m^2 |y|^a.
\]
Hence,
\[
r \int_{B_r} |\nab q|^2|y|^a \leq \kappa \int_{\pa B_r} q^2|y|^a,
\]
as desired.

Now let $c_m := \int_{\partial B_1} q_m^2|y|^a$, and set $m_\circ \geq 0$ to be the smallest integer so that $c_{m_\circ} \neq 0$.
Then,
\[
\frac{r \int_{B_r} |\nab q|^2|y|^a}{\int_{\pa B_r} q^2|y|^a} = \frac{\sum_{m = m_\circ}^\kappa mc_mr^{2m}}{\sum_{m = m_\circ}^\kappa c_mr^{2m}} = m_\circ + O(r^2),
\]
which concludes the proof.
\end{proof}

Just as in Section~\ref{sec:Blow-up Analysis}, we divide our attention between \ref{eq.Case1} and \ref{eq.Case2}.
Again, we begin with \ref{eq.Case1}. 
Our accumulation lemma in this case is analogous to \cite[Lemma 3.2]{FS18}. 
We repeat the common parts for completeness. 

We recall that, in the following lemmas, we are assuming that $0\in \Sigma_\kap$ is a singular point of order $\kap \in 2\N$. 

\begin{lemma}
\label{lem: accumulation pts}
In \ref{eq.Case1}, suppose that there exists a sequence of free boundary points $\Sigma_{\ge \kap} \ni X_\ell = (x_\ell, 0) \to 0$ and radii $r_\ell\downarrow 0$ with $|X_\ell|\le r_\ell/2$ such that $\tilde v_{r_\ell} \rightharpoonup q$ in $W^{1,2}(B_1, |y |^a)$ and $Z_\ell := X_\ell/r_\ell \to Z_\infty$. 
Then, 
\[
Z_\infty = (z_\infty,0) \in L_\ast \qquad\text{and}\qquad D^\alpha q(Z_\infty) = 0 \quad\text{for all}\quad \alpha = (\alpha',0) \text{ and } |\alpha| \leq \kappa - 2.
\] 
Moreover, if $\lambda_* = \kap$, then $q_{Z_\infty} := q(Z_\infty + \,\cdot\,)-q$ is invariant under $L_\ast + L(q)$; that is,
\[
q_{Z_\infty}(x+\xi,0) = q_{Z_\infty}(x,0) \quad\text{for all pairs}\quad (\xi,x) \in (L_\ast + L(q)) \times \R^n.
\]
\end{lemma}

\begin{proof}
From Proposition~\ref{prop: Almgren} applied at $X_\ell$, the frequency of $u(X_\ell + \,\cdot\,) -p_*$ is at least $\kap$.
(Here, $p_*$ is being considered as just an element of $\mathscr{P}_\kap$.
Recall, $p_*$ is the blow-up at $0$, not at $X_\ell$.)
Therefore, 
\[
N(\rho, u(X_\ell + r_\ell\,\cdot\,) - p_{*} (r_\ell \,\cdot\,)) = N(\rho r_\ell,u(X_\ell + \,\cdot\,) -p_*) \ge \kap \quad\text{for all}\quad \rho \in (0,1/2),
\]
or, equivalently, for all $\rho \in (0,1/2)$,
\begin{equation}
\label{eq: etaAlmgren}
\frac{\rho\int_{B_\rho}|\nabla \tilde v_{r_\ell}(Z_\ell+\,\cdot\,) + h_{r_\ell}^{-1}\nabla (p_*(X_\ell+r_\ell\,\cdot\,)-p_*(r_\ell\,\cdot\,))|^2|y|^a}{\int_{\pa B_\rho} | \tilde v_{r_\ell}(Z_\ell+\,\cdot\,) + h_{r_\ell}^{-1} (p_*(X_\ell+r_\ell\,\cdot\,)-p_*(r_\ell\,\cdot\,))|^2|y|^a} \ge \kap
\end{equation}
with
\[
h_{r_\ell} := \|v_{r_\ell}\|_{L^2(\pa B_1,|y|^a)}.
\] 
Now let 
\[
q_\ell (X) := \frac{p_{*}(X_\ell +r_\ell X) - p_{*}(r_\ell X)}{h_{r_\ell}},
\]
which is a $(\kap-1)$-degree, $a$-harmonic polynomial.
Also, observe that
\begin{equation}
\label{eqn: translated W12 bd}
\int_{B_{1/2}}|\tilde v_{r_\ell}(Z_\ell + \,\cdot\,)|^2|y|^a + \int_{B_{1/2}} |\nabla \tilde v_{r_\ell}(Z_\ell + \,\cdot\,)|^2|y|^a \leq \|\tilde v_{r_\ell}\|^2_{W^{1, 2}(B_1,|y|^a)} \le C.
\end{equation}

We claim that the coefficients of $q_\ell$ are uniformly bounded with respect to $\ell$, so that, up to subsequences, $q_\ell \to q_\infty$ locally uniformly where $q_\infty$ is some $a$-harmonic polynomial of degree $\kap - 1$. 
Indeed, suppose that this is not true.
Then, letting $\{a^\ell\}_{i\in \mathscr{I}}$ denote the coefficients of $q_\ell$ and setting $\sigma_\ell := \sum_{i\in \mathscr{I}} |a_i^\ell|$, we have that $\sigma_\ell \to \infty$.
Now set
\[
\bar{q}_\ell := \frac{q_\ell}{\sigma_\ell},
\] 
which is a polynomial with coefficients bounded by $1$, and let $\bar q_\infty$ denote its limit (up to a subsequence).
Notice that $\bar{q}_\infty$ is an $a$-harmonic, $(\kap-1)$-degree polynomial as $\bar{q}_\ell$ are all $a$-harmonic, $(\kap-1)$-degree polynomials.
So, from \eqref{eq: etaAlmgren}, dividing the numerator and denominator by $\sigma_\ell^2$, and by Lemma~\ref{lem: k-deg h poly Almgren bd}, we deduce that
\[
\kap \le \frac{\rho \int_{B_\rho}|\nabla \vep_{\ell} + \nab \bar q_\ell|^2|y|^a}{\int_{\pa B_\rho} |\vep_{\ell} + \bar q_\ell|^2|y|^a}\to \frac{\rho \int_{B_\rho}|\nabla \bar q_\infty|^2|y|^a}{\int_{\pa B_\rho} |\bar q_\infty|^2|y|^a} = N(\rho, \bar q_\infty) \leq \kap - 1
\]
since, by \eqref{eqn: translated W12 bd},
\[
\vep_\ell := \frac{\tilde v_{r_\ell}(Z_\ell + \,\cdot\,)}{\sigma_\ell} \to 0 \quad\text{in}\quad W^{1, 2}(B_{1/2}, |y|^a).
\]
Impossible.

Since $q_\ell$ converges, up to subsequences, to some $q_\infty$ uniformly in compact sets and by interior estimates for $a$-harmonic functions (see, e.g., \cite[Propsition~2.3]{JN17}), we have that $|D^\alpha q_\ell (0) |\le C$ for some $C$ independently of $\ell$ for any multi-index $\alpha = (\alpha_1, \dots,\alpha_{n},0)$.
Then, from the  $\kap$-homogeneity of $p_{*}$, we have 
\begin{equation}
\label{eqn: ql and pk}
D^\alpha q_\ell (0)  =  \frac{r_\ell^{|\alpha|}}{h_{r_\ell}}D^\alpha p_*(X_\ell) =  \frac{r_\ell^{\kap}}{h_{r_\ell}} D^\alpha p_* (Z_\ell)
\end{equation} 
for all $|\alpha| \leq \kappa - 1$.
Hence, using $|D^\alpha q_\ell (0) |\le C$ and $h_{r_\ell} = o(r_\ell^\kap)$, we determine that
\[
|D^\alpha p_* (Z_\ell)| = o(1) \to 0 \quad\text{as}\quad \ell \to \infty  
\]
when $|a| \leq \kappa - 1$.
That is, $D^\alpha p_{*} (Z_\infty) = 0$ for $|\alpha|\le \kap - 1$. 
Thanks to Lemma~\ref{lem: equiv},
\[
Z_\infty\in L_\ast.
\] 

Proceeding as in \cite[Lemma 3.2]{FS18} by means of the Monneau-type monotonicity formula from Lemma~\ref{lem.HMon}, we obtain
\begin{equation}
\label{eq: avgs}
\frac{1}{\rho^{a+2\kap}} \dashint_{\pa B_\rho} |q(Z_\infty+ \,\cdot\,) + q_\infty|^2|y|^a \le 2^{a+2\kap}\dashint_{\pa B_{1/2}} |q(Z_\infty + \,\cdot\,) + q_\infty|^2|y|^a
\end{equation}
for all $\rho \in (0, 1/2)$.
Notice that, until now, we have not used any information on the second blow-up $q$. 
From Proposition \ref{prop.case1}, $q$ is a $\lambda_*$-homogeneous, $a$-harmonic polynomial with $\lambda_*\ge \kap$, since we are in \ref{eq.Case1}. 
It follows that the polynomial $q(Z_\infty+\,\cdot\,) +q_\infty$ is only made up of monomials of degree greater than or equal to $\kap$. 
Thus, recalling \eqref{eqn: ql and pk}, we have that
\[
\lambda_* q(Z_\infty) = Z_\infty\cdot \nabla_{x} q(Z_\infty) = -Z_\infty\cdot \nabla_{x} q_\infty(0) = - \lim_{\ell \to \infty} \frac{r_\ell}{h_{r_\ell}} (Z_\infty \cdot \nabla_{x} p_*(X_\ell)) = 0.
\]
Here, we have also used that $Z_\infty\in L_\ast$, $X_\ell = (x_\ell,0)$, and $q$ is $\kap_\ast$-homogeneous. 
Moreover, taking derivatives, we have
\[
(\lambda_* - |\alpha|) D^\alpha q(Z_\infty) = - \lim_{\ell \to \infty} \frac{r_\ell^\kappa}{h_{r_\ell}} (Z_\infty \cdot \nabla_{x} D^\alpha p_*(Z_\ell)) = 0.
\] 
(By Lemma~\ref{lem: equiv}, $Z_\infty \cdot \nabla_{x} D^\alpha p_*(Z_\ell) = 0$.)
Therefore,
\[
D^\alpha q(Z_\infty) = 0 \quad\text{for all}\quad \alpha = (\alpha',0) \text{ and } |\alpha| \leq \kappa - 2.
\]
In addition, notice that by construction, $q_\ell$ is invariant under $L_\ast$.
Hence, so is $q_\infty$.

Finally, suppose $\lambda_* = \kap$. 
Then, $q(Z_\infty+\,\cdot\,)+q_\infty$ consists of only degree $\kap$ terms.
In other words, it is $\kap$-homogeneous.
Now notice that $q(Z_\infty+\,\cdot\,) = q + s_\infty$ where $s_\infty$ is a degree $\kap -1$ polynomial.
Consequently, $q(Z_\infty+\,\cdot\,)+q_\infty - q = s_\infty + q_\infty$ is a $\kap$-homogeneous polynomial.
This is only possible if $s_\infty + q_\infty \equiv 0$ (recall, $q_\infty$ is of degree $\kap - 1$.)
And so, it follows that 
\[
q_\infty = q - q(Z_\infty + \,\cdot\,),
\]
from which we deduce that $q_\infty$ is invariant under $L(q)$.
Since the invariant set of a function is a linear space,
\[
q_\infty(x+\xi,0) = q_\infty(x,0) \quad\text{for all pairs}\quad (\xi,x) \in (L_\ast + L(q)) \times \R^n.
\]
Lastly, we find that
\[
D^\alpha q_\infty(0) = 0 \quad\text{for all}\quad \alpha = (\alpha',0) \text{ and } |\alpha| \leq \kappa - 2,
\]
making $q_\infty$ a $(\kappa - 1)$-homogeneous, even in $y$, $a$-harmonic polynomial.
\end{proof}

Notice that if $Z_\infty \in L(q)$, then $q_\infty \equiv 0$.
Indeed, all of the derivatives of $q_\infty|_{\R^n \times \{ 0 \}}$ up to order $\kappa - 2$ vanish at the origin since $D^\alpha q(Z_\infty) = 0$ for all $\alpha = (\alpha',0)$ and $|\alpha| \leq \kappa - 2$.
So if $D^\alpha q(Z_\infty) = 0$ for all $\alpha = (\alpha',0)$ with $|\alpha| \leq \kappa - 1$ too, then $q_\infty$ would vanish up to infinite order at the origin, making it identically zero.
In other words,
\[
Z_\infty \in L(q) \text{ if and only if } q_\infty \equiv 0.
\]
This also follows directly from the form $q_\infty$ takes when $\lambda_* = \kappa$.

Before stating and proving a \ref{eq.Case2} accumulation lemma, we present a simple consequence of Lemma~\ref{lem: accumulation pts} and make a remark.

If $m_{\ast} = 0$, then $L_\ast = \{ 0 \}$.
Hence, from Lemma~\ref{lem: accumulation pts}, we deduce that $\Sigma_{\kap}^0$ is isolated in $\Sigma_{\geq \kap}$.

\begin{lemma}
\label{lem: Sigma0k is isolated}
Suppose \ref{eq.Case1} holds. 
Then, $0$ is an isolated point of $\Sigma_{\geq \kap}$.
\end{lemma}

\begin{proof}
Suppose, to the contrary, that $\Sigma_{\geq \kap} \ni X_\ell \to 0$ is a sequence of points ($X_\ell \neq 0$).
Let $r_\ell := 2|X_\ell|$.
By Lemma~\ref{lem: accumulation pts}, we have that, up to a subsequence,
\[
\tilde{v}_{r_\ell} \rightharpoonup q \quad \text{in}\quad W^{1,2}(B_1, |y|^a) \qquad\text{and}\qquad Z_\ell := \frac{X_\ell}{r_\ell} \to Z_\infty \in L_\ast \cap \pa B_{1/2}
\]
where $q$ is a $\kap^\ast$-homogeneous harmonic polynomial with $\lambda_* \geq \kap$.
But, this is impossible, since $L_\ast = \{ 0 \}$. 
\end{proof}

\begin{remark}
\label{rem.count}
In general, lower frequency singular points can accumulate to a higher frequency singular point.
Take, for example, the harmonic extension of $x_1^2 x_2^2$ to $\R^3$: 
\[
u(X) = x_1^2x_2^2 - (x_1^2 + x_2^2)y^2 + \frac{1}{3}y^4.
\]
This polynomial is a solution to the thin obstacle problem with $a = 0$, and has singular points of order $2$ approaching a singular point of order $4$. 
In particular, it is not true that $\Sigma_\kap^0$ is isolated from $\Sigma_{< \kap}$.

By the recent results of Colombo, Spolaor, and Velichkov, see \cite[Theorem 4]{CSV19}, we know that the set of even frequencies ($\kap = 2m$) is isolated from the set of all possible frequencies for the thin obstacle problem when $a = 0$. 
This, together with the upper semicontinuity of the frequency, implies that free boundary points of strictly higher order cannot accumulate to a singular point of lower order in this case.
Therefore, the above hypothesis ``$X_\ell\in \Sigma_{\geq \kappa}$ and $X_\ell \to 0 \in \Sigma_\kappa$'' reduces to ``$X_\ell\in \Sigma_{\kappa}$ and $X_\ell \to 0 \in \Sigma_\kappa$'', at least when $a = 0$.
\end{remark}

Now we prove a \ref{eq.Case2} accumulation lemma.
It will only be applied when $\lambda_* < \kap +1$ and $\lambda = \lambda_*$ (with $\lambda$ as defined in the lemma).
Nonetheless, we state it in more generality, for completeness.

We recall that $\Ext_a$ denotes the $a$-harmonic extension of a polynomial, see \eqref{eq.ext}. 
\begin{lemma}
\label{lem: accumulation pts 2}
In \ref{eq.Case2}, suppose that there exists a sequence of free boundary points $\Sigma_{\kap}^{n-1} \ni X_\ell = (x_\ell, 0) \to 0$ and radii $r_\ell\downarrow 0$ with $|X_\ell|\le r_\ell/2$ such that $\tilde v_{r_\ell} \rightharpoonup q$ in $W^{1,2}(B_1, |y |^a)$ and $(z_\ell,0) = Z_\ell := X_\ell/r_\ell \to Z_\infty$. 
Set
\[
\lambda_{*, X_\ell} := N(0^+, u(X_\ell+\,\cdot\,)-p_{*,X_\ell}),
\]
where $p_{*, X_\ell}$ denotes the first blow-up of $u$ at $X_\ell$.
Let $\be_\ast \in \mathbb{S}^{n} \cap \{y = 0\} \cong \mathbb{S}^{n-1}$ be such that $\be_\ast \perp L_\ast$, and let $q^{\rm even}$ and $q^{\rm odd}$ be the even and odd parts of $q$ with respect to $L_*$, 
\[
q^{\rm even}(X) = \frac12\left[q(X) + q(X-2(\boldsymbol{e}_\ast\cdot X) \boldsymbol{e}_\ast)\right]
\]
and
\[
q^{\rm odd}(X) = \frac12\left[q(X) - q(X-2(\boldsymbol{e}_\ast\cdot X) \boldsymbol{e}_\ast)\right].
\]
Let $\alpha_\kap > 0$ be as in Proposition~\ref{prop.case2} and set $\lambda := \liminf_\ell\{\lambda_{*, X_\ell}\} \ge \kappa +\alpha_\kap$.
Then,
\[
Z_\infty= (z_\infty,0) \in L_*
\]
and
\begin{equation}
\label{eqn: accu even int}
\dashint_{\pa B_\rho} |q^{\rm even}(Z_\infty + X) - c_\infty\Ext_a((\be_\ast \cdot x)^\kappa)|^2|y|^a \le C\rho^{2\lambda+a} \quad\text{for all }\quad \rho \in (0,1/2),
\end{equation}
for some constants $c_\infty$ and $C$ independent of $\rho$. 
Moreover, if $\lambda_*< \kap+1$, then $q^{\rm odd} \equiv 0$. 
If, in addition, $\lambda = \lambda_*$, then $c_\infty = 0$ in \eqref{eqn: accu even int}, and $q$ is invariant in the $Z_\infty$ direction; that is, $q(Z_\infty + X) = q(X)$ for all $X\in \R^{n+1}$. 
\end{lemma}

\begin{proof}
We divide the proof into two steps. 

\smallskip
\noindent{\bf Step 1:} We proceed using the ideas developed to prove \cite[Lemma 3.3]{FS18}. 
Recall that 
\[
p_{*, X_\ell}(X) := \lim_{r\downarrow 0} \frac{u(X_\ell + r X)}{r^\kappa}.
\]
Define
\[
q_\ell(X) := \frac{p_{*, X_\ell}(r_\ell X) - p_*(X_\ell+r_\ell X)}{h_{r_\ell}}
\quad
\text{with}
\quad
h_{r_\ell} := \|v_{r_\ell}\|_{L^2(\pa B_1,|y|^a)}.
\]
By Proposition~\ref{prop.case2} and Proposition~\ref{prop: Almgren}, for all $\rho \in (0,1/2)$,
\begin{equation}
\label{eq.thanksto}
N(\rho r_\ell, u(X_\ell +\,\cdot\,) - p_{*, X_\ell}) \ge \lambda_{*, X_\ell} \ge \kap+\alpha_\kap > \kap, 
\end{equation}
or, equivalently,
\[
\frac{\rho \int_{B_{\rho}} |\nabla \tilde v_{r_\ell}(Z_\ell+\,\cdot\,)-\nabla q_\ell|^2|y|^a}{\int_{\pa B_{\rho}} |\tilde v_{r_\ell}(Z_\ell+\,\cdot\,)-q_\ell|^2|y|^a}\ge \kap+\alpha_\kap
\]
(cf. \eqref{eq: etaAlmgren}).
Furthermore, arguing as in the proof of Lemma~\ref{lem:  accumulation pts}, we find that the family $\{ q_\ell \}_{\ell \in \N}$ has uniformly bounded coefficients.
This time, however, we use that $q_\ell$ is of degree $\kappa$ and $a$-harmonic rather than of degree $\kappa - 1$ and $a$-harmonic.
Indeed, as in Lemma~\ref{lem:  accumulation pts}, suppose not.
Then, dividing by the largest coefficient, we obtain uniformly bounded, $a$-harmonic polynomials $\bar{q}_\ell$ of degree $\kappa$ and the inequality 
\begin{equation}
\label{eq.comb1}
\frac{1}{2} \frac{\int_{B_{1/2}} |\nabla \varepsilon_\ell -\nabla \bar{q}_\ell|^2|y|^a}{\int_{\pa B_{1/2}} |\varepsilon_\ell -\bar{q}_\ell|^2|y|^a}\ge \kap+\alpha_\kap \quad\text{for all}\quad \ell \in \N
\end{equation}
and for some $\varepsilon_\ell \to 0$ in $W^{1,2}(B_{1/2}, |y|^a)$. 
Now notice that $\bar{q}_\ell$ are degree $\kap$ polynomials converging uniformly to some $\bar{q}_\infty$ (up to subsequences).
Also, since the translations that define $q_\ell$ are in $\{y = 0 \}$, $\bar{q}_\ell$ are $a$-harmonic. 
In turn, the limit $\bar{q}_\infty$ is an $a$-harmonic, $\kap$-degree polynomial. 
From \eqref{eq.comb1} and Lemma~\ref{lem: k-deg h poly Almgren bd}, we obtain 
\[
\kap \ge \frac{1}{2} \frac{\int_{B_{1/2}} |\nabla \bar{q}_\infty|^2|y|^a}{\int_{\pa B_{1/2}} |\bar{q}_\infty|^2|y|^a}\ge \kap+\alpha_\kap,
\]
a contradiction, since $\alpha_\kap > 0$. 
Thus, $q_\ell$ converges, up to subsequences, locally uniformly to some $q_\infty$, which is an $a$-harmonic polynomial of degree $\kappa$.
So $|D^\alpha q_\ell (0) |\le C$ for some $C$ independently of $\ell$ for any multiindex $\alpha = (\alpha_1, \dots,\alpha_{n},0)$, and for $|a| \leq \kappa - 1$,
\begin{equation}
\label{eqn: ql and pk top}
D^\alpha q_\ell (0)  =  \frac{r_\ell^{|\alpha|}}{h_{r_\ell}}D^\alpha p_*(X_\ell) =  \frac{r_\ell^{\kap}}{h_{r_\ell}} D^\alpha p_* (Z_\ell).
\end{equation} 
Then, as $h_{r_\ell} = o(r_\ell^\kap)$, we determine that
\[
|D^\alpha p_* (Z_\ell)| = o(1) \to 0 \quad\text{as}\quad \ell \to \infty  
\]
when $|a| \leq \kappa - 1$.
That is, $D^\alpha p_{*} (Z_\infty) = 0$ for $|\alpha|\le \kap - 1$. 
Thanks to Lemma~\ref{lem: equiv},
\[
Z_\infty\in L(p_*)\in L_*.
\]  

Now, by assumption, for some $\boldsymbol{e}_\ell\in \mathbb{S}^{n-1}$ and $c_\ell, c_\ast > 0$,
\[
p_{*, X_\ell}(x, 0) = c_\ell (\boldsymbol{e}_\ell\cdot x)^\kappa \quad\text{and}\quad p_{*}(x, 0) = c_\ast (\boldsymbol{e}_\ast\cdot x)^\kappa.
\]
Also, setting $a_\ell := \boldsymbol{e}_\ast \cdot z_\ell$, we see that
\begin{align*}
q_\ell(x, 0) & = h_{r_\ell}^{-1}\left( p_{*, X_\ell}(r_\ell x, 0) - p_*(x_\ell+r_\ell x, 0) \right)\\
& = r_\ell^\kap h_{r_\ell}^{-1}\left(c_\ell (\boldsymbol{e}_\ell\cdot x)^\kap - c_\ast(\boldsymbol{e}_\ast\cdot(z_\ell+x))^\kap\right)\\
& = r_\ell^\kap h_{r_\ell}^{-1}\bigg(c_\ell (\boldsymbol{e}_\ell\cdot x)^\kap - c_\ast(\boldsymbol{e}_\ast\cdot x)^\kap-c_\ast\kap a_\ell (\boldsymbol{e}_\ast \cdot x)^{\kap-1} \\
&\hspace{3.65cm}+  c_\ast a_\ell^2 \sum_{j = 2}^\kap  
\bigg(\genfrac{}{}{0pt}{}{\kappa}{j}\bigg)\a_\ell^{j-2} (\boldsymbol{e}_\ast \cdot x)^{\kap-j} \bigg).
\end{align*}
Since $p_{*, X_\ell} \to p_*$, we have that $c_\ell \to c_\ast$ and $\boldsymbol{e}_\ell \to \boldsymbol{e}_\ast$ (up to a sign).  
Moreover, as $Z_\ell \to Z_\infty \in L_*$ and $\be_\ast \perp L_\ast$, $a_\ell \to 0$.
Therefore, by the uniform boundedness in $\ell$ of the coefficients of $q_\ell(x, 0)$, we immediately find that 
\begin{align*}
q_\ell(x, 0) & = r^\kap_\ell h_{r_\ell}^{-1}\big(c_\ell (\boldsymbol{e}_\ell\cdot x)^\kap - c_\ast(\boldsymbol{e}_\ast\cdot x)^\kap-c_\ast\kap a_\ell (\boldsymbol{e}_\ast\cdot x)^{\kap-1}\big) + O(a_\ell)\\
& = r^\kap_\ell h_{r_\ell}^{-1}\big((c_\ell -c_\ast)(\boldsymbol{e}_\ell\cdot x)^\kap + c_\ast((\boldsymbol{e}_\ell\cdot x)^\kap - (\boldsymbol{e}_\ast\cdot x)^\kap) -c_\ast\kap a_\ell (\boldsymbol{e}_\ast\cdot x)^{\kap-1}\big) + O(a_\ell).
\end{align*}
Set $\boldsymbol{e}'_\ell := \frac{\boldsymbol{e}_\ell-\boldsymbol{e}_\ast}{|\boldsymbol{e}_\ell-\boldsymbol{e}_\ast|}$.
Then,
\[
\frac{(\boldsymbol{e}_\ell\cdot x)^\kap - (\boldsymbol{e}_\ast\cdot x)^\kap}{|\boldsymbol{e}_\ell-\boldsymbol{e}_\ast|} = \left(\frac{\boldsymbol{e}_\ell-\boldsymbol{e}_\ast}{|\boldsymbol{e}_\ell-\boldsymbol{e}_\ast|}\cdot x \right)\sum_{i = 1}^{\kap-1} (\boldsymbol{e}_\ast\cdot x)^i(\boldsymbol{e}_\ell\cdot x)^{\kap-1-i} =: \left( \boldsymbol{e}'_\ell \cdot x\right) Q_\ell(x).
\]
In addition, as $ \ell \to \infty$,
\[
\boldsymbol{e}'_\ell \to \boldsymbol{e}_\infty' \in \mathbb{S}^{n-1} \quad\text{and}\quad Q_\ell \to (\kap-1)(\boldsymbol{e}_\ast\cdot x)^{\kap-1},
\]
and $\boldsymbol{e}'_\infty \perp \be_\ast$.
Thus, 
\begin{equation}
\label{eq.pinf}
q_\infty(x, 0) = c_1 (\boldsymbol{e}_\ast\cdot x)^\kap + c_2 (\boldsymbol{e}'_\infty\cdot x)(\boldsymbol{e}_\ast\cdot x)^{\kap-1} + c_3 (\boldsymbol{e}_\ast\cdot x)^{\kap-1},
\end{equation}
for some constants $c_1, c_2$, and $c_3$.
So $q_\infty$ vanishes on $L_\ast$.

Thanks to Lemma~\ref{lem.HMon} applied to $u(X_\ell + r_\ell\,\cdot\,)-p_{*, X_\ell}$, denoting $\lambda_\ell := \lambda_{*, X_\ell}$, for all $\rho\in(0,1/2)$,
\[
\frac{1}{\rho^{2\lambda_{\ell}+a}} \dashint_{\pa B_\rho} |\tilde v_{r_\ell} (Z_\ell + \,\cdot\,)-q_\ell|^2|y|^a\le 2^{2\lambda_\ell-a}\dashint_{\pa B_{1/2}}|\tilde v_{r_\ell} (Z_\ell + \,\cdot\,)-q_\ell|^2|y|^a,
\]
from which we deduce that, taking $\ell\to\infty$,
\begin{equation}
\label{eq.qpoly}
\frac{1}{\rho^{2\lambda+a}} \dashint_{\pa B_\rho} |q (Z_\infty + \,\cdot\,)-q_\infty|^2|y|^a\le C\dashint_{\pa B_{1/2}}|q (Z_\infty + \,\cdot\,)-q_\infty|^2|y|^a.
\end{equation}
In turn, because $q_\infty(X) = \Ext_a (q_\infty(x, 0))$ and by \eqref{eq.pinf},
\begin{align*}
\dashint_{\pa B_\rho} |q^{\rm even}(Z_\infty+\,\cdot\,) - \Ext_a(c_1(\be_\ast \cdot x))^\kappa|^2|y|^a & = \dashint_{\pa B_\rho} |(q(Z_\infty+\,\cdot\,)-q_\infty)^{\rm even}|^2|y|^a\\
& \le \dashint_{\pa B_\rho} |q(Z_\infty+\,\cdot\,)-q_\infty|^2|y|^a\\
& \le C\rho^{2\lambda+a}\dashint_{\pa B_{1/2}} |q(Z_\infty+\,\cdot\,)-q_\infty|^2|y|^a,
\end{align*}
from which, taking $c_\infty = c_1$, we find \eqref{eqn: accu even int}. 
(Here, we have used that taking the even part of a function with respect to $L_\ast$, i.e., $f \mapsto f^{\rm even}$, is an orthogonal projection in $L^2(\pa B_\rho, |y|^a)$.)

\smallskip
\noindent{\bf Step 2:}
Let us now show that if $\lambda_*< \kap+1$, then $q^{\rm odd} \equiv 0$; and if, in addition, $\lambda = \lambda_*$, then $c_\infty = 0$ in \eqref{eqn: accu even int}. 
We remark that the fact that $q^{\rm odd}  \equiv 0 $ if $\lambda_*\notin\N$ is independent of Step 1.

If $X \in \R^{n+1} \setminus L_\ast$, then $X-2(\boldsymbol{e}_\ast\cdot X)\boldsymbol{e}_\ast \in \R^{n+1} \setminus L_\ast$; so
\[
L_a q^{\rm odd}(X) = L_a q(X) - L_a q(X-2(\boldsymbol{e}_\ast\cdot X)\boldsymbol{e}_\ast) = 0 \quad\text{for}\quad X \in \R^{n+1} \setminus L_*
\]
(by Proposition~\ref{prop.case2}, $q$ solves the very thin obstacle problem and is $a$-harmonic outside of $L_\ast$). 
On the other hand, if $X \in L_*$, then we have that $X-2(\boldsymbol{e}_\ast \cdot X) \boldsymbol{e}_\ast = X$.
And so, 
\[
L_a q^{\rm odd}(X) = L_a q(X) - L_a q(X) = 0 \quad\text{for}\quad X \in L_\ast.
\]
Therefore, $q^{\rm odd}$ is $a$-harmonic in $\R^{n+1}$. 
This, together with the fact that $q^{\rm odd}$ is $\lambda_*$-homogeneous (again, by Proposition~\ref{prop.case2}) and even in $y$, yields that, by Liouville's theorem for $a$-harmonic functions, $q^{\rm odd}$ is a $\lambda_*$-homogeneous polynomial (see, e.g., \cite[Lemma 2.7]{CSS08}). 
Hence, if $\kap < \lambda_*<\kap+1$, then $q^{\rm odd} \equiv 0$, and $q = q^{\rm even}$. 

Finally, let us now show that if $\lambda = \lambda_* < \kap+1$, then $c_\infty = 0$. 
Let 
\[
q_{Z_\infty}(X) := q(Z_\infty+X) -c_\infty \Ext_a((\be_\ast \cdot x)^\kappa),
\]
which is a solution to the very thin obstacle problem with zero obstacle on $L_\ast$.
If \eqref{eqn: accu even int} holds with $\lambda = \lambda_*$, then from Lemma~\ref{lem.HMon_VTOP 2} and recalling that $q = q^{\rm even}$, we deduce that
\[
N(0^+,q_{Z_\infty}) \geq \lambda_*.
\]
In turn, $q_{Z_\infty}$ is $\lambda_*$-homogeneous.
Indeed, for all $r > 0$, by Lemma~\ref{lem.MonFreq},
\[
\lambda_* \leq N(r,q_{Z_\infty}) \leq N(+\infty,q_{Z_\infty}) = N(+\infty, q(X) - c_\infty\Ext_a((\be_\ast \cdot x)^\kappa)) = \lambda_*.
\]
The penultimate equality holds since the limit as $r \to + \infty$ of Almgren's frequency function is independent of the point at which it is centered, and the last equality holds because $q$ is $\lambda_*$-homogeneous with $\lambda_* >\kap$, and thus $q$ out-scales a $\kap$-homogeneous polynomial.

Since $q_{Z_\infty}$ is $\lambda_*$-homogeneous, we deduce that
\begin{equation}
\label{eqn: alt form q}
q(X+ Z_\infty) = \frac{q(X) + q(X + 2Z_\infty)}{2}.
\end{equation}
To see this, first, observe that
\[
\begin{split}
\tau^{\lambda_*}q(X + \tau^{-1}Z_\infty) = q(\tau X + Z_\infty) = \tau^{\lambda_*}q_{Z_\infty}( X) + \tau^{\kappa}c_\infty\Ext_a((\be_\ast \cdot x)^\kappa),
\end{split}
\]
for all $\tau > 0$.
The first equality follows from the $\lambda_*$-homogeneity of $q$, while the second follows from the $\lambda_*$-homogeneity of $q_{Z_\infty}$.
So
\[
q(X + \tau^{-1}Z_\infty) - q_{Z_\infty}( X) = \tau^{\kappa-\lambda_*}c_\infty\Ext_a((\be_\ast \cdot x)^\kappa),
\]
for all $\tau > 0$.
Taking the limit as $\tau \to + \infty$ yields
\begin{equation}
\label{eqn: qzinfty = q}
q_{Z_\infty} = q.
\end{equation}
(Recall, $\lambda_* > \kappa$.)
That is,
\begin{equation}
\label{eq.transinv}
c_\infty\Ext_a((\be_\ast \cdot x)^\kappa) = q(X + Z_\infty) - q(X).
\end{equation}
And because $e_\ast \perp Z_\infty$,
\[
c_\infty\Ext_a((\be_\ast \cdot x)^\kappa) = q(X) - q(X-Z_\infty).
\]
Hence, \eqref{eqn: alt form q} holds, as desired.

To conclude, from the $\lambda_*$-homogeneity of $q$ and \eqref{eq.transinv}, observe that
\[
\pa_{\be_\ast}^{(\kappa)} q(Z_\infty) = \kappa! c_\infty.
\]
On the other hand, \eqref{eqn: alt form q} implies
\[
\pa_{\be_\ast}^{(\kappa)} q(Z_\infty) = \frac{\pa_{\be_\ast}^{(\kappa)} q(2Z_\infty)}{2} = 2^{\lambda_* - \kappa -1}\pa_{\be_\ast}^{(\kappa)} q(Z_\infty).
\]
Thus,
\[
(1-2^{\lambda_* - \kappa -1})\kappa!c_\infty = 0.
\]
Yet $\lambda_* - \kappa - 1 \neq 0$, by assumption. Consequently, $c_\infty = 0$.

Therefore, we have that if $\lambda_* < \kap +1$ and $\lambda = \lambda_*$, 
\begin{equation}
\label{eqn: accu even int 2}
\frac{1}{\rho^{n+2\lambda_*+a}}\int_{\pa B_\rho} |q(Z_\infty + \,\cdot\,)|^2|y|^a \le C.
\end{equation}

By Lemma~\ref{lem.HMon_VTOP 2}, $N(0^+, q(Z_\infty + \,\cdot\,)) = \lambda_*$ as $q(Z_\infty +\,\cdot\,) $ is a solution to the very thin obstacle problem.
On the other hand, since $q$ is $\lambda_*$-homogeneous, $N(+\infty, q(Z_\infty+ \,\cdot\,) ) = \lambda_*$, and from the monotonicity formula in Lemma~\ref{lem.MonFreq}, we deduce that $q(Z_\infty+ \,\cdot\,)$ is $\lambda_*$-homogeneous. 
Then,
\[
q(X+Z_\infty) = \tau^{\lambda_*}q(\tau^{-1}X + Z_\infty)= q(X + \tau Z_\infty) \quad\text{for all}\quad X\in \R^n \text{ and } \tau > 0;
\]
that is, $q$ is invariant in the $Z_\infty$ direction. 
\end{proof}

We close this section with a pair of remarks and a \ref{eq.Case2} version of Lemma~\ref{lem: Sigma0k is isolated}.
The observations made in these remarks are crucial to our analysis of when we can produce the next term in the expansion of $u$ around a singular point. 

\begin{remark}
\label{rem.onlypoly}
In Lemma~\ref{lem: accumulation pts 2}, as in Lemma~\ref{lem: accumulation pts}, if $q$ is an $a$-harmonic, $(\kap+1)$-homogeneous polynomial and $\lambda = \lambda_* = \kap+1$, we also have that 
\begin{equation}
\label{eq.qholds}
D^\alpha q(Z_\infty) = 0 \quad\text{for all}\quad \alpha = (\alpha', 0)\text{ and }|\alpha|\le \kap -2. 
\end{equation}
Indeed, observe that \eqref{eq.qpoly} becomes 
\[
\dashint_{\pa B_\rho} |q (Z_\infty + \,\cdot\,)-q_\infty|^2|y|^a\le C\rho^{2(\kap+1)+a},
\]
for all $\rho \in (0,1/2)$.
Hence, the polynomial $q(Z_\infty+\,\cdot\,) + q_\infty$ is only made up of monomials of degree $\kap+1$.
In particular, since $q$ is $(\kappa + 1)$-homogeneous and $q_\infty$ is of degree $\kappa$, $q(Z_\infty+\,\cdot\,) + q_\infty$ is a $(\kappa + 1)$-homogeneous polynomial.
So, for all multiindices $|\alpha| \leq \kappa$,
\[
D^\alpha q(Z_\infty) = D^\alpha q_\infty(0),
\]
which, by \eqref{eq.pinf}, implies \eqref{eq.qholds} holds, as desired.
\end{remark}

\begin{remark}
\label{rem.qeven}
The last part of the proof of Lemma~\ref{lem: accumulation pts 2} fails to show that $q^{\rm even}$ is invariant in the $Z_\infty$ direction when $\lambda = \lambda_* = \kap+1$.
In this case, however, we find that 
\[
q^{\rm e}_{Z_\infty}(X) := q^{\rm even}(Z_\infty+X) -c_\infty \Ext_a((\be_\ast \cdot x)^\kappa),
\]
is $\lambda_*$-homogeneous.
Hence,
\[
q^{\rm even}(X + \tau^{-1}Z_\infty) - q^{\rm e}_{Z_\infty}(X) = \tau^{-1}c_\infty\Ext_a((\be_\ast \cdot x)^\kappa),
\]
as before, for all $\tau > 0$. 
By letting $\tau \to \infty$, we deduce that $q^{\rm e}_{Z_\infty} = q^{\rm even}$, which substituting back yields 
\begin{equation}
\label{eq.holds}
q^{\rm even}(X + \tau Z_\infty) = q^{\rm even}(X)  +\tau c_\infty\Ext_a((\be_\ast \cdot x)^\kappa) \quad\text{for all}\quad X\in \R^n \text{ and } \tau > 0.
\end{equation}
Moreover, considering $X - Z_\infty$, we find that $q^{\rm even}(X) = q^{\rm even}(X-Z_\infty)+ c_\infty \Ext_a((\be_\ast \cdot x)^\kappa)$ for all $X \in \R^{n+1}$.
Hence, \eqref{eq.holds} for all $\tau \in \R$.
\end{remark}

\begin{lemma}
\label{lem: Sigma0k is isolated 2}
Suppose \ref{eq.Case2} holds. 
Then, $0$ is an isolated point of $\Sigma_{\geq \kap}$.
\end{lemma}

\begin{proof}
The proof is identical to that of Lemma~\ref{lem: Sigma0k is isolated}, but using Lemma~\ref{lem: accumulation pts 2} instead of Lemma~\ref{lem: accumulation pts}.
\end{proof}

\section{The Size of the Anomalous Set}
\label{sec:size}

The goal of this section is to further stratify the set of singular points and prove Proposition~\ref{prop.main2} and Remark~\ref{rem.zerodim}.
Proposition~\ref{prop.main2} (and Remark~\ref{rem.zerodim}) is a statement regarding the Hausdorff dimension of the anomalous singular points of order 2 and $(n-1)$-dimensional singular points of arbitrary order (i.e., singular points whose first blow-up has $(n-1)$-dimensional spine and is $\kappa$-homogeneous).
As such, let us recall the definition of anomalous singular points, and generic singular points, as well as some measure theoretic facts.

\subsection{Singular Points Revisited} 

Given the set of singular points of order $\kappa$ and dimension $m$ (i.e., whose first blow-up has $m$-dimensional spine), we recall that the anomalous points are those for which the homogeneity of second blow-ups is strictly less than $\kappa + 1$:
\[
\Sigma_\kap^{m,{\rm a}} := \{X_\circ \in \Sigma_\kap^m : N(0^+, u(X_\circ+\,\cdot\,) - p_{*, X_\circ}) < \kap+1\}.
\]
The generic points, on the other hand, are those for which the homogeneity of second blow-ups jumps by at least one:
\[
\Sigma_\kap^{m,{\rm g}} := \{X_\circ \in \Sigma_\kap^m : N(0^+, u(X_\circ+\,\cdot\,) - p_{*, X_\circ}) \geq \kap+1\}.
\]
In turn, $\Sigma_\kap^{m,{\rm g}} = \Sigma_\kap^m \setminus \Sigma_\kap^{m,{\rm a}}$.

While Proposition~\ref{prop.main2} and Remark~\ref{rem.zerodim} ignore higher order (greater than two) and lower dimensional (less than $n-1$) singular points, our analysis, in a sense, does not.
In particular, our results rely on the alignment of the nodal set and the spine of first blow-ups at anomalous singular points.
And so, we set
\[
\tilde \Sigma_\kap := \{X_\circ \in \Sigma_\kap : \mathcal{N}(p_{*, X_\circ}) = L(p_{*, X_\circ}) \}
\]
and define
\[
\tilde \Sigma_\kap^{m,{\rm a}} := \tilde \Sigma_\kap \cap \Sigma_\kap^{m,{\rm a}}  \quad\text{and}\quad \tilde \Sigma_\kap^{m,{\rm g}} := \tilde \Sigma_\kap \cap \Sigma_\kap^{m,{\rm g}}.
\]

\begin{remark}
A key consequence of the coincidence of $\mathcal{N}(p_{*, X_\circ})$ and $L(p_{*, X_\circ})$ is that $p_{*, X_\circ}|_{\R^n \times \{ 0 \}}$ is positive away from its spine, i.e., if $\mathcal{N}(p_{*, X_\circ}) = L(p_{*, X_\circ})$, then $p_{*, X_\circ}(x,0) > 0$ for any $x\in \R^{n}$ such that $x \notin L(p_{*, X_\circ})$.
\end{remark}
  
\begin{remark}
Notice that if $m = n-1$ or if $\kap = 2$, then $\tilde \Sigma_\kap^{m,{\rm a}} = \Sigma_\kap^{m,{\rm a}}$.
Moreover, if the spine and the nodal set of the first blow-up at anomalous points coincide, \ref{eq.Case1} and \ref{eq.Case2} exhaust all possibilities (cf. Remark~\ref{rmk: all poss}).
\end{remark}

\subsection{Some Measure Theory}

Given $\beta > 0$ and $\delta \in (0, \infty]$, we define the Hausdorff premeasures
\[
\mathcal{H}_\delta^\beta(E) := \inf\left\{\sum_i \omega_\beta\left(\frac{\diam E_i }{2}\right)^\beta : E \subset \bigcup_i E_i \,\text{ with } \diam E_i < \delta\right\},
\]
so that the $\beta$-dimensional Hausdorff measure of a set $E$ is
\[
\mathcal{H}^\beta(E) := \lim_{\delta\downarrow 0}\mathcal{H}^\beta_\delta(E). 
\]
(Here, $\omega_\beta$ is the volume of the $\beta$-dimensional unit ball.)
The Hausdorff dimension of a set can then be defined as 
\begin{equation}
\label{eq.dimdef}
\dim_{\mathcal{H}} E := \inf\{\beta > 0 : \mathcal{H}^\beta_\infty(E) = 0\}.
\end{equation}
(See, e.g., \cite{Sim83}.)

\begin{lemma}
\label{lem.lem35}
Let $E\subset \R^{n+1}$ be a set with $\mathcal{H}^\beta_\infty(E) > 0$ for some $\beta \in (0, n+1]$.
The following holds:
\begin{enumerate}[(i)]
\item For $\mathcal{H}^\beta$-almost every point $X_\circ\in E$, there is a sequence $r_j \downarrow 0$ such that 
\begin{equation}
\label{eq.3.18}
\lim_{k \to \infty} \frac{\mathcal{H}^\beta_\infty(E\cap B_{r_j}(X_\circ))}{r_j^\beta}\ge c_{n,\beta} > 0,
\end{equation}
where $c_{n,\beta}$ is a constant depending only on $n$ and $\beta$. 
We call these points ``density points''.
\item Assume that $0 \in E$ is a ``density point'', let $r_j \downarrow 0$ be a sequence along which \eqref{eq.3.18} holds, and define the ``accumulation set'' for $E$ at 0 as 
\[
\mathcal{A}_E := \{Z \in \overline{B_{1/2}} : \exists \{Z_\ell\}_{\ell \in \N}, \{j_\ell\}_{\ell \in \N} \text{ such that } Z_\ell \in r^{-1}_{j_\ell} E\cap B_{1/2} \text{ and } Z_\ell \to Z\}.
\]
Then,
\[
\mathcal{H}^\beta_\infty(\mathcal{A}) >0.
\]
\end{enumerate}
\end{lemma}
\begin{proof}
See \cite[Lemma 3.5]{FS18}.
\end{proof}

In order to prove that anomalous points form a small set in \ref{eq.Case2}, we will focus on ``almost continuity'' points of the frequency, in the spirit of \cite{FRS19}. 
More precisely, as shown in \cite{FRS19}, points where the frequency is discontinuous along ``too many'' sets of converging sequences have small Hausdorff measure. 
This fact, which plays a crucial role in \cite{FRS19}, allows us to use Lemma~\ref{lem: accumulation pts 2} to show that second blow-ups are translation invariant in directions of ``almost continuity'' of the frequency.

\begin{lemma}
\label{lem: almost cty}
Let $E \subset \R^{n+1}$ and $f : E \to \R$ be any function.
The set
\[
\{ X \in E : \text{for all $\{ X_\ell \}_{\ell \in \N}$ such that $X_\ell \neq X$ and $X_\ell \to X$, $f(X_\ell) \not\to f(X)$} \}
\]
is at most countable.
\end{lemma}

\begin{proof}
If $X_\circ$ is an element of the set in question, then $(X_\circ,f(X_\circ))$ is an isolated point of $\{ (X,f(X)) : X \in E \}$.
Since collection of isolated points of any subset of $\R^{n+2}$ is at most countable, the lemma follows.
\end{proof}

\begin{lemma}
\label{lem.E}
Let $E \subset\R^n$ and $f : E \to [0, \infty)$ be any function. 
Suppose for any $x\in E$ and any $\varepsilon > 0$, there exists a $\rho > 0$ such that for all $r \in (0, \rho)$,
\begin{equation}
\label{eq.condE}
E \cap \overline{B_r(x)} \cap f^{-1}\left([f(x) - \rho, f(x)+\rho]\right)\subset \{ z : \dist(z, \Pi_{x, r}) \le \varepsilon r \}
\end{equation}
for some $m$-dimensional plane $\Pi_{x, r}$ passing through $x$, possibly depending on $r$. 
Then, 
\[
\dim_{\mathcal{H}} E\le m.
\]
\end{lemma}

\begin{proof}
See \cite[Proposition 7.3]{FRS19}.
\end{proof}

\subsection{Proofs of Proposition~\ref{prop.main2} and Remark~\ref{rem.zerodim}}

We now move to the goal of this section.
We start with a set of results which pertain to \ref{eq.Case1}.

\begin{proposition}
\label{prop: spine is nodal set}
Assume $n \ge 2$.
\begin{enumerate}[(i)]
\item If $a \in [0,1)$, $\dim_{\mathcal{H}} \tilde \Sigma_\kap^{m,{\rm a}} \le m-1$ for any $1\le m \le n-1$.
\item If $a \in (-1,0)$, $\dim_{\mathcal{H}} \tilde \Sigma_\kap^{m,{\rm a}} \le m-1$ for any $1\le m \le n-2$.
\end{enumerate}
\end{proposition}

\begin{proof}
The first part of the proof follows the steps of \cite[Lemma 3.6]{FS18}. 
Set $\Sigma:= \tilde \Sigma_\kap^{m,{\rm a}}$.

\smallskip
\noindent{\bf Step 1:} 
We argue by contradiction. 
Suppose that $\mathcal{H}^\beta_\infty(\Sigma) > 0$ for some $\beta > m-1$. 
Then, there is a point $X_\circ\in \Sigma$ and sequence $r_j \downarrow 0$ such that
\[
\frac{\mathcal{H}^\beta_\infty(\Sigma\cap B_{r_j}(X_\circ))}{r_j^{\beta}} \ge c_{n,\beta} > 0. 
\]
Up to a translation, $X_\circ = 0$. 
By definition, we have that 
\[
\lambda_* := N(0^+,v_\ast) < \kap +1,
\]
and that, after extracting a subsequence, 
\[
\tilde v_{r_j} \to q\quad\text{in}\quad L^2(B_1, |y|^a). 
\]
Additionally, from Lemma~\ref{lem.lem35}(ii), we have that the accumulation set $\mathcal{A}:= \mathcal{A}_\Sigma$ satisfies
\begin{equation}
\label{eq.contacc}
\mathcal{H}_\infty^\beta(\mathcal{A})  >0.
\end{equation}

By the definition of $\mathcal{A}$, $Z\in \mathcal{A}$ if there are sequences $X_\ell \in \Sigma$ and $r_{j_\ell} \downarrow 0$ such that $|X_\ell|\le r_{j_\ell}$ and $X_\ell/r_{j_\ell} \to Z$. 
Thus, $X_\ell/2r_{j_\ell}\to Z/2$, and by Lemma~\ref{lem: accumulation pts} (notice that we are in \ref{eq.Case1}), $Z \in L_\ast$ and $Z \in T(q) := \{ X = (x, 0)\in\R^{n+1} : q(x+\,\cdot\,, 0)-q(\,\cdot\,, 0) \text{ is invariant under } L_\ast\}$. 
That is, 
\[
\mathcal{A} \subset \overline{B_1} \cap L_\ast \cap T(q).
\]
Notice that by assumption, $q$ is $\kap$-homogeneous and $\dim_\mathcal{H} L_\ast = m$. 
Also, $T(q)$ is a linear space. 
We will, therefore, reach a contradiction if we can show that $L_\ast \not\equiv T(q)\cap L_\ast$; since then, $\dim_{\mathcal{H}} L_\ast\cap T(q) \le m-1$, which contradicts \eqref{eq.contacc}.

\smallskip 
\noindent{\bf Step 2:} 
Let $\bar p := p|_{\R^n \times \{ 0 \}}$ for any $a$-harmonic, even in $y$ polynomial $p$, and recall that $\bar p$ uniquely determines $p$ (see the lines after \eqref{eq.ext}).
Suppose, again, to the contrary, that $L_\ast \equiv T(q)\cap L_\ast$. 
After a change of variables, since $L_*$ has dimension $m$, we can assume that $\bar p_* = \bar  p_\ast(x_1,\dots,x_l)$ for $l = n-m$.
Set $x^l = (x_1,\dots,x_l)$.  
Notice that $L_\ast = \{(0^l,x^m,0) : x^m \in \R^{m} \}$, where $0^l \in \R^l$ denotes the vector $0$ in $l$ dimensions. 
The inclusion $L_\ast \subset T(q) $ implies that $\bar q((0^l,x^m) +\,\cdot\,) - \bar q$ can only depend on $x^l$ for any $x^m \in \R^{m}$. 
This, together with the homogeneity of $q$, directly yields that 
\[
\bar q(x) = q_1(x^l) + \sum_{j = l+1}^{n} q_j(x^l) x_j =: \bar q_1(x^l) + \bar q_2(x),
\]
where $q_1$ and $q_j$ are $\kap$-homogeneous and $(\kap-1)$-homogeneous polynomial respectively depending only on $x_1,\dots,x_l$.

Now recall Lemma~\ref{lem: L2 sign}:
\begin{equation}
\label{eq.1}
\int_{B_1} qp|y|^a  \le 0 \quad\text{for all}\quad p \in \mathscr{P}_\kappa.
\end{equation}
Moreover, \eqref{eq.1} is an equality if $p = p_\ast$ (this is \eqref{eq.ineq1}). 
Notice, first, that (recall \eqref{eq.ext})
\begin{equation}
\label{eq.2}
\int_{B_1} \Ext_a(\bar q_1)\Ext_a(\bar q_2)|y|^a = 0. 
\end{equation}
Indeed, $\bar q_1$ does not depend on $x_{l+1},\dots,x_{n}$, whereas the terms of $\bar q_2$ are sums of linear terms in one of $x_{l+1},\dots,x_{n}$; thus, odd in one of the last variables. 

Since $0\in \Sigma = \tilde \Sigma_\kap^{m,{\rm a}}$, $\bar p_*(x^l,0^m) > 0$ for all $x^l \in \R^l \setminus \{0\}$. 
In particular, we can choose $C \gg 1$ such that $C\bar p_*+\bar q_1 \ge 0$ ($\bar p_*$ and $\bar q_1$ have the same homogeneity and depend on the same variables). 
That is, $\Ext_a(C\bar p_*+\bar q_1) = Cp_* + \Ext_a(\bar q_1) \in \mathscr{P}_\kappa$, from which it follows that
\[
0\ge \int_{B_1} (Cp_*+\Ext_a(\bar q_1))q|y|^a = \int_{B_1} \Ext_a(\bar q_1)^2|y|^a,
\]
using the equality in \eqref{eq.1} and \eqref{eq.2}. 
Hence,
\[
\bar q_1 \equiv 0.
\] 

Finally, fix $l+1\le j\le n$, and take $\bar p_j := C(|x^l|^\kappa + x_j^\kap) + q_j(x^l)x_j$ for some $C \gg 1$ so that $\bar p_j \ge 0$.
(The fact that such a constant $C > 0$ exists is straight-forward. 
Indeed, it suffices to show that $x_1^{\kap}+x_2^{\kap}-x_1^{\kap-1}x_2 \ge 0$, which after dividing by $x_2^\kap$ is analogous to showing that $\xi^\kap \ge \xi -1$ for all $\xi\in \R$; this is immediate.)
Arguing as before, by odd/even symmetry, we find that
\[
\int_{B_1} \Ext_a(q_jx_j)\Ext_a(q_ix_i)|y|^a = 0 \quad\text{for all}\quad l+1\le i \neq j\le n
\]
and
\[
\int_{B_1} \Ext_a(|x^l|^\kappa + x_j^\kappa)\Ext_a(q_ix_i)|y|^a = 0 \quad\text{for all}\quad l+1\le i, j\le n.
\]
And so, as $\bar q_1 \equiv 0$ and $\Ext_a(\bar p_j) \in \mathscr{P}_\kappa$,
\[
0 \geq \int_{B_1} \Ext_a(\bar p_j)q|y|^a = \int_{B_1} \Ext_a(q_jx_j)^2|y|^a,
\]
which is only true if $q_j  \equiv 0$. 
Because $j$ was fixed arbitrarily, we deduce that
\[
\bar q_2 \equiv 0.
\]
A contradiction.
\end{proof}

\begin{lemma}
\label{lem.sigmaempty}
Let $n \ge 2$. Then, $\tilde \Sigma_\kap^{0,{\rm a}}$ is empty.
\end{lemma}

\begin{proof}
Suppose $0 \in \tilde \Sigma_\kap^{0,{\rm a}}$.
Then, $\mathcal{N}_\ast = L_\ast = \{0\}$ and $\bar p_\ast := p_\ast|_{\R^n \times \{ 0\}} > 0$ outside of the origin.
Hence, there exists a $C \gg 1$ so that $\Ext_a(C\bar p_\ast + \bar q) = Cp_\ast + q \in \mathscr{P}_\kap$ (cf. Step 2 of the proof of Proposition~\ref{prop: spine is nodal set}).
Here, $\bar q$ is the restriction of any second blow-up of $u$ at $0$.
So, by \eqref{eq.ineq2} and \eqref{eq.ineq1}, we find that
\[
0 \geq \int_{\pa B_1} q(Cp_\ast + q)|y|^a = \int_{\pa B_1} q^2|y|^a,
\]
which cannot be: $q \not\equiv 0$.
\end{proof}

\begin{lemma}
\label{lem.sigma1}
Let $n \ge 2$ and $a \in [0,1)$ or $n \ge 3$ and $a \in (-1,1)$. 
Then, $\tilde \Sigma_\kap^{1,{\rm a}}$ is isolated in $\Sigma_{\ge\kap}$.
\end{lemma}
\begin{proof}
Suppose not and assume that $0\in \tilde \Sigma_\kap^{1,{\rm a}}$.
Then, there exists a sequence $X_\ell \in \Sigma_{\ge \kap}$ with $X_\ell\to 0$. 
Let $r_\ell := 2|X_\ell|$, and notice that $\dim_\mathcal{H} L_* = \dim_\mathcal{H} \{p_* = 0\} = 1$, by assumption. 
On the other hand, up to a subsequence, we can assume that $\tilde v_{ r_\ell} \to q$ in $L^2(B_1, |y|^a)$, which is $\kap$-homogeneous. 

The proof now follows exactly as in Step 2 of the proof of Proposition~\ref{prop: spine is nodal set}.
\end{proof}

\begin{remark}
\label{rem.counterexample}
In all of the above results, Proposition~\ref{prop: spine is nodal set}, Lemma~\ref{lem.sigmaempty}, and Lemma~\ref{lem.sigma1}, the coincidence of the nodal set and the spine of $p_*$ is crucial.
To illustrate how much, let us consider $\Sigma_4^{0,{\rm a}}$, which we would like to say is empty.
(Notice that $\Sigma_2^{0,{\rm a}}$ is empty; in this case, the nodal set and spine of the first blow-up at any point are aligned.)

In Proposition~\ref{prop: spine is nodal set}, in order to rule out a $\kappa$-homogeneous, $a$-harmonic $q$ as a second blow-up at anomalous points, we have used Lemmas~\ref{lem: accumulation pts} and \ref{lem: L2 sign}. 
Since we are dealing with $\Sigma_4^{0,{\rm a}}$, Lemma~\ref{lem: accumulation pts} provides no new information on $q$. 
Also, Lemma~\ref{lem: L2 sign} is too weak to rule out that $q$ is a $4$-homogeneous, harmonic (assume $a = 0$, for simplicity) polynomial.

Indeed, in $\R^3 = \{ (x_1,x_2,y) \}$, consider the harmonic extensions
\[
p_\ast = \Ext_0(x_1^2x_2^2)
\]
and
\[
q = \Ext_0 \left(b x_1^4 -\left(\frac{11}{24} + b\right) x_2^4 +x_1^2x_2^2\right) \quad\text{with}\quad b\in\left[-\frac13,-\frac18\right].
\]
Notice that the spine and nodal set of $p_\ast$ are different:
\[
L_\ast = \{ (0,0) \}\quad\text{while}\quad\mathcal{N}_\ast = \{x_1 = 0\}\cup\{x_2 = 0\}.
\]
Moreover, by direct (but tedious) computations, the pair $(p_\ast,q)$ is such that
\[
\langle q, p_\ast \rangle_0 = 0\qquad\text{and}\qquad \langle q, p \rangle_0 \leq 0 \quad\text{for all}\quad p \in \mathscr{P}_4.
\]
In turn, this pair could be a first and second blow-up pair at $0$ for a solution $u$ for which $\Sigma_4^0(u) = \{ 0 \}$, leaving open the possibility that $\Sigma_4^{0, {\rm a}}$ is not only not lower dimensional, but all of $\Sigma_4^{0}$.
\end{remark}

Now we study of the size of the anomalous set in \ref{eq.Case2}.

\begin{lemma}
\label{lem.discreten-1}
Let $n = 2$ and $a \in (-1,0)$. 
Then, $\Sigma_\kap^{1,{\rm a}}$ is at most countable. 
\end{lemma}

\begin{proof}
Assume that $0 \in \Sigma_\kap^{1,{\rm a}}$, which holds up to a translation, and that there exists a sequence $X_\ell \in \Sigma_\kap^{1}$ such that $X_\ell \to 0$ and $N(0^+,u(X_\ell + \,\cdot\,) - p_{\ast,X_\ell}) =: \lambda_{\ast, X_\ell} \to \lambda_* = N(0^+,v_{\ast})$. 
By Proposition~\ref{prop.case2} and the definition of anomalous set, $\Sigma_\kap^{1,{\rm a}}$, we have that 
\begin{equation}
\label{eq.contradtop}
\lambda_* \in [\kap+\alpha_\kap, \kap+1).
\end{equation}
Moreover, up to a subsequence, if $r_\ell := 2|X_\ell|$, 
\[
\tilde v_{r_\ell} \rightharpoonup q \quad \text{in}\quad W^{1,2}(B_1, |y |^a)
\]
where $q$ is a global $\lambda_*$-homogeneous solution to the very thin obstacle problem with zero obstacle on $L_*$. 
In addition, by Lemma~\ref{lem: accumulation pts 2}, $X_\ell/r_\ell \to Z_\infty \in L_*\cap \pa B_{1/2}$, $q=q^{\rm even}$ ($\lambda_* < \kap+1$ by assumption, forcing $q^{\rm odd} \equiv 0$), and $q$ is invariant in the $Z_\infty$ direction (i.e., in the $L_*$ direction). That is, $q$ is two-dimensional.
So by Lemma~\ref{lem: 2d class}, since $\lambda_* > 2$, $q$ is a polynomial and, in particular, $\lambda_{*} \ge \kap+1$.
But this contradicts \eqref{eq.contradtop}. 

In turn, by Lemma~\ref{lem: almost cty} applied to $E = \Sigma_\kap^{n-1}$ and $f(X) = N(0^+,u(X + \,\cdot\,) - p_{\ast,X})$, we conclude.
\end{proof}

\begin{proposition}
\label{prop: top anom is lower d}
Let $n \geq 3$ and $a \in (-1,0)$. 
Then, $\dim_{\mathcal{H}} \Sigma_\kap^{n-1, {\rm a}} \le n-2$. 
\end{proposition}

\begin{proof}
Let $\Sigma := \Sigma_\kap^{n-1, {\rm a}}$. 
We will show that $\Sigma$ fulfills the hypotheses of Lemma~\ref{lem.E} with $m = n-2$ and 
\begin{equation}
\label{eq.f}
f(X) := 
\begin{cases}
N(0^+, u(X+\,\cdot\,) - p_{*, X}) &\text{if } X\in \R^n\times\{0\}\\
0 & \textrm{ otherwise}.
\end{cases}
\end{equation}
Then, by Lemma~\ref{lem.E}, $\dim_{\mathcal{H}} \Sigma \le n-2$. 

Suppose, to the contrary, that \eqref{eq.condE} does not hold, that is, in particular, there exists some $X_\circ \in \Sigma$, $\varepsilon_\circ > 0$, $\rho_\ell\downarrow 0$ as $\ell \to \infty$, and $0<r_\ell <\rho_\ell$ for which
\begin{equation}
\label{eq.condE_c}
\inf_{\Pi\in \Pi_{X_\circ}} \dist \left(\Sigma\cap \overline{B_{r_\ell}(X_\circ)} \cap f^{-1}\left([f(X_\circ) - \rho_\ell, f(X_\circ)+\rho_\ell]\right), \Pi \right) \ge \varepsilon_\circ r_\ell,
\end{equation}
where $\Pi_{X_\circ}$ denotes the set of $(n-2)$-dimensional planes passing through $X_\circ$, and we denote $\dist(A, B) := \sup_{x\in A}\inf_{y \in B}|x-y|$. 
Up to a translation, assume $X_\circ = 0$.

We claim (and prove later) that thanks to \eqref{eq.condE_c}, for any $\ell\in \N$, there exist $n-1$ points 
\begin{equation}
\label{eq.subseqX}
X_1^\ell,\dots,X_{n-1}^\ell \in \Sigma\cap \overline{B_{r_\ell}} \cap f^{-1}\left([f(0) - \rho_\ell, f(0)+\rho_\ell]\right)
\end{equation}
such that 
\begin{equation}
\label{eq.subseqY}
|Y_1^\ell\wedge \dots\wedge Y_{n-1}^\ell| \ge \varepsilon_\circ^{n-1} \quad\text{where}\quad Y^\ell_i := r_\ell^{-1} X_i^\ell \in B_1\setminus B_{\varepsilon_\circ},
\end{equation}
for all $i\in \{1,\dots,n-1\}$. 
In particular, up to subsequences, $Y_i^\ell\to Y_i^\infty\in B_1\setminus B_{\varepsilon_\circ}$ for all $i\in \{1,\dots,n-1\}$, and passing to the limit, in \eqref{eq.subseqY}, 
\begin{equation}
\label{eq.indep}
|Y_1^\infty\wedge \dots\wedge Y_{n-1}^\infty| \ge \varepsilon_\circ^{n-1}.
\end{equation}
On the other hand, from \eqref{eq.subseqX}, we have that 
\begin{equation}
\label{eq.limf}
f(X_i^\ell)\to f(0) \quad\textrm{as}\quad \ell\to \infty\quad\text{for all}\quad i \in \{1,\dots,n-1\}.
\end{equation}

Up to subsequences, by Proposition~\ref{prop.case2},
\[
\frac{v_{r_\ell}}{\|v_{r_\ell}\|_{L^2(\pa B_1,|y|^a)}}\rightharpoonup q \quad\text{in}\quad W^{1,2}(B_1, |y |^a),
\]
and $q$ is some $\lambda_*$-homogeneous solution to the very thin obstacle problem. 
Moreover, since $0 \in \Sigma$,
\begin{equation}
\label{eq.contq}
N(0^+, v_\ast) \in [\kap+\alpha_\kap, \kap+1)
\end{equation}
Thus, for each $i\in \{1,\dots,n-1\}$, we can apply Lemma~\ref{lem: accumulation pts 2} with the sequence of radii $2r_\ell$. 
By definition and using the notation of Lemma~\ref{lem: accumulation pts 2}, we are in the case $\lambda_* < \kap +1$ (see \eqref{eq.contq}) and $\lambda = \lambda_*$ (thanks to \eqref{eq.limf}). 
So by Lemma~\ref{lem: accumulation pts 2}, $q$ is invariant in the directions $Y^\infty_i$ for all $i\in \{1,\dots,n-1\}$. 

From \eqref{eq.indep}, the set $\{Y^\infty_1,\dots,Y^\infty_{n-1}\}_{i\in \N} \subset L_*\times\{0\}$ is linearly independent. 
That is, $q$ is independent of the $n-1$ directions determined by this linearly independent set.
Therefore, it is a two-dimensional solution to the very thin obstacle problem. 
Hence, by Lemma~\ref{lem:  2d class}, $q$ is a polynomial, and $N(0^+, q) \ge \kap+1$.
But this runs contrary to $0$ living in $\Sigma = \Sigma_{\kap}^{n-1, {\rm a}}$, \eqref{eq.contq}. 

In turn, $\Sigma$ meets the hypotheses of Lemma~\ref{lem.E} with $m = n-2$ and $f$ as above.
And so, $\dim_{\mathcal{H}} \Sigma \le n-2$.

We now prove \eqref{eq.subseqX} and \eqref{eq.subseqY}.

After a dilation, it suffices to show that if $S\subset B_1$ is such that 
\[
\inf_{\Pi\in\Pi_0}\dist(S, \Pi) \ge \varepsilon,
\]
then there exist points $X_1,\dots,X_{n-1}\in S$ such that
\[
|X_1\wedge \dots\wedge X_{n-1}|\ge \varepsilon^{n-1}.
\]
But this follows from a simple construction.

Take any $X_1\in S\cap(B_1\setminus B_\vep)$, and let $\be_1 := X_1/|X_1|$ be the first element of our orthogonal $n-1$ dimensional basis on which we will compute the $(n-1)$-determinant.
Notice $X_1 = a_{1,1}\be_1$ for some $|a_{1,1}|\ge \vep$. 
Now take any $(n-2)$-dimensional plane passing through $X_1$ (and $0$), $\Pi^X_1$, take any $X_2 \in S\cap(B_1\setminus B_\vep)$ that is $\vep$ far from $\Pi_1^X$, and choose $\be_2 \perp \be_1$ and such that $\Span(\be_1, \be_2) = \Span(X_1, X_2)$. 
Then, $X_2 = a_{2,1} \be_1 + a_{2,2} \be_2$, and since $X_2$ is $\vep$ far from $\Pi^X_1\supset \Span(\be_1)$, $|a_{2,2}|\ge \vep$. 

Proceed recursively until $m = n-1$: let $X_m\in  S\cap(B_1\setminus B_\vep)$ be $\vep$ far from $\Pi^X_{m-1}$, where $\Pi^X_{m-1}$ is any $(n-2)$-dimensional plane containing $\{0, X_1,\dots,X_{m-1}\}$ (such a plane always exists since $m\le n-1$). 
Choose $\be_m \perp \Span(\be_1,\dots,\be_{m-1})$ and such that $\Span(\be_1,\dots,\be_{m}) = \Span(X_1,\dots,X_{m})$. 
Then,
\[
X_m = a_{m,1}\be_1+a_{m,2}\be_2+\dots+a_{m,m-1} \be_{m-1}+a_{m,m} \be_m,
\]
and since $X_m$ is $\vep$ far from $\Pi_{m-1}^X \supset \Span(\be_1,\dots,\be_{m-1})$, $|a_{m,m}|\ge \vep$. 
Therefore,
\begin{align*}
|X_1\wedge\dots\wedge X_{n-1}|
& = |a_{1,1}\be_1 \wedge a_{2,1}\be_1+a_{2,2}\be_2 \wedge\dots\wedge a_{n-1,1}\be_1+\dots+a_{n-1,n-1}\be_{n-1}|\\
& = |a_{1,1}a_{2,2}\cdots a_{n-1,n-1}| |\be_1\wedge \be_2 \wedge \dots\wedge \be_{n-1}|\\
& \ge \vep^{n-1},
\end{align*}
as desired.
\end{proof}

We close this section by collecting the results we have proved to understand the size of $\Sigma_\kappa^{m,{\rm a}}$ when $\kappa = 2$ and $m \leq n-1$ and when $\kappa \in 2\N$ and $m = n-1$. 

\begin{proof}[Proofs of Proposition~\ref{prop.main2} and Remark~\ref{rem.zerodim}]
We separate each case.
\begin{enumerate}[(i)]
\item This follows by Lemma~\ref{lem.sigmaempty}, noting that $\tilde \Sigma_2^{0,{\rm a}} = \Sigma_2^{0,{\rm a}}$.
\item If $a\in [0,1)$ or $a\in (-1, 0)$ and $m < n-1$, this follows from Proposition~\ref{prop: spine is nodal set} by noting that $\tilde \Sigma_2^{m,{\rm a}} = \Sigma_2^{m,{\rm a}}$. If $a\in (-1, 0)$ and $m = n-1$, this is due to Proposition~\ref{prop: top anom is lower d}.
\item If $a\in [0,1)$, we use Proposition~\ref{prop: spine is nodal set}, noticing that $\tilde \Sigma_\kap^{n-1,{\rm a}} = \Sigma_\kap^{n-1,{\rm a}}$. 
If $a\in (-1, 0)$, we use Proposition~\ref{prop: top anom is lower d}.
\end{enumerate}
Finally, regarding Remark~\ref{rem.zerodim}, if $n = 2$ and $a\in (-1, 0)$, $\Sigma_\kap^{1,{\rm a}}$ is countable by Lemma~\ref{lem.discreten-1}, and if $n = 2$ and $a \in [0,1)$, $\Sigma_\kap^{1,{\rm a}}$ is discrete by Lemma~\ref{lem.sigma1}. 
If $n\ge 3$, $\Sigma_2^{1,{\rm a}}$ is discrete by Lemma~\ref{lem.sigma1}, as well.
\end{proof}


\section{Whitney's Extension Theorem and the Proof of Theorem~\ref{thm.main1}}
\label{sec:whitney}

In this section, we prove our first higher regularity result Theorem~\ref{thm.main1}.
The proof of Theorem~\ref{thm.main1} is a model for the proofs of our main results, and utilizes an implicit function theorem argument and the following generalized Whitney's extension theorem.
(See \cite{Fef09} and the references therein.) 

\begin{lemma}[Whitney's Extension Theorem]
\label{lem.wet}
Let $\beta \in (0,1]$, $\ell \in \mathbb{N}$, $K \subset \R^{n+1}$ be compact, and $f : K \to \R$.
For every $Z_\circ \in K$, suppose that there exits a degree $\ell$ polynomial $P_{Z_\circ}$ for which
\begin{enumerate}[(i)]
\item $P_{Z_\circ}(Z_\circ) = f(Z_\circ)$; and 
\item $|D^\alpha P_{Z_\circ}(X_\circ) - D^\alpha P_{X_\circ}(X_\circ)| \leq C|Z_\circ - X_\circ|^{\ell + \beta - |\alpha|}$ for all $|\alpha| \leq \ell$ and $X_\circ \in K$, where $C > 0$ is independent of $Z_\circ$
\end{enumerate}
hold.
Then, there exists a function $F : \R^{n+1} \to \R$ of class $C^{\ell,\beta}$ and constant $C_{\ell,n} > 0$ for which
\[
F|_K \equiv f \quad\text{and}\quad |F(X) - P_{X_\circ}(X)| \leq C_{\ell,n}|X-X_\circ|^{\ell + \beta} \quad \text{for all}\quad X_0 \in K.
\]
\end{lemma}

Now we state and prove a collection of results that, in aggregate, prove Theorem~\ref{thm.main1}.

\begin{theorem}
\label{thm: C11}
The set $\Sigma_\kap^{m,{\rm g}}$ is contained in the countable union of $m$-dimensional $C^{1,1}$ manifolds. 
\end{theorem}

\begin{proof}
Let us define 
\begin{equation}
\label{eqn:Eh}
E_h := \{X_\circ \in \Sigma_\kap \cap \overline{B_{1-1/h}} : h^{-1} \rho^\kap \le \sup_{|X-X_\circ| = \rho}|u(X)|< h\rho^\kap, 0<\rho<1-|X_\circ|\}.
\end{equation}
From the continuity of the map 
\[
\Sigma_{\kappa} \ni X_\circ \mapsto p_{*,X_\circ}
\]
(see \cite[Proposition 4.6]{GR18}), we find that the map
\[
E_h \ni X_\circ \mapsto N(0^+, u(X_\circ + \,\cdot\,) - p_{*,X_\circ})
\]
is upper semicontinuous (it is the pointwise monotone decreasing limit of a sequence of continuous maps).
Here, the sets $E_h \subset \overline{B_{(h - 1)/h}}$ are closed and decompose $\Sigma_\kappa$ as follows:
\[
\Sigma_{\kappa} = \bigcup_{h=1}^{\infty} E_h
\] 
(this follows arguing exactly as in the proof of \cite[Lemma~1.5.3]{GP09}, using \cite[Lemma 4.5]{GR18}).
In turn, the set
\[
S_{\kappa,\lambda}^h := \{ X_\circ \in E_h : N(0^+, u(X_\circ + \,\cdot\,) - p_{*,X_\circ}) \geq \lambda \}
\] 
is compact in $\R^{n+1}$.

For each $X_\circ \in S_{\kappa,\lambda}^h$, define
\[
P_{X_\circ}(X) := p_{*,X_\circ}(X-X_\circ).
\]
We claim that $f \equiv 0, K = S_{\kappa,\lambda}^h$, and $\{P_{X_\circ}\}_{X_\circ \in K}$ satisfy the hypotheses of Whitney's Extension Theorem, Lemma~\ref{lem.wet}, with $\ell + \beta = \lambda$.

Clearly, (i) holds.

To show (ii) holds, first observe that Lemma~\ref{lem.HMon} implies that for all $X_\circ \in S_{\kappa,\lambda}^h$,
\[
\|u(X_\circ + r\,\cdot\,) -  p_{*,X_\circ}(r\,\cdot\,)\|_{L^2(B_{1},|y|^a)} \leq \frac{1}{n+1+a+2\lambda}\|u(X_\circ + r\,\cdot\,) -  p_{*,X_\circ}(r\,\cdot\,)\|_{L^2(\pa B_{1},|y|^a)}
\]
and
\[
\|u(X_\circ + r\,\cdot\,) -  p_{*,X_\circ}(r\,\cdot\,)\|_{L^2(\pa B_{1},|y|^a)} \leq C_h r^\lambda.
\]
So
\begin{equation}
\label{eqn : Monneau to Whitney}
\|u(X_\circ + r\,\cdot\,) -  p_{*,X_\circ}(r\,\cdot\,)\|_{L^2(B_{1},|y|^a)} \leq C_h r^\lambda
\quad \text{for all}\quad X_\circ \in S_{\kappa,\lambda}^h.
\end{equation}
(Of course, $r < 1 - |X_\circ|$.)
Now for any pair $Z_\circ,X_\circ \in S_{\kappa,\lambda}^h$,
\begin{equation}
\label{eqn: poly mod cty}
\|[P_{Z_\circ} - P_{X_\circ}](r\,\cdot\,)\|_{L^2(B_{1/2}(r^{-1}X_\circ),|y|^a)} \leq C_hr^\lambda
\end{equation}
where $r := 2|X_\circ - Z_\circ|$.
Indeed,
\[
\|[P_{Z_\circ} - P_{X_\circ}](r\,\cdot\,)\|_{L^2(B_{1/2}(r^{-1}X_\circ),|y|^a)} \leq {\rm I} + {\rm II}.
\]
with
\[
{\rm I} + {\rm II} := \|u(r\,\cdot\,) - P_{X_\circ}(r\,\cdot\,)\|_{L^2(B_{1/2}(r^{-1}X_\circ),|y|^a)} + \|u(r\,\cdot\,) - P_{Z_\circ}(r\,\cdot\,)\|_{L^2(B_{1/2}(r^{-1}X_\circ),|y|^a)}.
\]
Now assume that $r < h^{-1}$.
Then, by \eqref{eqn : Monneau to Whitney} applied at $X_\circ$ and $Z_\circ$,
\[
{\rm I} = \|u(X_\circ + r\,\cdot\,) - p_{*,X_\circ}(r\,\cdot\,)\|_{L^2(B_{1/2},|y|^a)} \leq \|u(X_\circ + r\,\cdot\,) - p_{*,X_\circ}(r\,\cdot\,)\|_{L^2(B_{1},|y|^a)} \leq  C_hr^\lambda
\]
and
\begin{align*}
{\rm II} &  = \|u(Z_\circ + r\,\cdot\,) - p_{*,Z_\circ}(r\,\cdot\,)\|_{L^2(B_{1/2}(r^{-1}(X_\circ-Z_\circ)),|y|^a)} \\
& \leq \|u(Z_\circ + r\,\cdot\,) - p_{*,Z_\circ}(r \,\cdot\,)\|_{L^2(B_{1},|y|^a)} \leq C_hr^\lambda.
\end{align*}
When $h^{-1}\leq r < 4$, \eqref{eqn: poly mod cty} is true by the triangle inequality, using that $p_{*,X_\circ}$ and $p_{*,Z_\circ}$ are homogeneous, and the bound $\|p_{*,X_\circ}\|_{L^2(B_1,|y|^a)}, \|p_{*,Z_\circ}\|_{L^2(B_1,|y|^a)} \leq C$.
Finally, since all norms are equivalent on the finite dimensional space of $\kappa$-homogeneous polynomials, \eqref{eqn: poly mod cty} implies that
\[
\|[P_{Z_\circ} - P_{X_\circ}](r\,\cdot\,)\|_{C^\ell(B_{1/2}(X_\circ/r))} \leq C_hr^\lambda
\]
for any $X_\circ,Z_\circ \in S_{\kappa,\lambda}^h$ with $r = 2|X_\circ-Z_\circ|$.
In turn,
\[
|D^\alpha P_{Z_\circ}(rX) - D^{\alpha}P_{X_\circ}(rX)| \leq C_hr^{\ell + \beta - |\alpha|} = C_h|X_\circ - Z_\circ|^{\ell + \beta - |\alpha|} \quad \text{for all}\quad X \in B_{1/2}(r^{-1}X_\circ).
\]
Taking $X = X_\circ/r$, we see that (ii) holds.

With our claim justified, applying Whitney's Extension Theorem, we find an $F \in C^{\ell,\beta}(\R^{n+1})$ such that
\[
|F(X) - P_{X_\circ}(X)| \leq C_{h}|X-X_\circ|^{\ell + \beta} \quad\text{for all}\quad X_\circ \in S_{\kappa,\lambda}^h.
\]
If $X_\circ \in S_{\kappa,\lambda}^h \cap \Sigma_\kappa^m$, by definition, there exist $n-m$ linearly independent unit vectors $\be_i \in \R^n$ and points $(x^i,0)$, $i = 1,\dots, n-m$, such that
\[
\pa_{\be_i}p_{*,X_\circ}(x^i,0) = \be_i \cdot \nabla_x p_{*,X_\circ}(x^i,0) \neq 0.
\]
Let $\bv_i$ be the unit vector parallel to $(x^i,0)$ and oriented so that $\bv_i \cdot x^i > 0$.
Then, we deduce that
\[
\pa_{\be_i}\pa^{(\kappa-1)}_{\bv_i}F(X_\circ) = \pa^{(\kappa-1)}_{\bv _i}\pa_{\be_i}p_{*,X_\circ}(0) \neq 0 \quad\text{for all}\quad i = 1,\dots n - m.
\]
On the other hand,
\[
\Sigma^m_{\kappa} \cap S_{\kappa,\lambda}^h \subset \bigcap_{i = 1}^{n-m} \{ \pa^{(\kappa-1)}_{\bv_i}F = 0 \}.
\]
Notice that $\pa^{(\kappa-1)}_{\bv_i}F \in C^{\ell - \kappa + 1,\beta}(\R^{n+1})$.
In turn, by the implicit function theorem, $\Sigma^m_{\kappa} \cap S_{\kappa,\lambda}^h$ is contained in an $m$-dimensional manifold of class $C^{\ell - \kappa + 1,\beta}$.

The theorem then follows by the definition of $\Sigma_\kap^{m,{\rm g}}$, which implies that $\ell = \kappa$ and $\beta = 1$.
\end{proof}

\begin{remark}
\label{rem.uniqmani}
In contrast to the classical (non-degenerate) obstacle problem, studied in \cite{FS18}, in the thin obstacle problem, singular points of many different orders may exist.
Their interaction (see Remark~\ref{rem.count}) makes it impossible to prove that $\Sigma_\kap^{m,{\rm g}}$ is contained in a single $m$-dimensional manifold.
But in the non-degenerate setting, this is ruled out, and only singular points of order 2 exist.
\end{remark}

\begin{theorem}
\label{thm: C11 2}
In the non-degenerate case, $\Sigma_{2}^{m, {\rm g}}$ is contained in a single $m$-dimensional $C^{1,1}$ manifold. 
\end{theorem}

\begin{proof}
In this setting, the singular set is closed. 
Consider $B_{1-\eta}\subset B_1$, for any $\eta \in (0,1)$. 
Thanks to the non-degeneracy condition (see Definition~\ref{def.nondeg}), there exists a constant $c_\circ > 0$, depending only on $n$, the non-degeneracy constant $c$, and $\eta$, such that 
\[
\sup_{B_r(X_\circ)} u \ge c_\circ r^2,
\]
for all small $r > 0$ and all $X_\circ \in \Sigma_2(u) \cap B_{1-\eta}$ (see \cite[Lemma 3.1]{BFR18}).
In particular, using the notation from the proof of Theorem~\ref{thm:  C11}, there exists some $h_\circ \ge \max\{c_\circ^{-1}, \eta^{-1}\}$ such that $\Sigma_2\cap B_{1-\eta} \subset E_{h_\circ}$. 
Thus, by the proof of Theorem~\ref{thm:  C11}, $\Sigma_2^m\cap S^{h_\circ}_{2, 3}$ is contained in a single $m$-dimensional manifold of class $C^{1, 1}$, and since this can be done for any $\eta >0$, we obtain that $\Sigma_2^m$ is locally contained in a single $m$-dimensional manifold. 
This concludes the proof. 
\end{proof}

\begin{proposition}
\label{prop: C2akappa}
If $a \in (-1,0)$, the set $\Sigma_{\kappa}^{n-1}$ is contained in a countable union of $(n-1)$-dimensional $C^{1,\alpha_\kappa}$ manifolds. Moreover, in the non-degenerate case, it is contained in a single $(n-1)$-dimensional $C^{1,\alpha}$ manifold, for some $\alpha > 0$ depending only on $n$. 
\end{proposition}

\begin{proof}
The proof follows that of Theorem~\ref{thm: C11} exactly, replacing $\beta =1$ with $\beta = \alpha_\kappa$; when $a \in (-1,0)$ and $m = n-1$, second blow-ups always have higher homogeneity: $\lambda_*\ge \kap+\alpha_\kap > \kap$ (see Proposition~\ref{prop.case2}).
In the non-degenerate case, we can proceed as in Theorem~\ref{thm:  C11 2} instead.
\end{proof}

Finally, we can proceed with the proof of one of our main results, Theorem~\ref{thm.main1}. 

\begin{proof}[Proof of Theorem~\ref{thm.main1}]
We separate each case. 
\begin{enumerate}[(i)]
\item Notice that we are in \ref{eq.Case1} (since $\kap = 2$).
So apply Lemma~\ref{lem: Sigma0k is isolated}.
\item $\Sigma_2^{m, {\rm a}}$ is lower dimensional by Proposition~\ref{prop.main2}, while $\Sigma_2^{m,{\rm g}}$ is covered by a countable union of $m$-dimensional $C^{1,1}$ manifolds by Theorem~\ref{thm: C11}.
\item Again, $\Sigma_\kap^{n-1, {\rm a}}$ is lower dimensional by Proposition~\ref{prop.main2}.
And $\Sigma_\kap^{n-1, {\rm g}}$ is covered by a countable union of $(n-1)$-dimensional $C^{1,1}$ manifolds by Theorem~\ref{thm: C11}.
\item This follows by Proposition~\ref{prop:  C2akappa}. 
\end{enumerate}
This completes the proof. 
\end{proof}

\section{The Main Results}
\label{sec.thirdorder}

In this section, we construct the second term in the expansion of $u$ at singular points, up to a lower dimensional set.
We start by defining a specific subset of the generic singular points at which the nodal set and spine of the first blow-up align.
These points will be those at which we produce the next term (of order exactly $\kappa + 1$) in the expansion of $u$ at a order $\kappa$ singular point, the goal of this work.

\begin{definition}
\label{defi.3rd}
Let $n \ge 2$ and $0\le m\le n-1$. 
We define the set $\Sigma_\kap^{m, {\rm nxt}}$ as the set of singular points $X_\circ \in \tilde{\Sigma}^{m, {\rm g}}_\kap$ for which there exists a sequence $r_\ell\downarrow 0$ as $\ell \to \infty$ such that the following holds:
there exists a $(\kap+1)$-homogeneous, $a$-harmonic polynomial $q_\circ$ such that
\begin{enumerate}[(i)]
\item \[
v_{r_\ell, \kap+1} := \frac{u(X_\circ + r_\ell \,\cdot\,) -p_{*,X_\circ}(r_\ell\,\cdot\,)}{r_\ell^{\kap+1}} \rightharpoonup q_\circ \quad \text{in}\quad W^{1, 2}(B_1,|y|^a);
\]
\item $D^\alpha q_\circ$ vanishes on $L(p_{*, X_\circ})$ for all $\alpha = (\alpha_1,\dots,\alpha_n,0)$ and $|\alpha| \le \kap -2$; and
\item 
\[
\| q_\circ \|_{L^2(\pa B_1, |y|^a)}^2 = H_{\kap+1}(0^+, u(X_\circ + \,\cdot\,) -p_{*,X_\circ}).
\]
\end{enumerate}
\end{definition}

In the first set of results of this section, we estimate the size of $\Sigma_\kappa^{m,{\rm nxt}}$ for certain pairs of $\kappa$ and $m$.

\begin{lemma}
\label{lem.3rd_1}
Let $n\ge 2$ and $a \in [0,1)$.
Then, $\dim_{\mathcal{H}} \Sigma^{n-1}_\kap\setminus \Sigma^{n-1, {\rm nxt}}_{\kap} \le n-2$. 
\end{lemma}

\begin{proof}
We proceed as in Proposition~\ref{prop: spine is nodal set}. 
Notice that by Proposition~\ref{prop: spine is nodal set}, $\Sigma_\kap^{n-1, {\rm a}}$ is already lower dimensional.
So we restrict our attention to $\Sigma^{n-1, {\rm g}}_\kap$.  
Let 
\[
\Sigma := \Sigma^{n-1, {\rm g}}_\kap \setminus \Sigma^{n-1, {\rm nxt}}_{\kap}
\]
and suppose that $\mathcal{H}_\infty^\beta(\Sigma) > 0$ for some $\beta > n-2$. 
Then, there exists a point $X_\circ \in \Sigma$ and a sequence $r_j \downarrow 0$ such that 
\[
\frac{\mathcal{H}_\infty^\beta(\Sigma\cap B_{r_j}(X_\circ))}{r_j^{\beta}} \ge c_{n,\beta} > 0.
\]
Up to a translation, assume that $X_\circ = 0$. 
By assumption, 
\[
\lambda_*:= N(0^+, v_\ast) \ge \kap+1,
\]
and, up to a subsequence $r_{\ell} = r_{j_\ell}$, 
\[
\tilde v_{r_\ell} \rightharpoonup q\quad \text{in}\quad W^{1, 2}(B_1,|y|^a),
\]
where $\tilde v_{r_\ell}$ is defined as in \eqref{eq.vtilde2}, and $q$ is $a$-harmonic and at least $(\kappa + 1)$-homogeneous.
Moreover, by Lemma~\ref{lem.lem35}(ii),
\[
\mathcal{H}_\infty^\beta(\mathcal{A}) > 0,
\]
where $\mathcal{A} = \mathcal{A}_\Sigma$. 
Now if $Z\in \mathcal{A}$, then there are sequences $X_\ell \in \Sigma$ and $r_\ell\downarrow 0$ such that $|X_\ell|\le r_\ell$ and $X_\ell/2r_\ell\to Z/2$. 
By Lemma~\ref{lem:  accumulation pts}, if we denote 
\[
D_{\kap-2}(q) := \{ X = (x,0) : D^\alpha q(X) = 0 \text{ for all } \alpha = (\alpha_1,\dots,\alpha_n,0) \text{ with } |\alpha|\le \kap-2 \},
\]
then $Z\in L_*\cap D_{\kap-2}(q)$, so that
\[
\mathcal{A} \subset \overline{B_1} \cap L_*\cap D_{\kap-2}(q).
\]

Now, using the monotonicity of $H_{\kap+1}(r_\ell, v_\ast)$ (see Lemma~\ref{lem.HMon}), we have that $H_{\kap+1}^{1/2}(0^+, v_\ast)$ exists.
So let 
\[
q_\circ := H_{\kap+1}^{1/2}(0^+, v_\ast)q
\]
and notice that
\[
v_{r_\ell,\kap+1} := \frac{u(r_\ell \,\cdot\,) -p_{*}(r_\ell \,\cdot\,)}{r_\ell^{\kap+1}} = \tilde v_{r_\ell} \frac{\|v_{r_\ell}\|_{L^2(\pa B_1)}}{r_\ell^{\kap+1}} = \tilde v_{r_\ell} H_{\kap+1}^{1/2}(r_\ell, v_\ast).
\]
In turn,
\[
v_{r_\ell,\kap+1} \rightharpoonup q_\circ \quad \text{in}\quad W^{1, 2}(B_1,|y|^a).
\]
Additionally, 
\begin{equation}
\label{eqn: l2 norm cond}
\|q_\circ\|_{L^2(\pa B_1, |y|^a)}^2 =  H_{\kap+1}(0^+, v_\ast)
\end{equation}
since $\|q\|_{L^2(\pa B_1, |y|^a)} = 1$.

If $\lambda_* > \kappa + 1$, then $\|v_*\|^2_{L^2(\pa B_r,|y|^a)} = o(r^{2(\kappa + 1)+n+a})$.
And so $H_{\kap+1}(0^+, v_\ast) = 0$, which, by \eqref{eqn: l2 norm cond}, implies that $q_\circ \equiv 0$.
But this is impossible: 
\begin{equation}
\label{eqn : 0 not in Sigm3rd}
0 \notin \Sigma^{n-1,{\rm nxt}}_\kappa.
\end{equation}
(In this case, $q_\circ$ is trivially $(\kappa + 1)$-homogeneous and $D_{\kappa-2}(q_\circ) = \R^n \times \{ 0 \}$.)
In turn, $\lambda_* = \kappa + 1$ and $q_\circ \not\equiv 0$.
Thus, by \eqref{eqn : 0 not in Sigm3rd},
\[
D_{\kappa-2}(q) \cap L_\ast = D_{\kappa-2}(q_\circ) \cap L_\ast \subsetneq L_\ast.
\]
Hence, by the analyticity of $q$,
\[
\dim_\mathcal{H} D_{\kappa-2}(q) \cap L_\ast \leq n-2.
\]
But then, $\mathcal{H}_\infty^\beta(\mathcal{A}) = 0$, a contradiction. 
\end{proof}

Notice that $\Sigma^{0}_2 = \Sigma^{0, {\rm nxt}}_{2}$ since $\Sigma^{0, {\rm a}}_2$ is empty and $\tilde{\Sigma}_2 = \Sigma_2$.

\begin{lemma}
\label{lem.3rd_2}
Let $n\ge 3$. 
Suppose that $1 \le m \le n-2$. 
Then, $\dim_{\mathcal{H}} \Sigma^{m}_2\setminus \Sigma^{m, {\rm nxt}}_{2} \le m-1$.
\end{lemma}

\begin{proof}
The proof is identical to the proof of Lemma~\ref{lem.3rd_1}.
Notice that $\Sigma_2^{m, {\rm a}}$ is lower dimensional by Proposition~\ref{prop: spine is nodal set}, and $\tilde \Sigma_2^{m, {\rm a}} = \Sigma_2^{m, {\rm a}}$.
\end{proof}

\begin{lemma}
\label{lem.3rd_3}
Let $n \geq 2$.
Suppose that $a \in (-1,0)$.
Then, $\dim_{\mathcal{H}} \Sigma^{n-1}_\kap \setminus \Sigma^{n-1, {\rm nxt}}_{\kap} \le n-2$.
\end{lemma}

\begin{proof}
We proceed as in Lemma~\ref{lem.3rd_1}. 
By Proposition~\ref{prop: top anom is lower d}, $\Sigma_\kap^{n-1, {\rm a}}$ is lower dimensional, so we define
\[
\Sigma := \Sigma^{n-1, {\rm g}}_\kap \setminus \Sigma^{n-1, {\rm nxt}}_{\kap}
\]
and suppose that $\mathcal{H}_\infty^\beta(\Sigma) > 0$ for some $\beta > n-2$. 
We can assume that at 0 and for some $r_j\downarrow 0$,
\[
\frac{\mathcal{H}_\infty^\beta(\Sigma\cap B_{r_j}(X_\circ))}{r_j^{\beta}} \ge c_{n,\beta} > 0 \quad\text{and}\quad
\lambda_*= N(0^+, v_{*})\ge \kap+1.
\]
Furthermore, up to a subsequence $r_{\ell} = r_{j_\ell}$, 
\[
\tilde v_{r_\ell} \rightharpoonup q\quad\text{in}\quad W^{1, 2}(B_1,|y|^a),
\]
where $\tilde v_{r_\ell}$ is defined as in \eqref{eq.vtilde2}, and $q$ is a global homogeneous solution to the very thin obstacle problem with homogeneity $\lambda_*\ge \kap+1$. 
Moreover, by Lemma~\ref{lem.lem35}(ii),
\[
\mathcal{H}_\infty^\beta(\mathcal{A}) > 0,
\]
where $\mathcal{A} = \mathcal{A}_\Sigma$, and
\[
\mathcal{A}\subset \overline{B_1}\cap L_*
\] 
by Lemma~\ref{lem: accumulation pts 2}.

Arguing as in Lemma~\ref{lem.3rd_1}, since $0\in \Sigma$, if we set
\[
q_\circ := H_{\kap+1}^{1/2}(0^+, v_\ast)q,
\]
we find that $q_\circ$ and $q$ are $(\kap+1)$-homogeneous and non-zero. 

Let us decompose $q$ into its odd and even parts with respect to $L_*$ as defined in Lemma~\ref{lem: accumulation pts 2}: $q = q^{\rm odd}+q^{\rm even}$.
Without loss of generality and for simplicity, assume that 
\[
L_* = \{x_n = y = 0\}.
\]

On one hand, by the proof of Lemma~\ref{lem: accumulation pts 2}, $q^{\rm odd}$ is an $a$-harmonic, $(\kap+1)$-homogeneous function, which by Liouville's theorem (\cite[Lemma 2.7]{CSS08}), is a polynomial.
On the other hand, since $\mathcal{H}_\infty^\beta(\mathcal{A}) > 0$, there are $n-1$ elements in $\mathcal{A}$, $Y_1,\dots,Y_{n-1}$, such that $\Span(Y_1,\dots,Y_{n-1}) = L_*$. 
By Remark~\ref{rem.qeven}, $q$ is then a polynomial. 
Indeed, for each $Y_i$, there exists a sequence $\{X_i^\ell\}_{\ell\in \N}$ with $X_i^\ell\in \Sigma$ such that $|X_i^\ell|\le r_\ell$ and $Y_i^\ell := X_i^\ell/r_\ell \to Y_i$ as $\ell\to \infty$. 
In addition, if we let 
\[
f(X) := N(0^+, u(X+\,\cdot\,) - p_{*, X})\quad\text{for}\quad  X\in \R^n \times \{ 0\},
\]
then $\kap+1\le f(X_i^\ell)$ (since $X_i^\ell\in \Sigma$).
Also, $\limsup_{\ell \to \infty} f(X_i^\ell) \le \lambda_* = \kap+1$ ($f$ is upper semicontinuous).
So Almgren's frequency at $0^+$ is continuous along the sequences $\{X_i^\ell\}_{\ell\in \N}$ and $i \in \{1,\dots,n-1\}$. 
Therefore, the hypotheses of Remark~\ref{rem.qeven} hold, and we have that
\begin{equation}
\label{eq.vit}
q^{\rm even}(X + \tau Y_i) = q^{\rm even}(X)  + \tau c_i\Ext_a(x_n^\kappa),
\end{equation}
(recall, $L_* = \{ x_n = y = 0\}$). 
Since $\{Y_i\}_{1\le i \le n-1}$ spans $L_*$, for any $X = (x',x_n,x_{n+1})\in \R^{n+1}$, 
\begin{align*}
X = (0,\dots,0,x_n,x_{n+1}) + (\ba_1\cdot x')Y_{1}+\dots+(\ba_{n-1}\cdot x')Y_{n-1},
\end{align*}
for some fixed vectors $\ba_j' \in \R^{n-1}$ for $j\in\{1,\dots,n-1\}$. 
Now applying \eqref{eq.vit} iteratively, we deduce that
\begin{align*}
q^{\rm even}(X) & = (c_1(\ba_1' \cdot x')+\dots c_{n-1}(\ba_{n-1}' \cdot x'))\Ext_a(x_n^\kappa) + \bar q(x_n, x_{n+1})\\
& = (\ba' \cdot x')\Ext_a(x_n^\kappa) + \bar q(x_n, x_{n+1}).
\end{align*}
Here, $\bar q(x_n, x_{n+1}) = q^{\rm even}(0,\dots,0,x_n,x_{n+1})$ and $\ba' \in\R^{n-1}$. 
As $q^{\rm even}$ is a solution to the very thin obstacle problem which is $(\kap+1)$-homogeneous and $(\ba' \cdot x')\Ext_a(x_n^\kappa)$ is $(\kap+1)$-homogeneous, $a$-harmonic, and vanishes on $L_*$, we find that $\bar q$ is a $(\kap+1)$-homogeneous solution to the very thin obstacle problem. 
By Lemma~\ref{lem: 2d class}, since $\bar q$ is two-dimensional and $(\kap+1)$-homogeneous, it is a polynomial. 
But $\bar q$ is also even in both $x_n$ and $y$, which implies $\bar q \equiv 0$. 
Therefore, $q^{\rm even}$ is also a polynomial. 
That is, $q$ is a $(\kap+1)$-homogeneous polynomial since both $q^{\rm odd}$ and $q^{\rm even}$ are polynomials.

To conclude, observe that because $q$ is a $(\kap+1)$-homogeneous polynomial, for any $Z\in \mathcal{A}$, there are sequences $X_\ell \in \Sigma$ and $r_\ell\downarrow 0$ such that $|X_\ell|\le r_\ell$ and $X_\ell/r_\ell\to Z$.  
By Remark~\ref{rem.onlypoly}, $D^\alpha q(Z) =  0$ for all $\alpha =(\alpha_1,\dots,\alpha_n, 0)$ such that $|\alpha|\le \kap-2$, i.e., 
\[
\mathcal{A}\subset \overline{B_1} \cap L_*\cap D_{\kap-2}(q) = \overline{B_1} \cap L_*\cap D_{\kap-2}(q_\circ).
\]
Since $0 \notin \Sigma^{n-1,{\rm nxt}}_\kappa$,
\[
D_{\kappa-2}(q) \cap L_\ast = D_{\kappa-2}(q_\circ) \cap L_\ast \subsetneq L_\ast,
\]
and by the analyticity of $q$,
\[
\dim_\mathcal{H} D_{\kappa-2}(q) \cap L_\ast \leq n-2.
\]
But then, $\mathcal{H}_\infty^\beta(\mathcal{A}) = 0$, a contradiction. 
\end{proof}

In some of the end point cases, we can say more.

\begin{corollary}
\label{cor.countC2}
\
\begin{enumerate}[(i)]
\item If $n = 2$ and $a \in [0,1)$, then $\Sigma^{1}_\kap\setminus \Sigma^{1, {\rm nxt}}_{\kap}$ is countable.
\item If $n = 2$ and $a \in (-1,0)$, then $\Sigma^{1}_\kap\setminus \Sigma^{1, {\rm nxt}}_{\kap}$ is countable.
\item If $n \ge 3$ and $m = 1$, then $\Sigma^{1}_2\setminus \Sigma^{1, {\rm nxt}}_{2}$ is countable.
\end{enumerate}
\end{corollary}

\begin{proof}
We separate each case.
\begin{enumerate}[(i)]
\item Notice that if $n = 2$ and $a \in [0,1)$, then $\Sigma_\kap^{1,{\rm a}}$ is discrete by Lemma~\ref{lem.sigma1}. 
Repeating the proof of Lemma~\ref{lem.3rd_1}, but assuming, to the contrary, that $\Sigma_\kap^{1, {\rm g}}\setminus \Sigma_\kap^{1,{\rm nxt}}$ has accumulation points, we deduce that $\Sigma_\kap^{1, {\rm g}}\setminus \Sigma_\kap^{1,{\rm nxt}}$ is discrete as well.
The result follows.
\item By Lemma~\ref{lem.discreten-1}, we see that $\Sigma_\kap^{1,{\rm a}}$ is countable. 
In addition, repeating the arguments used to prove Lemma~\ref{lem.3rd_3}, we deduce that $\Sigma_\kap^{1, {\rm g}}\setminus\Sigma_\kap^{1,{\rm nxt}}$ cannot have accumulation points.
\item Following the proof of (i), but using Lemma~\ref{lem.3rd_2}, we conclude.
\end{enumerate}
This completes the proof.
\end{proof}

The next pair of statement concern the almost monotonicity of a Monneau-type energy and the uniqueness and continuity of second blow-ups at points in $\Sigma_\kappa^{m,{\rm nxt}}$.

\begin{lemma}
\label{lem.qp}
Let $X_\circ \in \Sigma^{m, {\rm nxt}}_\kap \cap K$ for some compact set $K \subset B_1 \cap \{ y = 0 \}$, $q_\circ$ be as in Definition~\ref{defi.3rd}, and $H_\lambda$ be as in \eqref{eq.HMon}. 
Then, 
\[
\frac{\d}{\d r} H_{\kap+1} (r, u(X_\circ+ \,\cdot\,) -p_{\ast,X_\circ}-q_\circ) \ge -C_K\bigg\|\frac{q_\circ^\kap}{p_{\ast,X_\circ}^{\kap-1}}\bigg\|_{L^\infty(B_1 \cap \{ y = 0\})}.
\]
\end{lemma}

\begin{proof}
Without loss of generality, assume that $X_\circ = 0 \in \Sigma^{m, {\rm nxt}}_\kap$.
Set 
\[
w := v_{\ast}-q_\circ.
\]
Since
\[
\frac{\d}{\d r} H_{\lambda} (r, w) = \frac{2}{r^{n+a+ 2\lambda + 1}}\int_{\pa B_r} w(\nab w \cdot X - \lambda w)|y|^a,
\]
arguing as we did to show \eqref{eqn: key ineq 2}, we find that
\[
\frac{\d}{\d r} H_{\kap+1} (r, w) \geq \frac{2}{r^{n+a+ 2\kap+2}}\int_{B_r} w\,L_a w.
\]

Now observe that
\[
w \,L_a w = -(p_{\ast}+q_\circ)\,L_a u.
\]
From the numerical inequality $1 - \xi + \xi^\kappa \geq 0$ for all $\xi \geq 0$ and as $q_\circ|_{\R^n \times \{0\}} = 0$ when $p_{\ast}|_{ \R^n \times \{0\}} = 0$, we see that
\[
p_{\ast} + q_\circ \geq p_{\ast} - |q_\circ| \geq - \frac{q_\circ^\kap}{p_{\ast}^{\kap-1}} \quad \text{on}\quad \R^n \times \{0\}
\]
(recall that $\kappa$ is even and $p_{\ast} \geq 0$ on $ \R^n \times \{0\}$).
Therefore, using that $L_a u$ is a non-positive measure supported on $B_1 \cap \{ y = 0 \}$, we deduce that
\[
\frac{\d}{\d r} H_{\kap+1} (r, w) \geq - \frac{1}{r^{n+a+2\kap+2}}\int_{B_r} \frac{q_\circ^\kap}{p_{\ast}^{\kap-1}}\,L_a u.
\]

Because $D^\alpha q_\circ$ vanishes for all $\alpha = (\alpha_1,\dots,\alpha_n,0)$ with $|\alpha| \leq \kap - 2$ on $L_\ast$, we have that $q_\circ^\kap/p_{\ast}^{\kap-1}$ is locally bounded on $\R^n \times \{ 0\}$.
Moreover, $q_\circ^\kap/p_{\ast}^{\kap-1}$ is a $2\kap$-homogeneous polynomial.
Thus,
\[
- \frac{1}{r^{n+a+2\kap+2}}\int_{B_r} \frac{q_\circ^\kap}{p_{\ast}^{\kap-1}}L_a u 
\geq - \frac{C}{r^3}\bigg\|\frac{q_\circ^\kap}{p_{\ast}^{\kap-1}}\bigg\|_{L^\infty(B_1 \cap \{y =0\} )} \int_{B_{1/2}} |L_a u_{2r}|,
\]
as $r^{2-a}L_a u(rX) = L_a u_r(X)$. 
Now from the proof of Proposition~\ref{prop.case1}, we know that
\[
-\int_{B_{1/2}} |L_a u_{2r}| = -\|v_{2r} \|_{L^2(\partial B_1)}  \int_{B_{1/2}} |L_a \tilde v_{2r}| \geq -C\|v_{2r} \|_{L^2(\partial B_1)}.
\]
Moreover, thanks to Lemma~\ref{lem.HMon},
\[
-\|v_{2r} \|_{L^2(B_1)} \geq -C\|v_{2r} \|_{L^2(\partial B_1)} \geq -Cr^{\lambda_*}.
\]
In turn,
\[
\frac{\d}{\d r} H_{\kap+1} (r, w)\geq -Cr^{\lambda_*-3}\bigg\|\frac{q_\circ^\kap}{p_{\ast}^{\kap-1}}\bigg\|_{L^\infty(B_1 \cap \{ y=0\})},
\]
which, after recalling that $\lambda_\ast \geq \kappa + 1 \geq 3$, proves the lemma.
\end{proof}

\begin{proposition}
\label{prop: !limit + loc uni conv}
For every $X_\circ \in \Sigma^{m, {\rm nxt}}_{\kappa}$, there exists a unique $(\kap+1)$-homogeneous, $a$-harmonic polynomial $q_{*, X_\circ}$ such that 
\begin{equation}
\label{eqn: ! weak lim}
\frac{u(X_\circ+r\,\cdot\,)-p_{*,X_\circ}(r\,\cdot\,)}{r^{\kap +1}}\rightharpoonup q_{*, X_\circ} \quad \text{in}\quad W^{1, 2}(B_1) \text{ as } r\downarrow 0,
\end{equation}
$D^\alpha q_{*, X_\circ}$ vanishes on $L(p_{*, X_\circ})$ for any $\alpha = (\alpha_1,\dots,\alpha_n,0)$ with $|\alpha| \le \kap -2$, and 
\[
\| q_{*, X_\circ} \|^2_{L^2(\pa B_1, |y|^a)} = H_{\kap+1}(0^+, u(X_\circ + \,\cdot\,) - p_{*,X_\circ}).
\]
Moreover, the convergence in \eqref{eqn: ! weak lim} is uniform on compact subsets of $B_1 \cap \{ y = 0\}$, and the map 
\[
\Sigma^{m,{\rm nxt}}_{\kappa} \ni X_\circ \mapsto q_{*, X_\circ}
\]
is continuous.
\end{proposition}

\begin{proof}
Without loss of generality, we take $X_\circ = 0$.
Let $q_\circ$ denote the limit along the sequence $r_\ell$ given by Definition~\ref{defi.3rd}.
Let $\tilde{q}_\circ$ be another limit taken through another sequence, $\tilde{r}_{\ell}$, such that (after relabelling if necessary) $\tilde{r}_{\ell}\le r_\ell$. 
Then, by Lemma~\ref{lem.qp}, we have that 
\[
H_{\kap+1}(r_\ell, v_* - q_\circ)\ge H_{\kap+1}(\tilde{r}_{\ell}, v_*-q_\circ) - C\bigg\|\frac{q_\circ^\kap}{p_{*}^{\kap-1}}\bigg\|_{L^\infty(B_1 \cap \{ y = 0\})}|r_\ell-\tilde{r}_{\ell}| \quad\text{for all}\quad \ell \in \N.
\] 
Thus, using that $r_\ell^{-\kappa-1}v_{r_\ell} \to q_\circ$ strongly in $L^2(\pa B_1, |y|^a)$,
\begin{align*}
0 & = \lim_{\ell\to \infty} \int_{\pa B_1} (r_\ell^{-\kap-1}v_{r_\ell} - q_\circ)^2 |y|^a \\
& \ge \lim_{\ell\to \infty} \left(\int_{\pa B_1} (\tilde{r}_{\ell}^{-\kap-1}v_{\tilde{r}_{\ell}} - q_\circ)^2 |y|^a - C\bigg\|\frac{q^\kap}{p_{*}^{\kap-1}}\bigg\|_{L^\infty(B_1 \cap \{ y = 0 \})}|\tilde{r}_{\ell} - r_\ell| \right) \\
&= \int_{\pa B_1} (\tilde{q}_\circ-q_\circ)^2 |y|^a.
\end{align*}
And so, $\tilde{q}_\circ = q_\circ$, and the limit is unique.
The remainder of the proof follows the proof of \cite[Proposition 4.5]{FS18}. 
\end{proof}

\begin{remark}
Thanks to Proposition~\ref{prop: !limit + loc uni conv}, Definition~\ref{defi.3rd} can be amended to say {\it for every sequence} $r_\ell\downarrow 0$, instead of just {\it a sequence}. 
\end{remark}

An important consequence of Proposition~\ref{prop: !limit + loc uni conv}, particularly, the uniform convergence in compact sets of the limit \eqref{eqn: ! weak lim}, is the following: for each compact set $K \subset B_1 \cap \{ y = 0\}$, we have a modulus of continuity $\omega_K$ such that
\[
H_{\kap+1}(r, u(X_\circ + \,\cdot\,) - p_{*,X_\circ} - q_{*,X_\circ}) \leq \omega_K(r) \quad\text{for all}\quad X_\circ \in K \cap \Sigma_{\kappa}^{m,{\rm nxt}}.
\]
This modulus of continuity allows us to prove the following regularity result, a precursor to our main results. 

\begin{theorem}
\label{thm: C2}
The set $\Sigma_\kap^{m,{\rm nxt}}$ is contained in the countable union of $m$-dimensional $C^{2}$ manifolds. 
\end{theorem}

\begin{proof}
The proof will be completed in two steps.

\smallskip
\noindent{\bf Step 1:}
Let $E_h$ be the compact sets defined in the proof of Theorem~\ref{thm: C11}, and set
\[
E_{h,m} := \Sigma_{\kappa}^m \cap E_h \quad\text{and}\quad E_{h,m}^{{\rm nxt}} := \Sigma_{\kappa}^{m,{\rm nxt}} \cap E_h.
\]
Observe that by Lemma~\ref{lem.3rd_1} and Lemma~\ref{lem.3rd_2},
\[
\dim_\mathcal{H} E_{h,m} \setminus E_{h,m}^{{\rm nxt}} \leq m - 1
\]
when $m \geq 1$.
Hence, for any $j \in \mathbb{N}$, we can find a family of balls $\{ \hat{B}_i \}_{i = 1}^{\infty}$ such that
\[
E_{h,m} \setminus E_{h,m}^{{\rm nxt}} \subset \mathcal{O}_j := \bigcup _{i = 1}^{\infty} \hat{B}_i \quad\text{and}\quad \sum_{i = 1}^\infty \diam(\hat{B}_i)^{m - 1 + \frac{1}{j}} < \frac{1}{j}.
\]
In particular,
\[
\mathcal{H}^{m - 1 + \frac{1}{j}}_\infty(\mathcal{O}_j) < \frac{1}{j}.
\]
Now define
\[
\mathcal{U}_j := \bigg\{ X \in \R^{n+1} : \dist(X,\overline{E}_{h,m} \setminus E_{h,m}) < \frac{1}{j} \bigg\}
\quad\text{and}\quad
K_j := E_{h,m} \setminus (\mathcal{O}_j \cup \mathcal{U}_j).
\]
Notice that $\mathcal{O}_j$ and $\mathcal{U}_j$ are open, $K_j$ is closed, and
\[
K_j \subset E_{h,m}^{{\rm nxt}}.
\]
Moreover, we have that
\[
\bigcup_{j=1}^{\infty} K_j = E_{h,m} \setminus  \bigcap_{j=1}^{\infty} \mathcal{O}_j.
\]
Indeed, using the continuity of the map $\Sigma_\kappa \ni X_\circ \mapsto p_{*,X_\circ}$ (see \cite[Proposition 4.6]{GR18}) and that $E_{h}$ is closed and contained in $\Sigma_{\kappa}$, we find that
\[
\overline{\overline{E}_{h,m} \setminus E_{h,m}} \subset \bigcup_{d \geq m+1} E_{h,d},
\]
so that $\bigcap_{j=1}^{\infty} \mathcal{U}_j$ is disjoint from $E_{h,m}$.
Finally, by construction, $\mathcal{H}^{\beta}_\infty(\bigcap_{j=1}^{\infty} \mathcal{O}_j) = 0$ for all $\beta > m - 1$, which implies that $\dim_\mathcal{H} \bigcap_{j=1}^{\infty} \mathcal{O}_j \leq m - 1$.
In turn, if we can show that $E_{h,m}$ is contained in a $m$-dimensional $C^2$ manifold, then $\Sigma_{\kappa}^{m}$ can be covered by a countable collection of $m$-dimensional $C^2$ manifolds along with a set of Hausdorff dimension at most $m -1$.

\smallskip
\noindent{\bf Step 2:}
This step is essentially identical to the second half of the proof of Theorem~\ref{thm: C11}.
For completeness, however, we provide some details regarding the justification of hypothesis (ii) in the statement of Whitney's Extension Theorem.

For $X_\circ \in K_j$, define
\[
P_{X_\circ}(X) := p_{\ast,X_\circ}(X-X_\circ) + q_{\ast,X_\circ}(X-X_\circ).
\] 
Now let $X_\circ,Z_\circ \in K_j$ and $r := 2|X_\circ - Z_\circ|$.
Arguing as we did in the proof of Theorem~\ref{thm: C11}, but using Proposition~\ref{prop: !limit + loc uni conv}, we see that there exists a modulus of continuity $\omega_{K_j}$ such that
\[
\begin{split}
\frac{1}{r^{\kappa+1}}\|[P_{Z_\circ} - P_{X_\circ}](r\,\cdot\,)\|_{L^2(B_{1/2}(r^{-1}X_\circ),|y|^a)} 
\leq 2\omega_{K_j}(r). 
\end{split}
\]
So since all norms are equivalent on the finite dimensional space of polynomials of degree less than or equal to $\kappa + 1$, 
\[
\|[P_{Z_\circ} - P_{X_\circ}](r\,\cdot\,)\|_{C^{\kappa+1}(B_{1/2}(r^{-1}X_\circ))} \leq 2\omega_{K_j}(r)r^{\kappa+1},
\]
for any $X_\circ,Z_\circ \in K_j$ with $r = 2|X_\circ-Z_\circ|$.
In turn, given $|\alpha| \in \{0,\dots,\kappa+1\}$,
\[
|D^\alpha P_{Z_\circ}(X_\circ) - D^{\alpha}P_{X_\circ}(X_\circ)| \leq 2\omega_{K_j}(|X_\circ - Z_\circ|)|X_\circ - Z_\circ|^{\kappa + 1 - |\alpha|} \quad\text{for all}\quad X_\circ \in K_j.
\]
Thus, thanks to the Whitney's Extension Theorem, Lemma~\ref{lem.wet}, we conclude.
\end{proof}

\begin{theorem}
\label{thm: C2 nondeg}
In the non-degenerate case, the set $\Sigma_\kap^{m,{\rm nxt}}$ is contained in a single $m$-dimensional $C^{2}$ manifold. 
\end{theorem}

\begin{proof}
The proof follows the proof of Theorem~\ref{thm: C2}, but with the same modifications as the proof of Theorem~\ref{thm:  C11 2}. 
\end{proof}

To finish, we present the proofs of our two main results.

\begin{proof}[Proof of Theorem~\ref{thm: nondeg main}]
We prove each case separately. 
\begin{enumerate}[(i)]
\item As in the proof of Theorem~\ref{thm.main1}(i), this case holds by Lemma~\ref{lem: Sigma0k is isolated}.
\item $\Sigma_2^{1}\setminus\Sigma_2^{1, {\rm nxt}}$ is countable by Corollary~\ref{cor.countC2}.
So the proof follows from Theorem~\ref{thm: C2 nondeg}.

\item $\Sigma_2^{m, {\rm nxt}}$ is contained in an $m$-dimensional $C^2$ manifold by Theorem~\ref{thm:  C2 nondeg}. 
On the other hand, $\dim_{\mathcal{H}} \Sigma_2^{m}\setminus\Sigma_2^{m, {\rm nxt}} \le m-1$ by Lemmas~\ref{lem.3rd_1}, \ref{lem.3rd_2}, and \ref{lem.3rd_3}. 
\item See Proposition~\ref{prop: C2akappa}. 
\end{enumerate}
This concludes the proof.
\end{proof}

\begin{proof}[Proof of Theorem~\ref{thm: deg main}]
We consider each case separately.
\begin{enumerate}[(i)]
\item See Lemma~\ref{lem: Sigma0k is isolated}.
\item See the proof of Theorem~\ref{thm: nondeg main}(ii), but consider Theorem~\ref{thm:  C2} instead of Theorem~\ref{thm:  C2 nondeg}.
\item See the proof of Theorem~\ref{thm: nondeg main}(iii), but consider Theorem~\ref{thm:  C2} instead of Theorem~\ref{thm:  C2 nondeg}.
\item $\Sigma_\kap^{1}\setminus\Sigma_\kap^{1, {\rm nxt}}$ is countable by Corollary~\ref{cor.countC2}, and $\Sigma_\kap^{1, {\rm nxt}}$ is contained in the countable union of one-dimensional $C^2$ manifolds by Theorem~\ref{thm:  C2}. 
\item $\Sigma_\kap^{n-1, {\rm nxt}}$ is contained in the countable union of  $(n-1)$-dimensional $C^2$ manifolds by Theorem~\ref{thm:  C2}. 
On the other hand, $\dim_{\mathcal{H}} \Sigma_\kap^{n-1}\setminus\Sigma_\kap^{n-1, {\rm nxt}} \le n-2$ by Lemmas~\ref{lem.3rd_1} and \ref{lem.3rd_3}.
\item See Theorem~\ref{thm.main1}(iv). 
\end{enumerate}
This concludes the proof.
\end{proof}

\section{The Very Thin Obstacle Problem}
\label{sec.VTOP}

This section is dedicated to studying, what we have called, the very thin obstacle problem for $L_a$ when $a < 0$: a minimization problem like \eqref{eqn: Signorini problem}, but for $a \in (-1,0)$ and subject to a codimension two obstacle constraint.
Alternatively (see Section~\ref{sssec.FOP}), this problem corresponds to the fractional thin obstacle problem.
Namely, we consider
\begin{equation}
\label{eq.minpb}
\min_{w \in \mathscr{C}} \bigg\{ \int_{B_1} |\nab w|^2|y|^a \bigg\}, \quad\text{with}\quad  a \in (-1,0),
\end{equation}
where $\mathscr{C}$ is the convex subset of the Sobolev space $W^{1, 2}(B_1 , |y|^a)$ defined by 
\[
\mathscr{C} := \{ w \in W^{1,2}_0(B_1, |y|^a) + g : w(x',0, 0) \geq 0 \text{ and } w(x,-y) = w(x,y) \},
\]
given some boundary data $g \in C(B_1)$ (even with respect to $y$) such that $g|_{\pa B_1 \cap \{ x_n = y = 0 \}} \geq 0$.
The condition that $w$ is non-negative on the {\it very thin space} $\R^{n-1} \times \{ 0 \} \times \{ 0 \}$ needs to be understood in the trace sense, a priori.
Notice that since $a < 0$, the condition $ w \ge 0$ on $B_1'\times\{0\}\times\{0\}$ is relevant; the very thin space has non-zero $a$-harmonic capacity if and only if $a \in (-1,0)$.
Indeed, recalling the proof of Proposition~\ref{prop.case2}, the (``double'') trace operator $\tau : W^{1,2}(B_1,|y|^a)\to W^{s-\frac12,2}(B_1')\subset L^2(B_1')$ is well-defined and continuous. 

In this setting, the Euler--Lagrange equations characterizing the unique solution $u$ to \eqref{eq.minpb} are
\begin{equation}
\label{eq.VTOP}
\left\{ 
\begin{array}{rcll}
u(x', x_n,y) &\geq& 0 &\text{in } B_1 \cap \{ x_n = y = 0\}\\
L_a u &\leq& 0 &\text{in } B_1\\
u\,L_a u &=& 0 &\text{in } B_1\\
L_a u &=& 0 &\text{in } B_1 \setminus \Lambda(u) \\
u(x, y)& =& u(x, -y)& \text{in } B_1
\end{array}\right.
\end{equation}
where, as expected, 
\[
\Lambda(u) := \{ (x',0, 0) : u(x', 0,0) = 0 \}
\]
is the the {\it contact set}.
The {\it free boundary} here is the topological boundary in $\R^{n-1}$ of $\Lambda(u)$:
\[
\Gamma(u) := \pa \Lambda(u) \subset \R^{n-1} \times \{ 0\} \times \{0 \}.
\]

We close this introduction with a lemma that proves an analogous representation of $u$ to that in \eqref{eqn: Lap of u} for the solution to the thin obstacle problem.
Recall that (as defined in Subsection~\ref{ssec.notation}) $D_r$ denotes the two-dimensional disc centered at the origin of radius $r > 0$.

\begin{lemma}
\label{lem.operator}
Let $u$ be such that $L_a u = 0$ in $B_1\setminus\{x_n = y = 0\}$. 
Then, 
\[
L_a u (x', x_n, y)= f_a (x') \mathcal{H}^{n-1}\mres B_1',
\]
where 
\begin{align*}
f_a(x') & := \lim_{\vep\downarrow 0} \int_{\pa D_\vep} u_\nu |y|^a \, \d \sigma(x_n, y)\\
& = \lim_{\vep\downarrow 0} \int_{\pa D_\vep} \left(\frac{x_n}{\vep} \pa_n u(x', x_n, y) +\frac{y}{\vep} \pa_y u(x', x_n, y) \right)|y|^a \, \d \sigma(x_n, y).
\end{align*}
In particular, if $u$ is the solution to \eqref{eq.VTOP}, then 
\[
L_a u (x', x_n, y)= f_a(x')\mathcal{H}^{n-1}\mres \Lambda(u).
\]

\end{lemma}

\begin{proof}
For every $\varphi\in C_c^\infty(B_1)$,
\begin{align*}
\langle L_a u, \varphi \rangle & := -\int_{B_1} \nabla  u\cdot \nabla \varphi |y|^a\\
& = -\lim_{\vep \downarrow 0} \int_{B_1\cap \{x_n^2+y^2\ge \vep^2\}} \nabla  u\cdot \nabla \varphi |y|^a\\
& = \lim_{\vep \downarrow 0} \int_{B_1\cap \{x_n^2+y^2 =\vep^2\}}  u_\nu \, \varphi|y|^a\\
& = \int_{B_1'} f_a(x')\varphi(x',0,0)\, \d x',
\end{align*}
recalling that $L_a u = 0$ in $B_1\setminus\{x_n = y = 0\}$.
\end{proof}

In the following subsections, we prove a collection of results on the very thin obstacle problem, \eqref{eq.minpb} or, equivalently, \eqref{eq.VTOP}.

\subsection{A Non-local Operator} 

It is now well-known that the fractional Laplacian, or $s$-Laplacian, of a function $v$ defined on $\R^n$ can be reinterpreted as a weighted normal derivative of the $a$-harmonic extension of $v$ to the upper half-space $\R^{n+1}_+$ (see \cite{MO69,CSS08}).
In particular, if we let $\bar{v}$ denote this extension,
\[
- c_{n,s}(-\Delta)^s v(x) = \lim_{y \downarrow 0} y^a\pa_y \bar{v}(x,y).
\] 
This reinterpretation has been extremely useful in studying the thin obstacle problem (see \cite{CS07} and cf. \eqref{eqn: Lap of u}). 

In this subsection, we show that an analogous reinterpretation exists for a non-local operator of a function $v$ defined on $\R^{n-1}$ as a weighted normal derivative of an $a$-harmonic extension of $v$ to $\R^{n+1}$, and in the next subsection, we will use it to help us prove a collection of regularity results on the solution to \eqref{eq.minpb}.
For a given (sufficiency smooth) function $u : \R^{n+1} \to \R$, define
\begin{equation}
\label{eq.F0}
\mathcal{F}_a(u)(x') := \lim_{\vep\downarrow 0} \int_{\pa D_\vep} u_\nu(x', x_n, y) |y|^a \, \d \sigma(x_n,y).
\end{equation}
Hence, if $\bar{v} : \R^{n+1} \to \R$ is the unique $a$-harmonic extension that vanishes at infinity to $\R^{n+1}$ of a given function $v : \R^{n-1} \to \R$ that vanishes at infinity:
\begin{equation}
\label{eq.tildeuEL}
\left\{ 
\begin{array}{rcll}
L_a \bar{v} & = & 0 &\text{in } \R^{n+1}\setminus\{x_n = y = 0\}\\
\bar{v}(x', 0, 0) &=& v(x') &\text{on } \R^{n-1}\\
\lim_{|X|\to \infty} \bar{v}(X)&=& 0,
\end{array}\right.
\end{equation}
then we can define the non-local operator $\mathcal{I}_a$ on $v : \R^{n-1} \to \R$ by 
\begin{equation}
\label{eq.F}
\mathcal{I}_a(v) := \mathcal{F}_a(\bar{v}).
\end{equation}
Notice that $\bar{v}$ can be constructed as the unique solution to the following minimization problem:
\[
\min_{\mathscr{K}} \left\{\int_{\R^{n+1}} |\nab w|^2|y|^a \right\},\quad\text{with}\quad a \in (-1,0),
\]
where 
\[
\mathscr{K} := \{ w \in W^{1,2}(\R^{n+1},|y|^a) : w = v \text{ on } \{x_n = y = 0\},~\lim_{|X|\to \infty} w(X) = 0 \}.
\]

An important and interesting fact is $\mathcal{I}_a$ is nothing but the $-\frac{a}{2}$-Laplacian:

\begin{proposition}
\label{prop.mathcalF}
Let $\mathcal{I}_a$ be defined as in \eqref{eq.F}. 
Then, 
\[
\mathcal{I}_a = c_{n, a}(-\Delta)^{-\frac{a}{2}} \equiv c_{n, a} (-\Delta)^{s-\frac{1}{2}},
\]
for some positive constant $c_{n, a}$ depending only on $n$ and $a$. 
\end{proposition}

Before proving Proposition~\ref{prop.mathcalF}, notice that the Poisson kernel associated to \eqref{eq.tildeuEL} is 
\begin{equation}
\label{eq.Poissonkernel}
P_a(x', x_n, y) = C_{n, a} \frac{(x_n^2+y^2)^{-\frac{a}{2}}}{(|x'|^2+x_n^2+y^2)^{\frac{n-1-a}{2}}}.
\end{equation}
That is, if $v : \R^{n-1} \to \R$, then $v \ast_{x'} P(\,\cdot\,, x_n, y) = \bar{v}$.
Indeed, it is easy to see that $L_a P_a(x', x_n, y) = 0$ when $x_n^2+y^2 > 0$ and $P_a(x', 0, 0)$ is concentrated at $ x' = 0$.
Furthermore, since $P_a(x', r\cos\theta, r\sin\theta) = r^{-n+1}P_a(r^{-1}x', \cos\theta, \sin\theta)$, we deduce that $P_a$ is a multiple of the Dirac delta of the right dimensionality as $x_n^2+y^2 \downarrow 0$.

The intuition behind \eqref{eq.Poissonkernel} is as follows: the Poisson kernel for the fractional Laplacian can be thought as the Poisson kernel regular Laplacian extended to a fractional number of additional dimensions, $+a$ dimensions. 
In our case, we extend an additional dimension, not only in $y$, but also in $x_n$. 
So we are considering an $(1+a)$-dimensional extension starting from $n-1$ dimensions. 
That is, \eqref{eq.Poissonkernel} can be recovered from the Poisson kernel for the fractional Laplacian (see \cite{CS07}) by renaming the variable $y$ to $|(x_n, y)|$ (by Pythagoras) and replacing $a$ with $1+a$ and $n$ with $n-1$.

\begin{proof}[Proof of Proposition~\ref{prop.mathcalF}]
Thanks to \eqref{eq.Poissonkernel}, we have that 
\[
\bar{v}(x', x_n, y) = \int_{\R^{n-1}} v(z')P_a(x'-z', x_n, y) \,\d z'.
\]
In turn, 
\[
\mathcal{I}_a(v)(x') = \lim_{\vep\downarrow 0} \int_{\pa D_\vep} \pa_\nu \left(\int_{\R^{n-1}} v(z')P_a(x'-z', x_n, y) \,\d z'\right) |y|^a \, \d\sigma(x_n,y).
\]
Now since $P_a$ is radially symmetric in the $(x_n, y)$ variables,
\begin{align*}
\mathcal{I}_a(v)(x') & = \lim_{\vep\downarrow 0}  \pa_y \left(\int_{\R^{n-1}} v(z')P_a(x'-z', 0, y)\,\d z'\right)\bigg|_{y = \vep}\int_{\pa D_\vep}|y|^a \,\d\sigma(x_n,y)  \\
& = C \lim_{\vep\downarrow 0} \vep^{1+a}  \pa_y \left(\int_{\R^{n-1}} v(z')P_a(x'-z', 0, y)\,\d z'\right)\bigg|_{y = \vep} \\
&= C(-\Delta)^{-\frac{a}{2}} v(x'),
\end{align*}
where, in the last step, we have used that $P_a(x', 0, y)$ is the Poisson kernel for the fractional Laplacian of order $1+a$ in $n-1$ dimensions (see \cite[Sections 1 and 2]{CS07}).
\end{proof}

Thanks to Proposition~\ref{prop.mathcalF}, we can construct some useful H\"older regular barriers.

\begin{lemma}[H\"older Barriers]
\label{lem.HoldBarrier}
Let $\zeta(r):[0, \infty)\to [0, 1]$ be a smooth function with the following properties: $\zeta(r) \equiv 1$ for $0\le r\le 2$, $\zeta(r) \equiv 0$ for $r > 3$, and $\zeta'\le 0$. 
Set $h_\beta(x') := |x'|^\beta \zeta(|x'|)$ and define  
\begin{equation}
\label{eq.HoldBarrier}
\bar{h}_\beta(x', x_n, y) := \int_{\R^{n-1}} h_\beta(z')P_a(x'-z', x_n, y).
\end{equation}
Then, 
\[
\left\{ 
\begin{array}{rcll}
L_a\bar{h}_\beta &=& 0 &\text{in } \R^{n+1}\setminus\{ x_n = y = 0\}\\
\bar{h}_\beta(x', 0, 0) &=& |x'|^\beta &\text{on } B_1'\\
\bar{h}_\beta &\geq& c &\text{in } \pa B_1,
\end{array}\right.
\]
for some constant $c$ depending only on $n$ and $a$.
Moreover, $\bar{h}_\beta \in C^\gamma(B_1)$ for $\gamma := \min\{-a, \beta\}$.
\end{lemma}

\begin{proof}
First, observe that by the definition of $P_a$, $L_a \bar{h}_\beta = 0$ in $\R^{n+1}\setminus \{x_n = y = 0\}$ and $\bar{h}_\beta(x', 0, 0 ) = |x'|^\beta$ in $B_1'$. 
Second, by continuity, $\bar{h}_\beta \geq c > 0$ on $\pa B_1$ since $\bar{h}_\beta > 0$ on $\pa B_1 \cap \{x_n^2+y^2 > 0\}$ and $\bar{h}_\beta = 1$ on $\pa B_1\cap \{x_n^2+y^2 = 0\}$.
Finally, notice that if we fix $x_n = 0$, then $\bar{h}_\beta(x',0,y)$ is the $(1+a)$-harmonic extension of $h_\beta$ to $\R^n$ (see the proof of Proposition~\ref{prop.mathcalF}); note that $1+a \in (0,1)$. 
Namely, $\bar{h}_\beta(x', 0, y)$ is such that 
\begin{equation}
\label{eq.tildeuEL_2}
\left\{ 
\begin{array}{rcll}
L_{1+a} \bar{h}_\beta(x', 0, y) & = & 0 &\text{in } \R^{n}\setminus\{ y = 0\}\\
\bar{h}_\beta(x', 0, 0) &=& h_\beta(x') &\text{on } \R^{n-1}\\
\lim_{|(x', y)|\to \infty} \bar{h}_\beta(x', 0, y)&=& 0. 
\end{array}\right.
\end{equation}
Now by \cite[Proposition 2.3]{JN17}, we have that $\bar{h}_\beta$ is (locally) smooth in $x'$ and is (locally) $(-a)$-H\"{o}lder in the $y$ up to $\{ y = 0 \}$. 
Therefore, since $\bar{h}_\beta$ is radially symmetric in the $(x_n, y)$ variables, $\bar{h}_\beta \in C^\gamma(B_1)$, as desired.
\end{proof}

We conclude this subsection with a higher regularity result.

\begin{lemma}
\label{lem.continside}
If $u\in L^\infty(B_1)$ is such that $L_a u = 0$ in $B_1\setminus\{x_n = y = 0\}$ and $u(\,\cdot\,, 0, 0)\in C^{k+\beta}({B_{1}'})$ for $k\in \N\cup \{0\}$ and $\beta \in (0, 1]$, then for $\gamma := \min\{-a, \beta\}$, 
\[
[D^{k}_{x'}  u]_{C^\gamma(B_{1/2})}\le C\left(\|u\|_{L^\infty(B_1)} + \|u(\cdot,0,0)\|_{C^{k+\beta}(B_1')}\right),
\]
for some constant $C$ depending only on $n$, $a$, $k$, and $\beta$.
Moreover, if $u(\cdot, 0, 0)$ is continuous, then $u$ is continuous. 
\end{lemma}

\begin{proof} 
The proof follows simply by combining interior estimates for the operator $L_a$ and a barrier argument on $\{x_n = y = 0\}$, with the barrier $\bar{h}_\beta$ constructed in Lemma~\ref{lem.HoldBarrier}. 

Suppose $k = 0$ and let $\bar C$ be a constant such that
\[
\bar{C} \geq [u(\cdot, 0, 0)]_{C^\beta(B_1')} \qquad\text{and}\qquad \bar{C} \bar{h}_\beta \geq \|u\|_{L^\infty(B_1)} \quad\text{on}\quad \pa B_{1/2}.
\]
Then, $\bar{C}\bar{h}_\beta$ serves both as a barrier from above and from below at any point $x'\in B_{1/2}'$. 
This barrier combined with interior estimates for $a$-harmonic functions (see, e.g, \cite[Proposition 2.3]{JN17}) directly yields the desired estimate (as in \cite{MS06}, for instance). 

If $k \ge 1$, we apply the previous result iteratively, starting with $\beta = 1$, to the derivatives $D_{x'}^\alpha u$, up to a ball $B_{2^{-k-1}}$, and finish by a covering lemma. 

To prove the last part, let us suppose that $u(\cdot, 0, 0)$ is continuous.
We want to show that $u$ is continuous as well. 
Let us extend $u$ to the whole space with any cutoff function and consider $v(x', x_n, y) := u(\cdot, 0, 0)*P_a(\cdot, x_n, y)$. 
Notice that since $u(\cdot, 0, 0)$ is continuous, $v$ is continuous as well. 
Then, $u = v + w$ where $w$ satisfies $w(\cdot, 0, 0) \equiv 0$ and $L_a w = 0$ in $B_1\setminus \{x_n = y = 0\}$. 
Thus, by the above result, $w$ is smooth and therefore, $u$ is continuous. 
\end{proof} 

\begin{corollary}
\label{cor.ucontVTOP}
Let $u\in L^\infty(B_1)$ be a solution to \eqref{eq.VTOP}. 
Then, $u$ is continuous in $B_1$. 
\end{corollary}

\begin{proof}
The continuity on the very thin space follows from a standard argument in obstacle type problems (see \cite[Theorem~1]{Caf98}) using super-$a$-harmonicity of the solution and the mean value formula on the thin space for the operator $L_a$. 
The continuity in $B_1$ then follows from Lemma~\ref{lem.continside}. 
\end{proof}

\subsection{Basic Estimates}

In this subsection, we prove some regularity properties of solutions to \eqref{eq.minpb}.
Our first result contains two classical estimates: an energy estimate and an $L^\infty$ estimate.

\begin{lemma}
\label{lem.boundedsolution}
Let $u$ be a solution to \eqref{eq.minpb} and \eqref{eq.VTOP}. 
Then, 
\begin{equation}
\label{eq.solisbounded0}
\|u\|_{W^{1, 2}(B_{1/2}, |y|^a)}\leq C\|u\|_{L^2(B_1,|y|^a)}
\end{equation}
and
\begin{equation}
\label{eq.solisbounded}
\|u\|_{L^\infty(B_{1/2})}\leq C\|u\|_{L^2(B_1,|y|^a)},
\end{equation}
for some constant $C$ depending only on $n$ and $a$. 
\end{lemma}

\begin{proof}
This is standard (see \cite{AC04} or Lemma~\ref{lem:L2Linfty}). 
\end{proof} 

Next, we prove the solutions are Lipschitz and semiconvex in the directions parallel to the very thin space. 

\begin{lemma}
\label{lem.Lip}
Let $u$ be a solution to \eqref{eq.VTOP}. 
Then, for all $\be \in \{x_n = y = 0\} \cap \mathbb{S}^n$,
\begin{equation}
\label{eq.Lipschitz}
\|\pa_{\be} u\|_{L^\infty(B_{1/4})} \leq C\|u\|_{L^2(B_1,|y|^a)}
\end{equation}
and
\begin{equation}
\label{eq.semiconv2}
\inf_{B_{1/8}} \pa_{\be\be} u \geq -C\|u\|_{L^2(B_1,|y|^a)},
\end{equation}
for some constant $C$ depending only on $n$ and $a$.
\end{lemma}

\begin{proof}
The proofs of these estimates are identical to the proofs of Lemmas~\ref{lem: Bern} and \ref{lem: Bern X2}. 
That said, to get \eqref{eq.Lipschitz}, we need to use the incremental quotients $((u(x+h\be) - u(x))/h)^-$ and $((u(x-h\be) - u(x))/h)^-$, in the spirit of Lemma~\ref{lem: Bern X2}, and the continuity of $u$ (proved in Corollary~\ref{cor.ucontVTOP}). 
\end{proof}

An easy corollary of Lemma~\ref{lem.Lip} is that $u$ is $C^{-a}$.

\begin{corollary}
\label{cor.ca}
Let $u$ be the solution to \eqref{eq.VTOP}. 
Then,
\begin{equation}
[u]_{C^{-a}(B_{1/2})} \le C\|u\|_{L^\infty(B_1)},
\end{equation}
for some constant $C$ depending only on $n$ and $a$. 
\end{corollary}
\begin{proof}
This is an immediate consequence of Lemmas~\ref{lem.Lip} and \ref{lem.continside}. 
\end{proof}

Using Corollary~\ref{cor.ca}, we now prove an $L^\infty$ estimate on $\mathcal{F}_a(u)$.

\begin{lemma}
\label{lem.boundedoperator}
Let $u$ be the solution to \eqref{eq.VTOP} and $\mathcal{F}_a$ be as in \eqref{eq.F0}.
Then, 
\[
\|\mathcal{F}_a(u)\|_{L^\infty(B_{1/2}')}\le C \|u\|_{L^\infty(B_1)},
\]
for some constant $C$ depending only on $n$ and $a$. 
That is, $L_a u$ is a locally bounded, absolutely continuous measure, with respect to $\mathcal{H}^{n-1}$, supported on $\{x_n = y = 0\}$. 
\end{lemma}

\begin{proof}
Recall, if $L_a u = 0$ in $B_{r}(X_\circ)$, then 
\begin{equation}
\label{eq.est1}
\|\nabla_x u\|_{L^\infty(B_{r/2}(X_\circ))} \le C r^{-1}\osc_{B_{r}(X_\circ)} u
\end{equation}
and 
\begin{equation}
\label{eq.est2}
\||y|^a\pa_y u\|_{L^\infty(B_{r/2}(X_\circ))} \le C r^{a-1}\osc_{B_{r}(X_\circ)} u.
\end{equation}
(See, e.g., \cite[Proposition 2.3]{JN17}.)
Now let $x'\in B_{1/2}'$.
And assume that $\mathcal{F}_a(u)(x')< 0$, so that $u(x') = 0$ (otherwise, there is nothing to prove). 
We claim that 
\begin{equation}
\label{eq.weobtain}
\lim_{\vep\downarrow 0} \left|\int_{\pa D_\vep} \left(\frac{x_n}{\vep} \pa_n u(x', x_n, y) +\frac{y}{\vep} \pa_{y}u(x', x_n, y) \right)|y|^a\,\d\sigma(x_n, y)\right|\le C\|u\|_{L^\infty(B_1)}.
\end{equation}
From \eqref{eq.est1} and \eqref{eq.est2} and by Corollary~\ref{cor.ca}, rescaled to $B_\vep(x',x_n,y)$, we have that
\[
\sup_{\pa D_\vep}|\pa_n u|\le C\vep^{-1-a}\|u\|_{L^\infty(B_1)}\quad\textrm{ and }\quad\sup_{\pa D_\vep} ||y|^a \pa_y u|\le C\vep^{-1}\|u\|_{L^\infty(B_1)}.
\]
Hence, \eqref{eq.weobtain}, as desired. 
\end{proof}

The following theorem proves that $u$ is $C^{1, \tau}$ in the directions parallel to $\{x_n = y = 0\}$. 

\begin{theorem}
\label{thm.C1tau}
Let $u$ be the solution to \eqref{eq.VTOP}. 
Then, for all $\boldsymbol{e}'\in \{x_n = y = 0\}\cap \mathbb{S}^n$,
\[
[\pa_{\boldsymbol{e}'} u]_{C^{\tau}(B_{1/2})} \le C \|u\|_{L^2(B_1, |y|^a)},
\] 
for some constants $\tau>0$ small and $C$ depending only on $n$ and $a$. 
\end{theorem}

\begin{proof}
Define the cut-off function $\xi(X) := \zeta(|x'|^2)\zeta(x_n^2+y^2)$ where
\[
\zeta: [0, \infty)\to [0, 1],\quad\zeta' \le 0, \quad\zeta \equiv 1 \text{ in } [0, 1/8], \quad\text{and}\quad \zeta \equiv 0 \text{ in } [1/4, \infty),
\]
and set $\hat{u}(X) := u(X)\xi(X)$ in $B_1$ and $\hat{u}(X) \equiv 0$ outside of $B_1$.  
Notice that
\[
L_a \hat{u} = u L_a \xi  + |y|^a\nabla \xi \cdot \nabla u =: |y|^a \hat{f}(X) \quad\text{in}\quad \R^{n+1}\setminus \{x_n = y = 0\}.
\]

Now let $\hat{w}$ be such that 
\begin{equation}
\label{eq.w}
\left\{ 
\begin{array}{rcll}
L_a \hat{w} & = & |y|^a \hat{f} &\text{in } \R^{n+1}\setminus\{ x_n = y = 0\}\\
\hat{w}(x', 0, 0) &=& 0 &\text{on } \R^{n-1}\\
\lim_{|X|\to \infty} \hat{w}(X)&=& 0. 
\end{array}\right.
\end{equation}
Clearly, $L_a \hat{w} = 0$ in $B_{1/8} \setminus\{x_n = y = 0\}$, so that by Lemma~\ref{lem.continside}, $\hat{w}$ is smooth in $B_{1/16}$.
Hence, $\mathcal{F}_a(\hat{w})$ is smooth in $B_{1/16}'$.

Observe that 
\[
\left\{ 
\begin{array}{rcll}
L_a (\hat{u} - \hat{w}) & = & 0 &\text{in } \R^{n+1}\setminus\{ x_n = y = 0\}\\
\hat{u} - \hat{w}  &\geq& 0 &\text{on } \R^{n-1} \times \{ 0\} \times \{0\}. 
\end{array}\right.
\]
Moreover, by the symmetries of $\xi$ in the $(x_n, y)$ directions, we have that
\[
\left\{ 
\begin{array}{rcll}
\mathcal{F}_a (u\xi)(x') & = & 0 &\text{if } u(x', 0, 0) > 0\\
\mathcal{F}_a (u\xi)(x')  &\leq& 0 &\text{if } u(x', 0, 0) = 0;
\end{array}\right.
\] 
so 
\[
\left\{ 
\begin{array}{rcll}
\mathcal{F}_a (\hat{u} -\hat{w})(x') & = & -\mathcal{F}_a (\hat{w})(x') &\text{if } (\hat{u}-\hat{w})(x', 0, 0) > 0\\
\mathcal{F}_a (\hat{u} -\hat{w})(x')  &\leq& -\mathcal{F}_a (\hat{w})(x') &\text{if } (\hat{u}-\hat{w})(x', 0, 0) = 0.
\end{array}\right.
\] 
Alternatively, thanks to Proposition~\ref{prop.mathcalF}, $U(x') := (\hat{u} - \hat{w})(x', 0, 0)$ solves the following obstacle problem 
\begin{equation}
\label{eq.U}
\left\{ 
\begin{array}{rcll}
U & \ge & 0 & \text{in } \R^{n-1},\\
(-\Delta)^{-\frac{a}{2}} U & = & -C\mathcal{F}_a(\hat{w}) &\text{in } \{x': U(x') > 0\},\\
(-\Delta)^{-\frac{a}{2}} U & \le  & -C\mathcal{F}_a(\hat{w}) &\text{in } \R^{n-1}\\
\lim_{|x'|\to \infty} U(x')&=& 0. 
\end{array}\right.
\end{equation}
By \cite[Proposition 2.2]{CRS17}, recalling that $\mathcal{F}_a(\hat{w})$ is smooth in $B_{1/16}'$ and that $u$ is Lipschitz \eqref{eq.Lipschitz} and semiconvex \eqref{eq.semiconv2}, we deduce that $U\in C^{1,\tau}(B_{1/32}')$. 
And via a simple covering argument, $U \in C^{1,\tau}(B_{3/4}')$.

The theorem now follows from Lemma~\ref{lem.continside}. 
\end{proof}

The last result of this subsection is a H\"{o}lder regularity result for the $X$-directional derivative of $u$ for $X \in B_1$.

\begin{corollary}
\label{cor.welldefined}
Let $u$ be the solution to \eqref{eq.VTOP}. 
Then, $X \cdot \nabla u$ is continuous in $B_1$. In particular, 
\[
\|X \cdot \nabla u\|_{C^{\bar\tau}(B_{1/2})} \leq C\|u\|_{L^\infty(B_1)},
\]
for some constants $\bar\tau >0$ small and $C$ depending only on $n$ and $a$.
\end{corollary}

\begin{proof}
Let $X_\circ \in \Lambda(u)$.
By \eqref{eq.est1}, \eqref{eq.est2}, and Corollary~\ref{cor.ca},
\[
\sup_{B_{r/2}(X_\circ)} |x_n \pa_{n}u| + |y\pa_yu| \leq Cr^{-a}.
\]
This, Theorem~\ref{thm.C1tau}, the $C^{1, \tau}$ regularity of $u$ in $x'$, and interior estimates for $a$-harmonic functions in $B_1 \setminus \Lambda(u)$ (see, e.g., \cite{JN17}) yield the desired result (again, as in \cite{MS06}, for instance).
\end{proof}

\subsection{Monotonicity Formulae}

In this subsection, we prove that $u$ has the same monotonicity properties as its cousin, the solution to the thin obstacle problem.
We start with Almgren's frequency function.

\begin{lemma}
\label{lem.MonFreq}
Let $u$ be the solution to \eqref{eq.VTOP} and $0 \in \Lambda(u)$. 
Then, Almgren's frequency function on $u$
\[
r\mapsto N(r, u) := \frac{r\int_{B_r} |\nabla u|^2|y|^a}{\int_{\pa B_r} u^2|y|^a}
\]  
is non-decreasing for $0< r < 1$. 
Moreover, $N(u, r)\equiv \lambda$ if and only if $u$ is homogeneous of degree $\lambda$ in $B_1$, i.e., $x\cdot \nabla u - \lambda u = 0$ in $B_1$.
\end{lemma}

\begin{proof}
The proof of this lemma is standard, and follows the lines of the proof that Almgren's frequency function is monotone on solutions the thin obstacle problem. 
Nonetheless, some of the steps now require justification because of the inherent lower regularity of the very thin obstacle problem.
Justifying these steps is where Theorem~\ref{thm.C1tau}\,---\,more precisely, Corollary~\ref{cor.welldefined}\,---\,comes into play. 

Set, for $0<r < 1$,
\[
D(r) = D(r, u) := \int_{B_r} |\nabla u |^2|y|^a \quad\text{and}\quad H(r) = H(r, u) := \int_{\pa B_r}u^2|y|^2, 
\]
so that $N(r) := N(r, u) = rD(r)/H(r)$. 
Notice that both quantities are pointwise defined, since $u\in W^{1, 2}(B_1, |y|^a)\cap C^{-a}_{\rm loc}(B_1)$, and in particular, $N(r)$ is continuous. 
Following the proof of Proposition~\ref{prop: Almgren} (where we remark that $D$ and $H$ were defined differently), we immediately find that
\[
H'(r) = \frac{n+a}{r} H(r) + 2 \int_{\pa B_r} u u_\nu |y|^a 
\]
and 
\begin{equation}
\label{eq.secterm}
D'(r) = \frac{n+a-1}{r} D(r) + \frac{2}{r}\int_{B_r}\nabla u\cdot \nabla (X\cdot \nabla u)|y|^a.
\end{equation}
By Corollary~\ref{cor.welldefined}, the quantity $H'(r)$ is well-defined pointwise (and finite). 
On the other hand, $D(r)$ is absolutely continuous, being the integral in $B_r$ of an integrable function, so that its derivative is well-defined pointwise and  almost everywhere finite (and non-negative). 
Thus, $N(r, u)$ is locally absolutely continuous.

Integrating by parts in the second term of \eqref{eq.secterm}, we deduce that 
\[
\frac{1}{r}\int_{B_r}\nabla u\cdot \nabla (X\cdot \nabla u)|y|^a =  \int_{\pa B_r} u_\nu^2|y|^a - \frac{1}{r}\int_{B_r}  (X\cdot \nabla u) L_a u.
\]
Now notice that $L_a u$ is a finite measure concentrated on $\{x_n = y = 0\}$ (see Lemma~\ref{lem.boundedoperator}), and $X\cdot \nabla u $ is continuous (see Corollary~\ref{cor.welldefined}). 
Moreover, by the proof of Corollary~\ref{cor.welldefined}, $X\cdot \nabla u  = 0$ whenever $L_a u< 0$. 
In turn, the second term above vanishes. 
On the other hand, by the continuity of $X\cdot \nabla u$, the first term is well-defined pointwise. 
Hence, 
\[
D'(r) = \frac{n+a-1}{r} D(r) + 2 \int_{\pa B_r} u^2_\nu |y|^a.
\]

Integrating by parts again, observe that 
\[
D(r) = \int_{B_r} |\nabla u|^2|y|^a = \int_{\pa B_r} u u_\nu |y|^a  - \int_{B_r} u L_a u = \int_{\pa B_r} u u_\nu |y|^a,
\]
where the term $\int_{B_r} u L_a u = 0$ arguing as before: $u$ is continuous (Corollary~\ref{cor.ca}) and vanishes whenever $L_a u < 0 $, and $L_a u$ is a finite measure concentrated on $\{x_n = y = 0\}$ (Lemma~\ref{lem.boundedoperator}). 

Combing the above estimates, we determine that 
\[
\frac{N'(r)}{N(r)} = \frac{D'(r)}{D(r)}-\frac{H'(r)}{H(r)}+\frac{1}{r} = 2\left(\frac{\int_{\pa B_r} u_\nu^2|y|^a}{\int_{\pa B_r} u u_\nu|y|^a}-\frac{\int_{\pa B_r} u u_\nu|y|^a}{\int_{\pa B_r} u^2|y|^a}\right)\ge 0,
\]
by the Cauchy--Schwarz inequality, which yields the monotonicity of $N(r, u)$. 
Analyzing the equality case, we see that if $N(r)$ is constant, then $u$ is homogeneous of degree $N(r)$ (see, e.g., \cite[Lemma 1]{ACS08}). 
\end{proof}

Next we prove a Monneau-type monotonicity formula. 

\begin{lemma}
\label{lem.HMon_VTOP}
Let $u$ be the solution to \eqref{eq.VTOP} and $0 \in \Lambda(u)$.
Given $\lambda \geq 0$, define
\begin{equation}
\label{eq.HMon_VTOP}
H_\lambda(r, u) := \frac{1}{r^{n+a+2\lambda}}\int_{\pa B_r} u^2|y|^a.
\end{equation}
For all $0\le \lambda\le N(0^+, u)$, the map $r\mapsto H_\lambda(r, u)$ is non-decreasing. 
\end{lemma}

\begin{proof}
Arguing as in the proof of Lemma~\ref{lem.MonFreq}, using $\int_{B_r} u L_a u = 0$, we compute that
\begin{equation}
\label{eq.HMon2_2}
\frac{H_\lambda'}{H_\lambda}(r, u) = \frac{2}{r}(N(r, u) -\lambda).
\end{equation} 
(See, also, the proof of Lemma~\ref{lem.HMon}.) 
The lemma then follows from Lemma~\ref{lem.MonFreq}: $N(r, u) \ge N(0^+, u) \ge \lambda$.
\end{proof}

Now we move to the Weiss energies.

\begin{lemma}
\label{lem.WMon_VTOP}
Let $u$ be the solution to \eqref{eq.VTOP} and $0 \in \Lambda(u)$. 
Given $\lambda \geq 0$, define
\begin{equation}
\label{eq.WMon_VTOP}
W_\lambda(r, u) := H_\lambda(r,u)(N(r,u)-\lambda).
\end{equation}
For all $\lambda \geq 0$, the map $r\mapsto W_\lambda(r, u)$ is non-decreasing.
\end{lemma}

\begin{proof}
Arguing as in the proof of Lemma~\ref{lem.MonFreq}, using $\int_{B_r} u L_a u = 0$, an explicit computation directly yields
\[
\frac{\d}{\d r} W_\lambda(r, u) = \frac{2}{r^{n+1+a+2\lambda}}\int_{\pa B_r} (X \cdot \nab u - \lambda u)^2 |y|^a \ge 0,
\]
as desired. 	
\end{proof}

We close this subsection with a useful limit.

\begin{lemma}
\label{lem.HMon_VTOP 2}
Let $u$ be the solution to \eqref{eq.VTOP} and $0 \in \Lambda(u)$.
Suppose that $N(0^+, u) = \lambda^*$. 
Given $\lambda > \lambda^\ast$,
\[
\lim_{r \downarrow 0} H_\lambda(r, u) = +\infty.
\]
\end{lemma}

\begin{proof}
Suppose, to the contrary, we can find a sequence of radii $r_\ell \downarrow 0$ such that $H_\lambda(r_\ell,u) \leq C$ for all $\ell \in \mathbb{N}$.
Then, for $\mu \in (\lambda^\ast,\lambda)$, $H_\mu(r_\ell,u) \to 0$ as $\ell \to \infty$.
Hence, as $W_\mu(r,u) \geq -\mu H(r,u)$ for all $r > 0$,
\[
\liminf_{\ell \to \infty} W_\mu(r_\ell,u) \geq \liminf_{\ell \to \infty} - \mu H_\mu(r_\ell,u) = 0.
\]
By the monotonicity of $r \mapsto W_\mu(r,u)$, Lemma~\ref{lem.WMon_VTOP}, we find that
\[
N(r_\ell,u) \geq \mu,
\]
for all $\ell \in \mathbb{N}$.
But this is impossible: $\mu > \lambda^\ast := N(0^+,u)$.
\end{proof}

\subsection{Blow-up Analysis and Consequences}

This subsection is dedicated to the analysis of blow-ups of $u$ at points $X_\circ \in \Lambda(u)$.
As such, for $X_\circ \in \Lambda(u)$, define
\begin{equation}
\label{eq.def:VTOP BU}
u_{X_\circ,r}(X) := u(X_\circ + rX) \quad\text{and}\quad \tilde{u}_{X_\circ,r} := \frac{u_{X_\circ,r}}{\|u_{X_\circ,r}\|_{L^2(\pa B_1, |y|^a)}}.
\end{equation}

We start by showing that blow-ups exists and are global, homogeneous solutions to \eqref{eq.VTOP}.

\begin{lemma}
\label{lem: blowups VTOP}
Let $u$ be the solution to \eqref{eq.VTOP} and suppose that $X_\circ \in \Lambda(u)$.
Let $\tilde{u}_{X_\circ,r}$ be as in \eqref{eq.def:VTOP BU}.
Then, for every sequence $r_j \downarrow 0$, there exists a subsequence $r_{j_\ell}\downarrow 0$ such that 
\begin{equation}
\label{eq.blowuptilde}
\tilde{u}_{X_\circ,r_{j_\ell}} \rightharpoonup \tilde u_{X_\circ,0} \quad \text{in}\quad W^{1,2}(B_1, |y|^a) \quad\text{as}\quad \ell \to \infty
\end{equation}
for some $\tilde  u_{X_\circ,0} \in W^{1,2}(B_1, |y|^a)$. 
Moreover, $\tilde  u_{X_\circ,0}\not\equiv 0$ is a global, homogeneous solution to a very thin obstacle problem with zero obstacle.
If, in addition, $u$ is homogeneous, then $\tilde u_{X_\circ,0}$ is translation invariant with respect to $X_\circ$.
\end{lemma}

\begin{proof}
By Lemma~\ref{lem.MonFreq}, we see that given any sequence $r_j \downarrow 0$, the family $\{ \tilde{u}_{X_\circ,r_{j}} \}_{j \in \mathbb{N}}$ is uniformly bounded in $W^{1,2}(B_1, |y|^a)$.
Hence, there is a subsequence $r_{j_\ell}\downarrow 0$ such that
\[
\tilde{u}_{X_\circ,r_{j_\ell}} \rightharpoonup \tilde u_{X_\circ,0} \quad \text{in}\quad W^{1,2}(B_1, |y|^a).
\]
As $\|\tilde{u}_{X_\circ,r_{j_\ell}}\|_{L^2(\pa B_1, |y|^a)} = 1$,
\[
\|\tilde u_{X_\circ,0}\|_{L^2(\pa B_1, |y|^a)} = 1.
\]
Clearly, $\tilde u_{X_\circ,0} \not\equiv 0$.

Since the family of functions $\{ \tilde{u}_{X_\circ,r_{j_\ell}} \}_{\ell \in \N}$ is locally uniformly H\"{o}lder continuous (by Corollary~\ref{cor.ca}), we have that $\tilde{u}_{X_\circ,r_{j_\ell}} \to \tilde u_{X_\circ,0}$ locally uniformly.
Moreover, $L_a\tilde{u}_{X_\circ,r_{j_\ell}} \rightharpoonup L_a \tilde u_{X_\circ,0}$ (which is non-positive) weakly* as measures (see, e.g., the proof of Proposition~\ref{prop.case1}).
Therefore, for every $\rho > 0$,
\begin{equation}
\label{eqn: uLau VTOP BU}
0 = \int_{B_\rho} \tilde{u}_{X_\circ,r_{j_\ell}} L_a \tilde{u}_{X_\circ,r_{j_\ell}} |y|^a \to \int_{B_\rho} \tilde u_{X_\circ,0} L_a \tilde u_{X_\circ,0} |y|^a \quad\text{as}\quad \ell \to \infty,
\end{equation}
so that, since $\tilde u_{X_\circ,0} L_a \tilde u_{X_\circ,0}\le 0$,
\[
\tilde u_{X_\circ,0} L_a \tilde u_{X_\circ,0} = 0 \quad\text{in}\quad \R^{n+1}.
\]
This, together with the uniform convergence of $\tilde u_{X_\circ, r_{j_\ell}}$ and the weak$^*$ convergence of $L_a\tilde u_{X_\circ, r_{j_\ell}}$ to $L_a\tilde  u_{X_\circ, 0}$ directly yields that $\tilde u_{X_\circ, 0}$ is a global solution to the very thin obstacle problem with zero obstacle. 

Furthermore, from the local uniform continuity of $X \cdot \nab \tilde{u}_{x_\circ,r_\ell} $ given by Corollary~\ref{cor.welldefined},
\[
\int_{\pa B_\rho} \tilde{u}_{X_\circ,r_{j_\ell}} (X \cdot \nab \tilde{u}_{X_\circ,r_{j_\ell}})|y|^a \to \int_{\pa B_\rho}\tilde  u_{X_\circ,0} (X \cdot \nab \tilde u_{X_\circ,0})|y|^a  \quad\text{as}\quad \ell \to \infty.
\]
Consequently, for all $\rho > 0$,
\[
N(\rho,\tilde u_{X_\circ,0}) = \lim_{r_{j_\ell} \downarrow 0} N(\rho,\tilde{u}_{X_\circ,r_{j_\ell}}),
\]
and, in particular,
\[ 
N(\rho,\tilde u_{X_\circ,0})  = N(0^+,u(X_\circ + \,\cdot\,)) =: \lambda_{X_\circ}
\]
for all $\rho > 0$.
(By scaling, $\lim_{r_{j_\ell} \downarrow 0} N(\rho,\tilde{u}_{X_\circ,r_{j_\ell}}) = \lim_{r_{j_\ell} \downarrow 0} N(\rho r_{j_\ell},u(X_\circ + \,\cdot\,))$.)
Hence, by Lemma~\ref{lem.MonFreq}, $\tilde u_{X_\circ,0}$ is $\lambda_{X_\circ}$-homogeneous, and the first part of the proof is complete.  

Now assume that $u$ is $\lambda$-homogeneous. 
Then, 
\[
\int_{E} \nabla \tilde{u}_{X_\circ,r_{j_\ell}}(X) \cdot (X_\circ + r_{j_\ell}X)|y|^a = \int_{E} \lambda r_{j_\ell} \tilde{u}_{X_\circ,r_{j_\ell}}(X)|y|^a
\]
for any compact set $E \subset B_1$.
In turn, as $\tilde{u}_{X_\circ,r_{j_\ell}} \rightharpoonup \tilde u_{X_\circ,0}$ weakly in $W^{1,2}(B_1, |y|^a)$, taking $r_{j_\ell} \downarrow 0$, we find that
\begin{equation}
\label{eqn: trans inv}
X_\circ \cdot \nab \tilde u_{X_\circ,0}(X) = 0
\end{equation}
for almost every $X \in B_1$.
Finally, by Corollary~\ref{cor.welldefined} and the $\lambda_{X_\circ}$-homogeneity of $\tilde u_{X_\circ,0}$ established above, we see that \eqref{eqn: trans inv} holds for all $X \in \R^{n+1}$.
\end{proof}

Just as we did in the thin obstacle setting, we define the nodal set of a solution $u$ to \eqref{eq.VTOP}:
\begin{equation}
\label{eq.nodal VTOP}
\mathcal{N}(u) := \{ (x',0,0) : u(x',0,0) = |\nab_{x'}u(x',0,0)| = f_a(x') = 0 \}
\end{equation}
where $f_a$ is defined as in Lemma~\ref{lem.operator}.

In the following result, we prove an estimate on the size of the points whose blow-ups have spines 
\[
L(\tilde u_{X_\circ,0}) := \{ \xi' \in \R^{n-1} : \xi' \cdot \nab_{x'} \tilde u_{X_\circ,0}(x',0,0) = 0 \text{ for all }x' \in \R^{n-1} \}
\]
with a certain dimensional bound.

\begin{proposition}
\label{prop.claim}
Let $u$ be a solution to \eqref{eq.VTOP}. Then, 
\begin{equation}
\label{eq.claimHaus}
\dim_\mathcal{H}(\{ X_\circ \in \mathcal{N}(u) : \dim L(\tilde u_{X_\circ,0}) \leq d \text{ for all blow-ups } \tilde u_{X_\circ,0} \}) \leq d,
\end{equation}
for any $d\in \{0,\dots,n-1\}$. Moreover, if $d = 0$, the previous set is countable. 
\end{proposition}

\begin{proof}
The proof follows the first half of the proof of \cite[Theorem 1.3]{FoSp18}; and so, we have to check that the assumptions of \cite[Theorem 3.2]{W97} are fulfilled.
In particular, we argue in parallel to \cite[Section 8.1]{FoSp18}. 

Define the upper semicontinuous function $f : B'_1 \to \R^+$ by
\[
f(x_\circ') := 
\begin{cases}
N(0^+,u(X_\circ + \,\cdot\,)) &\text{if } X_\circ \in \mathcal{N}(u)\\
0 &\text{if } X_\circ \notin \mathcal{N}(u),
\end{cases}
\]
and for any $x_\circ' \in B_1'$, let $\mathscr{G}_{x_\circ'}$ be the family of upper semicontinuous functions $g: \R^{n-1} \to \R^+$ given by
\[
g(z') :=
\begin{cases}
N(0^+,\tilde u_{X_\circ,0}(Z + \,\cdot\,)) &\text{if } X_\circ \in \mathcal{N}(u)\\
0 &\text{if } X_\circ \notin \mathcal{N}(u)
\end{cases}
\]
where $\tilde u_{X_\circ,0}$ is a possible blow-up limit of $u$ at $X_\circ = (x_\circ',0,0)$ (as produced in Lemma~\ref{lem: blowups VTOP}), and of course, $Z = (z',0,0)$.
Observe, arguing as in \cite[Lemma 5.2]{FoSp18}, that for all $g \in \mathscr{G}_{x_\circ'}$,
\[
\text{if $g(z') = g(0)$, then $g(z' + \tau x') = g(z' + x')$ for all $x' \in \R^{n-1}$ and $\tau > 0$};
\]
that is, $g$ is {\it conical}, following the definitions used in \cite[Section 8.1]{FoSp18} and \cite{FMS15}. 

Furthermore, let $\{g_j\}_{j \in \N} \subset \mathscr{G}_{x_\circ'}$.
For each $g_j$, we have an associated blow-up $\tilde u_{X_\circ,0,j}$ which has $L^2(\pa B_1,|y|^a)$-norm equal to $1$.
And arguing as in Lemma~\ref{lem: blowups VTOP} and then applying a diagonal argument, we can find a subsequence $\{\tilde u_{X_\circ,0,j_\ell}\}_{\ell \in \N}$ that converges weakly in $W^{1,2}(B_1,|y|^a)$ and locally uniformly in $C^{-a}(B_1)$ to a blow-up of $u$ at $X_\circ$.
Call $ \tilde u_{X_\circ,0}^{(\infty)}$ this blow-up and define
\[
g_\infty(z') := 
\begin{cases}
N(0^+,\tilde u_{X_\circ,0}^{(\infty)}(Z + \,\cdot\,)) &\text{if } X_\circ \in \mathcal{N}(u)\\
0 &\text{if } X_\circ \notin \mathcal{N}(u).
\end{cases}
\]
By construction, $g_\infty \in \mathscr{G}_{x_\circ'}$.
Now given any convergent sequence $x'_\ell \to x'_\infty \in \R^{n-1}$ as $\ell \to \infty$, by Lemma~\ref{lem.MonFreq} and the upper semicontinuity of the frequency,
\[
\begin{split}
\limsup_{\ell \to \infty} N(0^+,\tilde u_{X_\circ,0,j_\ell}(X_\ell+\,\cdot\,)) &\leq \inf_{\rho > 0} \limsup_{\ell \to \infty}N(\rho, \tilde u_{X_\circ,0,j_\ell}(X_\ell+\,\cdot\,))\\
&= \inf_{\rho > 0} N(\rho, \tilde u_{X_\circ,0}^{(\infty)}(X_\infty+\,\cdot\,))\\
&= N(0^+, \tilde u_{X_\circ,0}^{(\infty)}(X_\infty+\,\cdot\,)).
\end{split}
\]
In turn,
\[
\limsup_{\ell \to \infty} g_{j_\ell}(x'_\ell) \leq g_\infty(x'_\infty),
\]
and $\mathscr{G}_{x_\circ'}$ is a class of {\it compact conical} functions (see \cite[Section 8.1]{FoSp18} and \cite[Definition 3.3]{FMS15}).
Like before, $X_\ell = (x_\ell',0,0)$ and $X_\infty = (x_\infty',0,0)$.

In addition, we need to check the structural hypotheses of \cite[Theorem 3.2]{W97}, which we do as in \cite[Section 8.1(i) and (ii)]{FoSp18}. 
For all $g \in \mathscr{G}_{x_\circ'}$, from the proof of Lemma~\ref{lem: blowups VTOP},
\[
g(0) = f(x_\circ').
\]
Moreover, suppose $r_j \downarrow 0$. 
By Lemma~\ref{lem: blowups VTOP}, we can find a subsequence $r_{j_\ell} \downarrow 0$ and element $g_\infty \in \mathscr{G}_{x_\circ'}$ so that for any convergent sequence $x'_\ell \to x'_\infty \in B_1'$ as $\ell \to \infty$,
\[
\limsup_{\ell \to \infty} f(x_\circ' + r_{j_\ell}x'_\ell) \leq g_\infty(x_\infty').
\]
In particular,
\[
g_\infty(z') := 
\begin{cases}
N(0^+,\tilde u_{X_\circ,0}(Z + \,\cdot\,)) &\text{if } X_\circ \in \mathcal{N}(u)\\
0 &\text{if } X_\circ \notin \mathcal{N}(u)
\end{cases}
\]
with $\tilde u_{X_\circ,0}$ being the weak $W^{1,2}(B_1,|y|^a)$ limit of $\tilde{u}_{X_\circ,r_{j_\ell}}$ (it is also the limit in $C^{-a}_{\rm loc}(B_1)$ of $\tilde{u}_{X_\circ,r_{j_\ell}}$).
Indeed,
\[
\begin{split}
\limsup_{\ell \to \infty} N(0^+, u(X_\circ + r_{j_\ell}X_\ell + \,\cdot\,)) &\leq \inf_{\rho > 0} \limsup_{\ell \to \infty}N(r_{j_\ell}\rho, u(X_\circ + r_{j_\ell}X_\ell + \,\cdot\,)) \\
&= \inf_{\rho > 0} \limsup_{\ell \to \infty} N(\rho, \tilde{u}_{X_\circ,r_{j_\ell}}(X_\ell + \,\cdot\,)) \\
&= \inf_{\rho > 0} N(\rho, \tilde u_{X_\circ,0}(X_\infty + \,\cdot\,))\\
&= N(0^+, \tilde u_{X_\circ,0}(X_\infty + \,\cdot\,)).
\end{split}
\]
(Again, $X_\ell = (x_\ell',0,0)$ and $X_\infty = (x_\infty',0,0)$.)
Hence, applying \cite[Theorem 3.2]{W97} (or see \cite[Section 8.1]{FoSp18}), we prove \eqref{eq.claimHaus}.  
\end{proof}

We close this section recalling the classification of two-dimensional homogeneous solutions to \eqref{eq.VTOP}, which was proved in \cite[Proposition A.1(i)]{FoSp18}, and an important consequence.

\begin{lemma}
\label{lem: 2d class}
Let $n = 1$.
Let $u$ be a $\lambda$-homogeneous solution to \eqref{eq.VTOP}, subject to its own boundary data.
Then,
\[
\lambda \in \{ -a, 1, 2, 3, \dots \}.
\]
In addition, when $\lambda \in \N$, $u$ is an $a$-harmonic polynomial in $\R^2$.
\end{lemma}

\begin{proof}
The possible values of $\lambda$ are classified in \cite[Proposition A.1(i)]{FoSp18}, whence $ \lambda \in \N$.
Moreover, these integrally homogeneous solutions are polynomials; in particular, they are $a$-harmonic.
That said, in \cite{FoSp18}, only homogeneities greater or equal than $1+s$ are considered.
Within the proof of \cite[Proposition A.1(i)]{FoSp18}, however, if homogeneities in $(0, 1)$ are also considered, then only one extra homogeneity appears: $-a$, by taking $\nu = -1+s$ (using the notation of \cite{FoSp18}).
\end{proof}

\begin{corollary}
\label{cor: int freq}
Let $n \ge 2$ and $u$ be the solution to \eqref{eq.VTOP}.
Then,
\[
\dim_{\mathcal{H}}( \{ X_\circ \in \Lambda(u) : N(0^+,u(X_\circ + \cdot)) \notin \N \cup \{-a\} \}) \leq n-2.
\]
\end{corollary}

\begin{proof}
If $Z_\circ \in \Lambda(u) \setminus \{ X_\circ \in \mathcal{N}(u) : \dim L(u_{X_\circ,0}) \leq n - 2 \text{ for all blow-ups } \tilde u_{X_\circ,0}\}$, then there exists a blow-up $\tilde u_{Z_\circ,0}$ such that $\dim L(\tilde u_{Z_\circ,0}) = n-1$. 
In turn, since two-dimensional homogeneous solutions to the very thin obstacle problem with zero obstacle are polynomials or a multiple of $|X|^{-a}$ (by Lemma~\ref{lem: 2d class}), we deduce that $N(0^+,u(Z_\circ + \,\cdot\,)) \in \N \cup \{-a\}$.
Hence, from Proposition~\ref{prop.claim}, we conclude.
\end{proof}


\section{Final Remark: Global Problems}
\label{sec.GlobPbs}
In this final section, we state three global obstacle problems\,---\,all equivalent\,---\,to provide some additional perspective on the very thin obstacle problem. 
Let 
\begin{equation}
\label{eq.obstacle}
\psi\in C^{1,1}(\R^{n-1})
\end{equation}
be our obstacle, which we assume decays rapidly at infinity. 

\bigskip
\noindent{\it The very thin obstacle problem for $L_a$ in $\R^{n+1}$ with $a \in (-1,0)$.} 
Our first problem is a global version of the very thin obstacle problem for $L_a$ with obstacle $\psi$ on $\{x_n = y = 0\}$. 
Namely, we can consider either the global minimizer of the energy \eqref{eq.minpb} among those functions that sit above the obstacle $\psi$ on $\{x_n = y = 0\}$ and go to zero at infinity or, 
equivalently, the solution to Euler--Lagrange equations
\begin{align}
\label{eq.Glob1} 
& \qquad
\left\{ 
\begin{array}{rcll}
w_{1}(x', 0,0) &\geq& \psi(x') &\text{in } \R^{n-1}\\
L_a w_1 &=& 0 &\text{in } \R^{n+1}\setminus\{(x',0,0) : w_1(x',0,0) = \psi(x')\}\\
L_a w_1 &\leq& 0 &\text{in } \R^{n+1}\\
\lim_{|X|\to \infty} w_1(X)&=& 0. 
\end{array}\right.
\end{align}
Since $a \in (-1,0)$, it makes sense to say that the solution sits above the $\psi$ on the set $\{x_n = y = 0\}$. 

\bigskip
\noindent{\it The thin obstacle problem for $(-\Delta)^s$ in $\R^n$ with $s \in (1/2,1)$} 
Our second problem is the fractional thin obstacle problem. 
That is, we consider
\begin{align}
\label{eq.Glob2}
& \qquad\left\{ 
\begin{array}{rcll}
w_2(x', 0) &\geq& \psi(x') &\text{in } \R^{n-1}\\
(-\Delta)^s w_2 &=& 0 &\text{in } \R^{n}\setminus\{(x',0) : w_2(x',0) = \psi(x')\}\\
(-\Delta)^s w_2 &\leq& 0 &\text{in } \R^{n}\\
\lim_{|x|\to \infty} w_2(x)&=& 0. 
\end{array}\right.
\end{align}

\bigskip
\noindent{\it The obstacle problem for $(-\Delta)^{s-\frac{1}{2}}$ in $\R^{n-1}$ with $s \in (1/2,1)$.} 
Our third and final problem is the obstacle problem for the fractional Laplacian $(-\Delta)^{s-\frac{1}{2}}$ in $\R^{n-1}$. 
This problem is classical already, and its Euler--Lagrange equations are
\begin{align}
\label{eq.Glob3}
&\qquad  \left\{ 
\begin{array}{rcll}
w_3(x') &\geq& \psi(x') &\text{in } \R^{n-1}\\
(-\Delta)^{s-\frac{1}{2}} w_3 &=& 0 &\text{in } \R^{n-1}\setminus\{x' : w_3(x') = \psi(x')\}\\
(-\Delta)^{s-\frac{1}{2}} w_3 &\leq& 0 &\text{in } \R^{n-1}\\
\lim_{|x'|\to \infty} w_3(x')&=& 0. 
\end{array}\right.
\end{align}

\bigskip
\begin{proposition}
If $w_1(x',x_n,y)$ is the solution to \eqref{eq.Glob1}, then $w_2(x', x_n) = w_1(x', x_n, 0)$ is the solution to \eqref{eq.Glob2}, and $w_3(x') = w_2(x', 0) = w_1(x', 0, 0)$ is the solution to \eqref{eq.Glob3}.
\end{proposition}

\begin{proof}
The fact that $w_2(x',x_n)$ is a solution to \eqref{eq.Glob2} comes from the extension problem for the fractional Laplacian (see \cite{CS07}).
The fact that $w_3(x')$ solves \eqref{eq.Glob3} is due to Lemma~\ref{lem.operator} and Proposition~\ref{prop.mathcalF}. 
\end{proof}




\begin{thebibliography}{99}
\bibitem[AM11]{AM11}
	\newblock J. Andersson, H. Mikayelyan,
	\newblock {\em $C^{1,\alpha}$ regularity for solutions to the p-harmonic thin obstacle problem},
	\newblock Int. Math. Res. Not. IMRN 2011, no. 1, 119-134.

\bibitem[AC04]{AC04}
	\newblock I. Athanasopoulos, L. Caffarelli, 
	\newblock {\em Optimal regularity of lower dimensional obstacle problems},
	\newblock Zap. Nauchn. Sem. S.-Peterburg. Otdel. Mat. Inst. Steklov. (POMI) {\bf 310} (2004).
	
\bibitem[ACS08]{ACS08} 
	\newblock I. Athanasopoulos, L. Caffarelli, S. Salsa, 
	\newblock {\em The structure of the free boundary for lower dimensional obstacle problems},
	\newblock  Amer. J. Math. {\bf 130} (2) (2008), 485-498.
	
\bibitem[BFR18]{BFR18} 
	\newblock B. Barrios, A. Figalli, X. Ros-Oton, 
	\newblock {\em Global regularity for the free boundary in the obstacle problem for the fractional Laplacian},
	\newblock  Amer. J. Math. {\bf 140} (2) (2018), 415-447.
	
\bibitem[BG90]{BG90}
	\newblock J.P. Bouchaud, A. Georges,
	\newblock {\em Anomalous diffusion in disordered media: Statistical mechanisms, models and physical applications},
	\newblock Phys. Rep. {\bf 195} (1990), 127-293.
	
\bibitem[BLOP19]{BLOP19}
	\newblock S. Byun, K. Lee, J. Oh, J. Park,
	\newblock {\em Regularity results of the thin obstacle problem for the $p(x)$-Laplacian},
	\newblock J. Funct. Anal. {\bf 276} (2019), 496-519.
	
\bibitem[CDV20]{CDV19}
	\newblock X. Cabr\'e, S. Dipierro, E. Valdinoci,
	\newblock {\em The Bernstein technique for integro-differential equations},
	\newblock arXiv 2010.00376.
	
\bibitem[Caf98]{Caf98}
	\newblock L. Caffarelli,
	\newblock {\em The obstacle problem revisited},
	\newblock J. Fourier Anal. Appl. {\bf 4} (1998), 383-402.

\bibitem[CRS17]{CRS17}
	\newblock L. Caffarelli, X. Ros-Oton, J. Serra,
	\newblock {\em Obstacle problems for integro-differential operators: regularity of solutions and free boundaries},
	\newblock Invent. Math. {\bf 208} (2017), 1155-1211.

\bibitem[CSS08]{CSS08}
	\newblock L. Caffarelli, S. Salsa, L. Silvestre, 
	\newblock {\em Regularity estimates for the solution and the free boundary of the obstacle problem for the fractional Laplacian},
	\newblock Invent. Math. {\bf 171} (2008), no. 2, 425-461.
	
\bibitem[CS07]{CS07}
	\newblock L. Caffarelli, L. Silvestre, 
	\newblock {\em An extension problem related to the fractional Laplacian},
	\newblock Comm. Partial Differential Equations {\bf 32} (2007), 1245-1260.
	
\bibitem[CV10]{CV10}
	\newblock L. Caffarelli, A. Vasseur,
	\newblock {\em Drift diffusion equations with fractional diffusion and the quasi-geostrophic equation},
	\newblock Ann. Math. {\bf 171} (2010), 1903-1930.
	
\bibitem[CDM16]{CDM16}
	\newblock J. A. Carrillo, M. G. Delgadino, A. Mellet,
	\newblock {\em Regularity of local minimizers of the interaction energy via obstacle problems},
	\newblock Comm. Math. Phys. {\bf 343} (2016), 747-781.
	
\bibitem[CSV20]{CSV19}
	\newblock M. Colombo, L. Spolaor, B. Velichkov, 
	\newblock {\em Direct epiperimetric inequalities for the thin obstacle problem and applications},
	\newblock Comm. Pure Appl. Math. {\bf 73} (2020), 384-420.

\bibitem[CT04]{CT04}
	\newblock R. Cont, P. Tankov,
	\newblock {\em Financial modeling with jump processes},
	\newblock Chapman \& Hall/CRC Financial Mathematics Series. Chapman \& Hall/CRC, Boca Raton, FL, 2004.
	
\bibitem[DL76]{DL76}
	\newblock G. Duvaut, J. L. Lions,
	\newblock {\em Inequalities in Mechanics and Physics},
	\newblock Springer, Berlin, 1976.

\bibitem[Fef09]{Fef09}
	\newblock C. Fefferman, 
	\newblock {\em Extension of $C^{m,\omega}$-smooth functions by linear operators},
	\newblock Rev. Mat. Iberoam. 25 (2009), no. 1, 1–48.
	
\bibitem[Fer16]{Fer16}
	\newblock X. Fern\'andez-Real,
	\newblock {\em $C^{1,\alpha}$ estimates for the fully nonlinear Signorini problem}, 
	\newblock Calc. Var. Partial Differential Equations {\bf 55} (2016), Art. 94, 20 pp.
	
\bibitem[FeSe20]{FeSe18}
	\newblock X. Fern\'andez-Real, J. Serra,
	\newblock {\em Regularity of minimal surfaces with lower dimensional obstacles},
	\newblock J. Reine Angew. Math., to appear.

\bibitem[FS18]{FS18}
	\newblock A. Figalli, J. Serra,
	\newblock {\em On the fine structure of the free boundary for the classical obstacle problem},
	\newblock Invent. Math. {\bf 215} (2019), no. 1, 311–366.
	
\bibitem[FRS20]{FRS19}
	\newblock A. Figalli, X. Ros-Oton, J. Serra,
	\newblock {\em Generic regularity of free boundaries for the obstacle problem},
	\newblock Publ. Math. IHÉS, to appear..
	
\bibitem[FMS15]{FMS15} M. Focardi, A. Marchese, E. Spadaro, 		 	
	\newblock {\em Improved estimate of the singular set of Dir-minimizing Q-valued functions via an abstract regularity result},
	\newblock J. Funct. Anal. {\bf 268} (2015), no. 11, 3290-3325.

\bibitem[FoSp18]{FoSp18}
	\newblock M. Focardi,  E. Spadaro,
	\newblock {\em On the measure and the structure of the free boundary of the lower dimensional obstacle problem},
	\newblock Arch. Rat. Mech. Anal. {\bf 230} (2018), 125-184.
	
\bibitem[FoSp20]{FoSp18b}
	\newblock M. Focardi, E. Spadaro,
	\newblock {\em How a minimal surface leaves a thin obstacle},
	\newblock Ann. Inst. H. Poincare Anal. Non Lineaire {\bf 37} (2020), 1017-1046.

\bibitem[GP09]{GP09}
	\newblock N. Garofalo, A. Petrosyan, 
	\newblock {\em Some new monotonicity formulas and the singular set in the lower dimensional obstacle problem},
	\newblock Invent. Math. {\bf 177} (2009), no. 2, 414-461.
	
\bibitem[GR19]{GR18}
	\newblock N. Garofalo, X. Ros-Oton,
	\newblock {\em Structure and regularity of the singular set in the obstacle problem for the fractional Laplacian},
	\newblock Rev. Mat. Iberoam. {\bf 35} (2019), 1309-1365.		

\bibitem[HKM93]{HKM93}
	\newblock J. Heinonen, T. Kilpel\"{a}inen, O. Martio,
	\newblock {\em	Nonlinear Potential Theory of Degenerate Elliptic Equations},
	\newblock Oxford Mathematical Monographs, Oxford University Press Inc., New York, 1993.
	
\bibitem[JN17]{JN17}
	\newblock Y. Jhaveri, R. Neumayer,
	\newblock {\em	Higher regularity of the free boundary in the obstacle problem for the fractional Laplacian},
	\newblock Adv. Math. {\bf 311} (2017), 748-795.	
	
\bibitem[KO88]{KO88}
	\newblock N. Kikuchi, J. T. Oden,
	\newblock {\em Contact Problems in Elasticity: A Study of Variational Inequalities and Finite Element Methods.}
	\newblock SIAM Studies in Applied Mathematics, vol. 8. Society for
Industrial and Applied Mathematics, Philadelphia, 1988.

\bibitem[Kil94]{Kil94}
	\newblock T. Kilpel\"{a}inen,
	\newblock {\em Weighted Sobolev spaces and capacity},
	\newblock  Ann. Acad. Sci. Fenn. Ser. A I Math. {\bf 19} (1994), no. 1, 95-113.	
	
\bibitem[Kim07]{Kim07}
	\newblock D. Kim,
	\newblock {\em Trace theorems for Sobolev-Slobodeckij spaces with or without weights},
	\newblock J. Funct. Spaces Appl. {\bf 5} (2007), no. 3, 243-268.

\bibitem[KRS19]{KRS16}
	\newblock H. Koch, A. R\"uland, W. Shi,
	\newblock {\em  Higher regularity for the fractional thin obstacle problem},
	\newblock New York J. Math. {\bf 25} (2019), 745-838.	
	
\bibitem[KW13]{KW13}
	\newblock B. Krummel, N. Wickramasekera,
	\newblock {\em  Fine properties of branch point singularities: Two-valued harmonic functions},
	\newblock arXiv 1311.0923.	
	
\bibitem[Mer76]{Mer76}
	\newblock R. Merton,
	\newblock {\em Option pricing when the underlying stock returns are discontinuous},
	\newblock J. Finan. Econ. {\bf 5} (1976), 125-144.

\bibitem[MS06]{MS06} E. Milakis, L. Silvestre, 
	\newblock {\em Regularity for fully nonlinear elliptic equations with Neumann boundary data},
	\newblock Comm. Partial Differential Equations {\bf 31} (2006), 1227-1252.	
	
\bibitem[MS08]{MS08} E. Milakis, L. Silvestre, 
	\newblock {\em Regularity for the nonlinear Signorini problem},
	\newblock Adv. Math. {\bf 217} (2008), 1301-1312.

\bibitem[MO69]{MO69} S. A. Molchanov and E. Ostrovskii, 
	\newblock {\em Symmetric stable processes as traces of degenerate diffusion processes}, 
	\newblock Theor. Probability Appl. {\bf 14} (1969), 128-131.
	
\bibitem[NLM88]{NLM88}
	\newblock S. M. Nikol'skii, P. I. Lizorkin, N. V. Miroshin,
	\newblock {\em Weighted function spaces and their applications to the investigation of boundary value problems for degenerate elliptic equations},
	\newblock Izv. Vyssh. Uchebn. Zaved. Mat. {\bf 8} (1988), 4-30.

\bibitem[RS16]{RS16}
	\newblock X. Ros-Oton, J. Serra,
	\newblock {\em Regularity theory for general stable operators},
	\newblock J. Differential Equations {\bf 260} (2016), 8675-8715.
	
\bibitem[RS17]{RS17}
	\newblock X. Ros-Oton, J. Serra,
	\newblock {\em The structure of the free boundary in the fully nonlinear thin obstacle problem},
	\newblock Adv. Math. {\bf 316} (2017), 710-747.
	
\bibitem[Ros18]{Ros18}
	\newblock X. Ros-Oton,
	\newblock {\em 	Obstacle problems and free boundaries: an overview},
	\newblock SeMA J. {\bf 75} (2018), 399-419.
	
\bibitem[RuSh17]{RuSh17}
	\newblock A. R\"uland, W. Shi,
	\newblock {\em Optimal regularity for the thin obstacle problem with $C^{0,\alpha}$  coefficients},
	\newblock Calc. Var. Partial Differential Equations {\bf 56} (2017), Art. 129, 41 pp.

\bibitem[Sig33]{Sig33}
	\newblock A. Signorini,
	\newblock {\em Sopra alcune questioni di elastostatica},
	\newblock Atti Soc. It. Progr. Sc. {\bf 21} (1933), no. 2, 143-148.

\bibitem[Sig59]{Sig59}
	\newblock A. Signorini,
	\newblock {\em Questioni di elasticit\`a non linearizzata e semilinearizzata},
	\newblock Rend. Mat. e Appl. {\bf 18} (1959), no. 5, 95-139.

\bibitem[Sim83]{Sim83}
	\newblock L. Simon, 
	\newblock {\em Lectures on geometric measure theory},
	\newblock Proceedings of the Centre for Mathematical Analysis Australian National University, 3. Australian National University, Centre for Mathematical Analysis, Canberra, 1983.

\bibitem[W97]{W97}
	\newblock B. White,
	\newblock {\em Stratification of minimal surfaces, mean curvature flows, and harmonic maps},
	\newblock J. Reine Angew. Math. {\bf 488} (1997), 1-35.	
	
\end{thebibliography}
\end{document}